\documentclass[11pt,twoside,a4paper]{amsart}
\usepackage[english]{babel}

\usepackage[utf8]{inputenc}

\usepackage{pstricks-add}
\usepackage{array}
\usepackage{amsfonts}
\usepackage{amssymb}
\usepackage{amsmath}
\usepackage{graphicx}
\usepackage{graphics}
\usepackage{amsthm}

\usepackage{amssymb}

\usepackage{mathabx}

\usepackage[cmtip,all]{xy}

\usepackage{fancyhdr}
 
\fancypagestyle{noheadrule}{
\fancyhf{}

\fancyhead[LE,RO]{\thepage}
}

\newtheorem{theo}{Theorem}[section]
\newtheorem{de}[theo]{Definition}
\newtheorem{prop}[theo]{Proposition}
\newtheorem{lem}[theo]{Lemma}

\theoremstyle{definition}
\newtheorem{rem}[theo]{Remark}

\normalsize
\title[Pure point/Continuous decomposition of measures and diffraction]{Pure point/Continuous decomposition of translation-bounded measures and diffraction}

\author{Jean-Baptiste Aujogue}

\thanks{
\\
Work supported by the FONDECYT Post-doctoral grant n. 3150535.}

\begin{document}

\maketitle


\begin{abstract} In this work we consider translation-bounded measures over a locally compact Abelian group $\mathbb{G}$, with particular interest for their so-called diffraction. Given such a measure $\Lambda$, its diffraction $\widehat{\gamma}$ is another measure on the Pontryagin dual $\widehat{\mathbb{G}}$, whose decomposition into the sum $\widehat{\gamma} = \widehat{\gamma}_{\mathrm{p}}+ \widehat{\gamma}_{\mathrm{c}}$ of its atomic and continuous parts is central in diffraction theory. The problem we address here is whether the above decomposition of $\widehat{\gamma}$ lifts to $\Lambda$ itself, that is to say, whether there exists a decomposition $\Lambda = \Lambda _{\mathrm{p}} + \Lambda _{\mathrm{c}}$, where $\Lambda _{\mathrm{p}}$ and $\Lambda _{\mathrm{c}}$ are translation-bounded measures having diffraction $\widehat{\gamma}_{\mathrm{p}}$ and $\widehat{\gamma}_{\mathrm{c}}$ respectively. Our main result here is the almost sure existence, in a sense to be made precise, of such a decomposition. It will also be proved that a certain uniqueness property holds for the above decomposition. Next we will be interested in the situation where translation-bounded measures are weighted Meyer sets. In this context, it will be shown that the decomposition, whether it exists, also consists of weighted Meyer sets. We complete this work by discussing a natural generalization of the considered problem.
\end{abstract}



\section{\textsf{Outline}}

In this work we consider certain measures over a locally compact Abelian (LCA) group $\mathbb{G}$, for which we aim to establish a general result on structures that goes beyond the ones called \textit{purely diffractive}. The motivation of this work comes from solid state physics, where materials with atomic-like kinematic diffraction have focused a lot of attention during the last 30 years since the discovery by Shechtman \textit{et al.} \cite{SheBleGraCah} of physical materials nowadays called \textit{Quasicrystal}, a work for which Schechtman was awarded the Nobel price in 2011.








Mathematically, one describes a material by a \textit{translation-bounded measure} $\Lambda$ over the group $\mathbb{G}$ (thought as the ambient space, and usually taken as $\mathbb{R}^d$, $d=1$, 2 or 3). Very often a simpler picture is displayed by a weighted Dirac comb over a point set having specific geometric properties (lattice, Delone, FLC, Meyer property and so on), supposed to model the position and nature of the constituents of the alloy. However this is not required for a proper definition of diffraction so we shall not assume this, at least not at the beginning. Although when dealing with some translation-bounded measure one thinks of an alloy as made of finitely many atoms, the structures one considers are infinite-sized idealizations. Now the mathematical diffraction theory as proposed by Hof in \cite{Hof}, see also the recent review article \cite{BaaGri}, assigns to a general translation-bounded measure $\Lambda$ on a LCA group $\mathbb{G}$ another measure $\gamma$ on $\mathbb{G}$, given as a self-convolution product of $\Lambda$ passed to the infinite volume limit along a certain sequence $(\mathcal{A}_k)_k$ of sets (a Van Hove sequence)
\begin{align*}  \gamma  := \lim _{n \rightarrow \infty} \frac{1}{\left|\mathcal{A}_k \right|} \mathsf{\Lambda}\vert _{\mathcal{A}_k} * \widetilde{\mathsf{\Lambda}\vert _{\mathcal{A}_k}}
\end{align*}

\vspace{0.1cm}
Such limit, whenever it exists, and it always will after possibly extracting some sub-sequence, is called the \textit{autocorrelation} of $\Lambda$ (with respect to the averaging sequence $(\mathcal{A}_k)_k$), and is translation-bounded on $\mathbb{G}$. Its Fourier transform $\widehat{\gamma}$ is called the \textit{diffraction measure} (or simply diffraction) of $ \Lambda$, and is a positive translation-bounded measure on the Pontryagin dual $\widehat{\mathbb{G}}$. It is this latter which is thought as the outcome of a diffraction experiment set on an alloy modeled by $\Lambda$, see the discussion in \cite{Hof}. As any complex measure over $\widehat{\mathbb{G}}$, the diffraction measure splits as the sum 
\vspace{0.1cm}
\begin{align*} \widehat{\gamma} = \widehat{\gamma}_{\mathrm{p}}+ \widehat{\gamma}_{\mathrm{c}} 
\end{align*}

\vspace{0.1cm}

of two mutually singular summands, namely a purely atomic (or \textit{pure point}) part $\widehat{\gamma}_{\mathrm{p}}$ and a continuous part $\widehat{\gamma}_{\mathrm{c}}$. The summand $\widehat{\gamma}_{\mathrm{p}}$ is called the \textit{Bragg spectrum} of the underlying translation-bounded measure $\Lambda$, and any $\omega \in \widehat{\mathbb{G}}$ in its support, that is, with $\widehat{\gamma}(\left\lbrace \omega \right\rbrace )> 0$, is called a \textit{Bragg peak} for $\Lambda$. The measure $\Lambda$ is then called \textit{purely diffractive}, or said to have pure point spectrum, whenever $\widehat{\gamma}$ is a pure point measure, that is, if its continuous part $\widehat{\gamma}_{\mathrm{c}}$ vanishes. Purely diffractive measures are, at least in the setting of weighted point sets, rather well understood, see the general characterizations made in \cite{LeeMooSol, Gou2, BaaMoo, BaaLen}. It encompasses the (uniform Dirac combs over) lattices and regular model sets \cite{Baa2, Sch, Moo2}, and is agremented by many explicit examples \cite{BaaMoo0, BaaMooSch, Baa2}, see again the review article \cite{BaaGri}. The situation of mixed diffraction spectrum is much less understood, although some explicit computations exist in this case \cite{BaaBirGri, BaaGri2}.\\

In this work we will be specially interested in the situation of mixed diffraction spectrum. Namely, if $\Lambda$ is a translation-bounded measure with non-pure point diffraction say $\widehat{\gamma}$, we wish to know whether $\Lambda$ still admits a natural "purely diffractive part" $\Lambda _{\mathrm{p}}$ whose diffraction would exactly be the pure point part $\widehat{\gamma}_{\mathrm{p}}$, and such that $\Lambda - \Lambda _{\mathrm{p}}$ would have diffraction the non-vanishing continuous part $\widehat{\gamma}_{\mathrm{c}}$. That is to say, we raise the following question:\\


\textbf{Q.} \textit{Given a translation-bounded measure $\Lambda$ on $\mathbb{G}$ having diffraction $\widehat{\gamma} = \widehat{\gamma}_{\mathrm{p}}+ \widehat{\gamma}_{\mathrm{c}} $, does there exist a decomposition of the form $\Lambda = \Lambda _{\mathrm{p}} + \Lambda _{\mathrm{c}}$, with $\Lambda _{\mathrm{p}}$ and $\Lambda _{\mathrm{c}}$ being translation-bounded measures having diffraction $\widehat{\gamma}_{\mathrm{p}}$ and $\widehat{\gamma}_{\mathrm{c}}$ respectively?}\\




 As we will see, it is possible to give a rigorous answer to this question when one does not consider a translation-bounded measure $\Lambda$ as a deterministic object, but rather a \textit{random} one. Physically this means that the precise structure of the underlying alloy is unknown, although one knows the value of frequency of appearance of any possible local configuration. On the mathematical side this means that we do not consider a single translation-bounded measure but the whole collection $\mathcal{M}^{\infty}(\mathbb{G})$ of such measures, endowed with a translation-invariant (ergodic) probability distribution $m$ on it. In other words we shall deal here with an \textit{ergodic system of translation-bounded measures} $(X, \mathbb{G}, m)$, $X\subset \mathcal{M}^{\infty}(\mathbb{G})$ being the support of $m$. In the setting of (uniform Dirac combs over) point sets one speaks of (ergodic) stationary point process \cite{Gou2, BaaBirMoo, DenMoo}, see also the more abstract approach \cite{LenMoo}. This approach by dynamical systems presents the great advantage to connect the diffraction of a random translation-bounded measure modeled by a dynamical system $(X, \mathbb{G}, m)$, which is well-defined and unique, with the \textit{dynamical spectrum} of $(X, \mathbb{G}, m)$ via the so-called Dworkin argument, see \cite{Dwo, BaaLen}, also Paragraph \ref{Par:diffraction} in the main text. 


\subsection*{\textsf{Overview on the content}} 
Our main result (see Section \ref{Section:decomposition}) is that, for an ergodic random translation-bounded measure the decomposition almost surely exists, and is provided by two other ergodic random translation-bounded measures which are naturally obtained from the former. In the vocabulary of dynamical systems, this states as follows (Theorem \ref{theo:decomposition} and Theorem \ref{theo:uniqueness.decomposition} in the main text):\\


\textbf{Theorem 1.} \textit{Let $(X, \mathbb{G},m)$ be an ergodic system of translation-bounded measures with diffraction $\widehat{\gamma}= \widehat{\gamma}_{\mathrm{p}} + \widehat{\gamma}_{\mathrm{c}}$.}\\
\\
\textrm{(i)}\textit{ There exist two ergodic systems of translation-bounded measures:}
\vspace{0.2cm} 
\begin{itemize}
\item \textit{$(X_{\mathrm{p}}, \mathbb{G},m_{\mathrm{p}})$ having diffraction $\widehat{\gamma}_{\mathrm{p}}$,}
\vspace{0.2cm} 
\item \textit{$(X_{\mathrm{c}}, \mathbb{G},m_{\mathrm{c}})$ having diffraction $\widehat{\gamma}_{\mathrm{c}}$,}
\end{itemize}
\vspace{0.2cm} 
\textit{which are Borel factors of $(X, \mathbb{G},m)$ under two Borel factor maps} $$\xymatrixcolsep{3.5pc}\xymatrix{  X \ni \Lambda \, \ar@{|->}[r] & \,\Lambda _{\text{p}} \in X_{\text{p}} } \qquad \xymatrixcolsep{3.5pc}\xymatrix{  X \ni \Lambda \, \ar@{|->}[r] & \,\Lambda _{\text{c}} \in X_{\text{c}} }$$
\textit{such that the equality $\Lambda = \Lambda _{\mathrm{p}}+ \Lambda _{\mathrm{c}}$ occurs for $m$-almost every $\Lambda\in X$.}\\
\\
\textrm{(ii)}\textit{ The pair of systems of translation-bounded measures satisfying part }\textrm{(i)} \textit{is unique, as well as the involved pair of Borel factor maps up to almost everywhere equality.}\\

In this work we are particularly interested in the case where translation-bounded measures are \textit{weighted Meyer sets} of $\mathbb{G}$. These are weighted Dirac combs supported on sets of very special kind, the so-called Meyer sets. In this setting, the natural question is, given an ergodic system of weighted Meyer sets, to know whether translation-bounded measures in the spaces $X_{\mathrm{p}}$ and $X_{\mathrm{c}}$ resulting from Theorem 1 also are weighted Meyer sets. As we will see, this turns out to be correct (see Section \ref{Section:weighted.meyer.sets}). 
 
In fact, much more is true: A particular property of weighted Meyer sets is that they are always supported on some point set of special feature, namely a \textit{model set} \cite{Meyer, Sch, Moo2,  Auj, HucRic}. Such point sets emerge if one not only consider the "ambient" space $\mathbb{G}$ but a product $H \times \mathbb{G}$ of it with some locally compact Abelian group $H$, inside which one consider (whenever it exists) a lattice $\Sigma$. Then denoting $z^{\Vert }$ and $z^{\perp}$ the projections of an element $z\in H\times \mathbb{G}$ along $\mathbb{G}$ and $H$ respectively, if one select a compact topologically regular subset $W$ of $H$ then one gives rise to a point set, called here a "closed" model set, of the form
\begin{align*} \Delta := \left\lbrace z ^{\Vert} \in \mathbb{G} \; : \; z \in \Sigma , \; z ^{\perp} \in W\right\rbrace 
\end{align*}

\vspace{0.1cm}
Closed model sets are always Meyer sets, and for a partial converse any weighted Meyer set is supported on a point set of this form \cite{Meyer, Moo2}. Considering this we have the following (see Theorem \ref{theo:decomposition.weighted.meyer.sets}):\\

\textbf{Theorem 2.} \textit{Let $(X, \mathbb{G}, m)$ be an ergodic system of weighted Meyer sets. Then the following property holds $m$-almost everywhere: Whenever $\Lambda$ is supported on a closed model set, then so are its summands $\Lambda _{\mathrm{p}}$ and $\Lambda _{\mathrm{c}}$.}\\


This result in particular asserts that once $X$ consists of weighted Meyer sets then so does $X_{\mathrm{p}}$ and $X_{\mathrm{c}}$. It is in particular interesting (and surprising) to see the systematic emergence of weighted Meyer sets, namely the ones forming $X_{\text{c}}$ as long as this latter is non-trivial, having (for almost all of them) purely continuous diffraction. Only few is known about this latter type of weighted Meyer sets, although some study have been performed (on weighted lattices) for instance in \cite{BaaBirGri}.\\




The remaining part (Section \ref{Section:admissible.operators}) of this work is about a generalization of our motivating question. Namely, in the decomposition $\widehat{\gamma} = \widehat{\gamma}_{\mathrm{p}}+ \widehat{\gamma}_{\mathrm{c}} $ stated earlier the continuous part $\widehat{\gamma}_{\mathrm{c}}$ always splits into two mutually singular components 
\vspace{0.1cm}
\begin{align*}\widehat{\gamma}_{\mathrm{c}}= \widehat{\gamma}_{\mathrm{ac}}+ \widehat{\gamma}_{\mathrm{sc}}
\end{align*}

\vspace{0.1cm}
where $\widehat{\gamma}_{\mathrm{ac}}$ is its absolutely continuous part with respect to a fixed Haar measure on $\widehat{\mathbb{G}}$, and $\widehat{\gamma}_{\mathrm{sc}}$ the remaining part, its singular continuous part. It is then natural to ask whether a translation-bounded measure $\Lambda$ with diffraction $\widehat{\gamma} $ admits a decomposition of the form $\Lambda = \Lambda _{\mathrm{p}} + \Lambda _{\mathrm{ac}}+ \Lambda _{\mathrm{sc}}$, with $\Lambda _{\mathrm{p}}$, $\Lambda _{\mathrm{ac}}$ and $\Lambda _{\mathrm{sc}}$ being translation-bounded measures having diffraction $\widehat{\gamma}_{\mathrm{p}}$, $\widehat{\gamma}_{\mathrm{ac}}$ and $\widehat{\gamma}_{\mathrm{sc}}$ respectively. As before, we shall consider this question on a dynamical point of view and, given an ergodic system of translation bounded measures $(X, \mathbb{G}, m)$, look for dynamical systems of translation-bounded measures $(X_{\mathrm{ac}}, \mathbb{G},m_{\mathrm{ac}})$ and $(X_{\mathrm{sc}}, \mathbb{G},m_{\mathrm{sc}})$ such that a generalized form of Theorem 1 holds. 



We could not prove or disprove the existence of such systems. Instead, we propose here a reformulation of this problem: Our observation here is that the two Borel factors we look for have, whenever they exist, diffraction being absolutely continuous with respect to $\widehat{\gamma}$, hence of the form $f. \widehat{\gamma}$ for functions $f$ on $\widehat{\mathbb{G}}$ belonging in this case to $L^{\infty } (\widehat{\mathbb{G}}, \widehat{\gamma})$. Then what we would get is a criterion on the existence of a Borel factor of $(X, \mathbb{G}, m)$, whose diffraction is of the form $f. \widehat{\gamma}$ where the density function $f$ is prescribed. Our aim is to provide here such a criterion, which involves certain operators acting on $L^2(X,m)$ called here \textit{admissible} (see Definition \ref{def.admissibility}), and that states as follows (see Theorem \ref{theo:correspondence}):\\  

\textbf{Theorem 3.} \textit{Let $(X, \mathbb{G}, m)$ be a dynamical system of translation-bounded measures with diffraction $\widehat{\gamma}$. There is a bijective correspondence between:}\\


\textrm{(a)} \textit{Borel factor maps over some dynamical system of translation-bounded measures having diffraction $\widehat{\gamma}' \ll \widehat{\gamma}$, with Radon-Nikodym derivative $\frac{d \widehat{\gamma}'}{d \widehat{\gamma}} \in L^{\infty } (\widehat{\mathbb{G}}, \widehat{\gamma})$,}

\vspace{0.2cm}

\textrm{(b)} \textit{admissible operators $Q$ on $L^2(X,m)$.}\\

\textit{Given a Borel subset $\mathcal{P} \subseteq \widehat{\mathbb{G}}$, one has $\frac{d \widehat{\gamma}'}{d \widehat{\gamma}} = \mathbb{I}_{\mathcal{P}}$ if and only if $\vert Q \vert= \mathsf{E}(\mathcal{P})P_\Theta$, with $\mathsf{E}(\mathcal{P})$ the spectral projector associated to $\mathcal{P}$ and $P_\Theta$ an appropriate projector on $L^2(X,m)$.}\\

Therefore, existence of our desired Borel factors rests on the existence of admissible operators on $L^2(X,m)$ whose absolute value is given by an appropriate spectral projectors, something still tedious to prove due to the technicality of our admissibility condition. It is worth precising that this correspondence does not encompasses many interesting factors, as in some cases the knowledge of the diffraction of all factors allows to completely determine the dynamical spectrum, see the discussion in \cite{BaaLenVEn, BaaGahGri}.

\section{\textsf{Generalities on dynamical systems}}\label{Section:basics.dynamical.systems} 

We set here some notions, notations and results about abstract dynamical systems. All along this article $\mathbb{G}$ denotes a second countable locally compact Abelian (LCA) group, which is then also $\sigma$-compact and metrizable, with group operation noted additively, and with fixed Haar measure whose valuation on a Borel set $A$ is noted $\left| A \right|$. The integral with respect to the Haar measure will everywhere be noted $\int _{\mathbb{G}} ... \, dt$. Its Pontryagin dual $\widehat{\mathbb{G}}$ is also a second countable LCA group, whose group operation will be noted multiplicatively.

\subsection{\textsf{Dynamical systems and Borel factors}}\label{paragraph:dynamical.systems} ~ 

\vspace{0.2cm} 
By a \textit{dynamical system} we mean a pair $(X, \mathbb{G})$ consisting of a compact metrizable space $X$ together with a continuous $\mathbb{G}$-action (noted on the right)
\begin{align*}\xymatrixcolsep{3pc}\xymatrix{X \times \mathbb{G}  \ni (x, t) \,  \ar@{|->}[r] & \, x.t \in X }  
\end{align*}
Whenever $X$ is equipped with a $\mathbb{G}$-invariant (resp. ergodic) probability measure $m$ then the resulting triple $(X, \mathbb{G},m)$ will be called a measured (resp. ergodic) dynamical system. The topological support $Supp(m)$ of $m$ is the smallest compact subset of $X$ whose complementary set has measure zero, and $(X, \mathbb{G},m)$ is said to have full support whenever $Supp(m)=X$.

If $(X, \mathbb{G})$ is a dynamical system and $x \in X$ is chosen then we let $X_x$, the \textit{hull} of $x$, to be the closure of the $\mathbb{G}$-orbit of $x$ in $X$, giving rise to another dynamical system $(X_x, \mathbb{G})$ with restricted $\mathbb{G}$-action. A dynamical system $(X, \mathbb{G})$ is called minimal if $X_x =X$ for any $x\in X$. The following proposition is proved in \cite{FerKat}, Chapter $9$, Proposition $2.2.2$ therein:

\vspace{0.2cm} 

\begin{prop}\label{prop:ergodicity} In any ergodic system $(X, \mathbb{G},m)$ the set of point $ x \in X$ such that $X_x = Supp (m)$ is Borel of full measure in $X$.
\end{prop}

\vspace{0.2cm} 



Let us now consider two measured dynamical systems $(X, \mathbb{G},m)$ and $(X', \mathbb{G},m')$. A Borel map $\pi : X \longrightarrow  X'$ is called a Borel $\mathbb{G}$-map if for any fixed $t\in \mathbb{G}$ the set of points $x\in X$ such that $\pi (x.t)= \pi (x).t$ has full measure in $X$. It is called a Borel factor map if in addition $m'$ is the push-forward of $m$ by $\pi$, and in this case the system $(X', \mathbb{G},m')$ is called a Borel factor of $(X, \mathbb{G},m)$. Whenever $\pi $ is a Borel $\mathbb{G}$-map (resp. a factor map) then any Borel function coinciding with $\pi$ on a Borel set of full measure is again a $\mathbb{G}$-map (resp. a factor map). By \cite{FerKat}, Chapter $9$, Proposition $2.2.1$ therein, any $\mathbb{G}$-map $\pi : X \longrightarrow  X'$ is almost everywhere equal to a Borel $\mathbb{G}$-map $ \tilde{\pi } : X \longrightarrow  X'$ admitting a $\mathbb{G}$-stable Borel subset $\tilde{X}$ of full measure in $X$ such that $\tilde{\pi }(x.t)= \tilde{\pi }(x).t $ hold for any $ x\in \tilde{X }$ and any $t \in \mathbb{G}$. In particular one can deduce that a Borel factor of an ergodic dynamical system is also ergodic.\\

Let $(X, \mathbb{G}, m)$ be an ergodic dynamical system, and $(X', \mathbb{G}, m')$ some Borel factor with Borel factor map $\pi : X \longrightarrow  X'$. Denote $\mathcal{L}^{\infty}(X)$ the space of bounded complex-valued Borel functions on $X$. Then by the Disintegration Theorem the measure $m$ disintegrates over the factor $X'$ into a set of measures supported on fibers for the factor map $\pi$: Precisely, there exists a Borel map 
\begin{align*} \xymatrixcolsep{3pc}\xymatrix{ X' \ni x  \ar@{|->}[r] & \mu _{x} \in Prob (X) }
\end{align*}\vspace{0.2cm}
which satisfies the following two statements:

\vspace{0.2cm}
$(D1)$  The push-forward measure $\pi (\mu _x)= \delta _x$ for $m'$-almost every $x\in X'$,
 
$(D2)$  For each $f\in \mathcal{L}^{\infty}(X)$ one has $ \displaystyle \int _{X}f dm = \int _{x\in X'}\left[ \int _{X}f d\mu _x \right]  dm'(x)$.\\
\\
Moreover any two such maps coincide $m'$-almost everywhere on $X'$. The affirmation that a disintegration is a Borel map means that it is a Borel map when $Prob (X) $ is endowed with its vague topology, which exactly means that for each $f\in \mathcal{L}^{\infty}(X)$ the mapping associating to $x\in X'$ the value $ \int _{X}f d\mu _x $ is a Borel map. This also ensure that the right term of the equality in $(D2)$ is well defined. As a direct consequence of the almost everywhere uniqueness, a disintegration is always a $\mathbb{G}$-map in the following sense: As $\mathbb{G}$ acts on $X$ continuously it also acts on the space $\mathcal{M}^b(X)$ of bounded complex Radon measures of $X$ via the formula $\mu .t := \mu (.(-t))$, and leave stable the space $Prob (X)$ of probability measures on $X$. Then the almost everywhere uniqueness ensures that for any $t\in \mathbb{G}$ one has $\mu _{x.t}= (\mu _x ).t$ for $m'$-almost every $x\in X'$. Since any change of the disintegration map on a set of measure 0 in $X'$ does not violate $(D1)$ and $(D2)$ one can choose the disintegration to also satisfy\\

$(D3)$ One has $\mu _{x.t}= (\mu _x) .t$ for any $t\in \mathbb{G}$ and any $x$ in a full Borel subset $\tilde{X'}$.

\vspace{0.1cm}



%



\subsection{\textsf{Unitary representation and dynamical spectrum}} ~ 

\vspace{0.2cm} 

A measured dynamical system $(X, \mathbb{G}, m)$ gives naturally rise to a strongly continuous unitary representation of $\mathbb{G}$ on the separable Hilbert space $L^2(X,m)$, where on a function $h\in L^2(X,m)$ the unitary operator associated with $t\in \mathbb{G}$ reads $ U_t(h)= h(.(-t))$. From Stone's theorem (\cite{Loomis}, Section 36D, or Volume 2 of \cite{FellDoran}, Section 10.2 therein) there is an associated \textit{projection-valued measure} on the Borel sets of $\widehat{\mathbb{G}}$,
\begin{align*} \mathsf{E}: \text{Borel sets of }\widehat{\mathbb{G}} \;   \longrightarrow \text{Projectors of }L^2(X,m)
\end{align*}

that is to say, a map such that $\mathsf{E}(\mathcal{P})$ is an orthogonal projector on $L^2(X,m)$ for each Borel subset $\mathcal{P} \subseteq \widehat{\mathbb{G}}$, with $\mathsf{E}(\emptyset)$ the trivial operator and $\mathsf{E}(\widehat{\mathbb{G}})$ the identity operator, and which is $\sigma$-additive on the Borel sets of $\widehat{\mathbb{G}}$. This projection-valued measure is uniquely characterized according to the following property: For any two $h,k \in L^2(X,m)$ the scalar product $ \langle h,\mathsf{E}(.)k\rangle $ sets a bounded complex measure on $\widehat{\mathbb{G}}$, such that one has
\begin{align}\label{fourier} \langle h,U_t k\rangle = \int _{\omega \in \widehat{\mathbb{G}}}\omega (t)\, d\langle h,\mathsf{E}(\omega)k\rangle
\end{align}

In other words $ \langle h,\mathsf{E}(.)k\rangle $ is the unique bounded complex measure on $\widehat{\mathbb{G}}$ whose Fourier transform is the continuous bounded function on $\mathbb{G}$ mapping $t$ to $\langle h,U_t k\rangle$. A projector of the form $\mathsf{E}(\mathcal{P})$ is called a \textit{spectral projector}. One can deduces from equality (\ref{fourier}) that a bounded operator commuting with any operators $U_t$ also commutes with any spectral projectors. In particular the spectral projectors themselves commute with any $U_t$. As a result, the range of any spectral projector is a closed $\mathbb{G}$-stable subspace of $L^2(X,m)$.\\

For each $\omega \in \widehat{\mathbb{G}}$ we simply denote $\mathsf{E}(\omega )$ the spectral projector associated with the singleton $\lbrace \omega \rbrace $. As a consequence of (\ref{fourier}), the range of $\mathsf{E}(\omega )$ exactly consists of the \textit{eigenfunctions} of the system $(X, \mathbb{G}, m)$ with eigenvalue $\omega$, that is to say, functions $h \in L^2(X,m)$ such that $U_t(h)= \omega (t) h$ as $L^2$-functions for any $t \in \mathbb{G}$. A $\omega \in \widehat{\mathbb{G}}$ is called an \textit{eigenvalue} of the system $(X, \mathbb{G}, m)$ if it admits a non-trivial eigenfunction, that is, if the spectral projector $\mathsf{E}(\omega )$ is non-trivial. A measure dynamical system $(X, \mathbb{G}, m)$ is said to have \textit{pure point dynamical spectrum} if the eigenvalues span $ L^2(X,m)$, or equivalently, if the associated projection-valued measure $\mathsf{E}$ is purely atomic.


\subsection{\textsf{The maximal Kronecker Borel factor of an ergodic dynamical system}} ~ 

\vspace{0.2cm} 

When the measured dynamical system $(X, \mathbb{G}, m)$ is ergodic each eigenvalue $\omega$ has a unique eigenfunction up to a multiplicative constant, that is, the range of $\mathsf{E}(\omega )$ is 1-dimensional in $ L^2(X,m)$, and moreover any such eigenfunction is almost everywhere constant in modulus and thus belongs to $ L^\infty(X,m)$. Therefore the product of an eigenfunction associated with $\omega$ with another associated with $\omega '$ is well-defined, and is an eigenfunction associated with $\omega . \omega '$. Hence the set of eigenvalues of an ergodic dynamical system $(X, \mathbb{G}, m)$ always forms a (countable) subgroup $\mathcal{E}$ of $\widehat{\mathbb{G}}$.\\

Consider now any subgroup $\mathcal{E}$ of $\widehat{\mathbb{G}}$, endowed with the discrete topology for which it is a LCA group. As discrete Abelian group, it admits a Pontryagin dual T$_{ \mathcal{E} }$ which is a compact Abelian group, on which $\mathbb{G}$ acts \textit{by rotation}: The natural continuous injection morphism of $\mathcal{E}$ in $\widehat{\mathbb{G}}$ leads by Pontryagin dualization a continuous group morphism
\begin{align*} \xymatrixcolsep{3pc}\xymatrix{\mathsf{i}: \mathbb{G} \ar[r] &  \text{T}_{ \mathcal{E} }}
\end{align*}

with dense range, which in turns define a $\mathbb{G}$-action (set on the right) by $z.t:= z.\mathsf{i}(t)$. The resulting dynamical system $(\text{T}_{ \mathcal{E} }, \mathbb{G})$ is then a minimal \textit{Kronecker system}, and the normalized Haar measure on $\text{T}_{ \mathcal{E} }$ is the unique invariant measure with respect to this action, providing so the ergodic dynamical system $(\text{T}_{ \mathcal{E} }, \mathbb{G}, m_{Haar})$.

\vspace{0.2cm} 

\begin{prop}\label{prop.construction.Borel.factor.map} Let $(X, \mathbb{G}, m)$ be an ergodic dynamical system and let $\mathcal{E}\subset \widehat{\mathbb{G}}$ its eigenvalue group. Then $(\text{T}_{ \mathcal{E} }, \mathbb{G}, m_{Haar})$ is a Borel factor of $(X, \mathbb{G}, m)$.
\end{prop}

\vspace{0.2cm} 

For $(X, \mathbb{G}, m)$ an ergodic dynamical system with eigenvalue group $\mathcal{E}$, the Borel factor $(\text{T}_{ \mathcal{E} }, \mathbb{G}, m_{Haar})$ associated with $\mathcal{E}$ is called its \textit{maximal Kronecker Borel factor} \cite{}. The term "maximal" refers to the fact that any Kronecker Borel factor of $(X, \mathbb{G}, m)$ factors through $(\text{T}_{ \mathcal{E} }, \mathbb{G}, m_{Haar})$. The celebrated Halmos-Von Neumann Theorem asserts that an ergodic dynamical system has pure point dynamical spectrum if and only if it is Borel conjugated with its maximal Kronecker Borel factor.\\


Given an ergodic dynamical system $(X, \mathbb{G}, m)$ with eigenvalue group $\mathcal{E}$, any Borel factor map $\pi$ onto its maximal Kronecker Borel factor $(\text{T}_{ \mathcal{E} }, \mathbb{G}, m_{Haar})$ induces a $\mathbb{G}$-commuting isometric embedding of Hilbert spaces
\begin{align*} \xymatrixcolsep{3pc}\xymatrix{ L^2\left( \text{T}_{ \mathcal{E}}, m_{Haar}\right)  \ni h \,  \ar@{|->}[r] &  h \circ \pi \in L^2(X, m) }  
\end{align*}

It maps a continuous character on $\text{T}_{ \mathcal{E}}$, which corresponds to some $\omega \in \mathcal{E}$ by Pontryagin duality, to a normalized eigenfunction of the system whose eigenvalue is $\omega$. As a byproduct, its range is precisely the closed $\mathbb{G}$-stable Hilbert subspace $\mathcal{H}_{\mathrm{p}}$ of $L^2(X, m)$ spanned by the eigenfunctions of the system. The orthogonal projection onto this subspace is the spectral projector $\mathsf{E}(\mathcal{E})$ associated to the countable subset $\mathcal{E}\subset \widehat{\mathbb{G}}$, and can be linked with the disintegration theorem applied to the factor map $\pi$ over $\text{T}_{ \mathcal{E}}$ (regardless the choice of $\pi$, which is essentially unique):

\begin{prop}\label{prop:projection.and.disintegration} $\mathsf{E}(\mathcal{E})(f)$ is equal, for any $f\in \mathcal{L}^2(X, m)$, to the $L^2$-class of
\begin{align*}  \xymatrixcolsep{3.5pc}\xymatrix{ x  \ar@{|->}[r] &\displaystyle \int _X f \, d \mu _{\pi (x)} } 
\end{align*}
\end{prop}




\section{\textsf{Translation-bounded measures}}\label{Section:basics} 

\subsection{\textsf{Basics on Fourier analysis}}\label{Par:basics.fourier} ~ 

\vspace{0.2cm} 

Here we mainly follow \cite{BaaLen, Sch}. The $\mathbb{C}$-vector space of compactly supported complex valued continuous functions on $\mathbb{G}$ will be denoted $\mathcal{C}_c(\mathbb{G})$, and is topologized as follows: For each compact subset $K\subset \mathbb{G}$ let $\mathcal{C}_K(\mathbb{G})$ be the subspace of functions supported on $K$ equipped with the suppremum norm $\Vert . \Vert _{\infty}$, for which it is a Banach space. Then let $(K_N)_N$ be a nested sequence of compact regular sets in $\mathbb{G}$ whose interiors cover $\mathbb{G}$: It gives $\mathcal{C}_c(\mathbb{G}) = \cup _{N} \mathcal{C}_{K_N}(\mathbb{G})$ and defines an inductive limit topology $\mathcal{T}_{lim}$ on $\mathcal{C}_c(\mathbb{G})$ which is the weakest making the natural inclusions
\vspace{0.1cm} 
\begin{align*}\xymatrixcolsep{2pc}\xymatrix{(\mathcal{C}_{K_N}(\mathbb{G}), \Vert . \Vert _{\infty}) \,  \ar@{^{(}->}[r] & (\mathcal{C}_c(\mathbb{G}), \mathcal{T}_{lim}) }  
\end{align*}

\vspace{0.1cm} 
continuous. This topology does not depend on the choice of nested sequence $(K_N)_N$, and the resulting space $(\mathcal{C}_c(\mathbb{G}), \mathcal{T}_{lim})$ is a separable topological vector space. The space $\mathcal{C}_c(\mathbb{G})$ is closed under the operations $\phi _-:= \phi (-.)$ and $\widetilde{\phi }:= \overline{\phi (-.)}$, and under convolution product $\phi * \psi (t) := \int _{\mathbb{G}} \phi (s) \psi (t-s) \, ds$. One has the usual Fourier transform (\cite{Rudin})
\begin{align*}\xymatrixcolsep{3pc}\xymatrix{\mathcal{C}_c(\mathbb{G}) \ni \phi \ar@{|->}[r] &  \widehat{\phi} \in \mathcal{C}_0(\widehat{\mathbb{G}})} \qquad \text{with} \qquad \widehat{\phi}(\omega ) := \displaystyle \int _{\mathbb{G}} \phi (t) \overline{\omega (t)}\, dt
\end{align*}

The space $\mathcal{C}_0(\widehat{\mathbb{G}})$ of continuous complex valued functions on $\widehat{\mathbb{G}}$ null at infinity is closed under pointwise product, and one has the relations $\widehat{(\phi _-)} = (\widehat{\phi })_- $, $\widehat{\widetilde{\phi}} = \overline{\widehat{\phi}}$ and $\widehat{\phi * \psi} = \widehat{\phi}. \widehat{\psi} $ . We now turn to complex Borel measures on $\mathbb{G}$. The dual space $\mathcal{C}_c(\mathbb{G})^*$ of $\mathcal{C}_c(\mathbb{G})$ is by construction of the topology $\mathcal{T}_{lim}$ the space of linear functionals on $\mathcal{C}_c(\mathbb{G})$ which are bounded on each subspace $(\mathcal{C}_{K_N}(\mathbb{G}), \Vert . \Vert _{\infty})$ (with a constant depending on $K_N$), and is precisely given, according to the Riesz-Markov representation Theorem (\cite{ReedSimon}, Theorem IV. 18 therein), by the space $\mathcal{M}(\mathbb{G})$ of (possibly unbounded) complex Borel measures on $\mathbb{G}$, where a functional in $\mathcal{C}_c(\mathbb{G})^*$ is given by the integral against a measure in $\mathcal{M}(\mathbb{G})$. The vague topology $\mathcal{T}_{vague}$ on $\mathcal{M}(\mathbb{G})$ is the weak-$*$ topology on $\mathcal{C}_c(\mathbb{G})^*$ after identification, which is the weakest topology making the functions 
\begin{align}\label{elemenetary.functions}\xymatrixcolsep{3pc}\xymatrix{ \Lambda \;  \ar@{|->}[r] & \; \displaystyle \mathcal{N}_\phi (\Lambda  ) :=  \int _\mathbb{G} \, \phi _-\, d\Lambda = \int _\mathbb{G} \, \phi (-s)\, d\Lambda (s) }  
\end{align}

continuous on $\mathcal{M}(\mathbb{G})$ for each $\phi \in \mathcal{C}_c(\mathbb{G})$. Hence defined the space $(\mathcal{M}(\mathbb{G}), \mathcal{T}_{vague})$ is a locally convex topological vector space. It carries a natural $\mathbb{G}$-action defined for $t\in \mathbb{G}$ by $\xymatrixcolsep{2pc}\xymatrix{ \Lambda \,  \ar@{|->}[r] &  \Lambda * \delta _t } $, where the convolution product of two convolable measures is determined on $\phi \in \mathcal{C}_c(\mathbb{G})$ by $\int _{\mathbb{G}} \phi \, d\Lambda * \Lambda'  := \int _{\mathbb{G}} \int _{\mathbb{G}} \phi (s+t) \, d\Lambda (s) \, d \Lambda '(t)$. There is a natural involution on $\Lambda \in \mathcal{M}(\mathbb{G})$ set by $\widetilde{\Lambda}(\phi) = \overline{\Lambda(\widetilde{\phi})}$ for $\phi \in \mathcal{C}_c(\mathbb{G})$.


\subsection{\textsf{Dynamical systems of translation-bounded measures}}\label{Par:translation.bounded.measures} ~ 

\vspace{0.2cm}

We denote below $(K,M)$ to be a pair with $K \subset \mathbb{G}$ compact of non-empty interior and $M\geqslant 0$. A complex Borel measure $\Lambda\in \mathcal{M}(\mathbb{G})$ is called $(K,M)$-translation-bounded if
\vspace{0.1cm} 
\begin{align*}  \left| \Lambda \right| (K+t)\leqslant M \qquad \text{for each } \; t\in \mathbb{G}
\end{align*}

\vspace{0.1cm} 
Here, $\left| \Lambda \right|$ is the absolute value, or total variation of $\Lambda$, which is the least positive complex Borel measure on $\mathbb{G}$ such that $\left| \Lambda (B) \right| \leqslant\left| \Lambda \right|(B)$ for any Borel $B$. A complex Borel measure in $ \mathcal{M}(\mathbb{G})$ is called translation-bounded if it is $(K,M)$-translation-bounded for some $(K,M)$. The collection of $(K,M)$-translation-bounded measures on $\mathbb{G}$ will be denoted $\mathcal{M}_{(K,M)}(\mathbb{G})$, and that of translation-bounded measures $\mathcal{M}^{\infty}(\mathbb{G})$. It is clear that $\mathcal{M}^{\infty}(\mathbb{G})$, when endowed with the vague topology, is a topological vector subspace of $ \mathcal{M}(\mathbb{G})$ that remains stable under the natural $\mathbb{G}$-action.

\vspace{0.2cm} 

\begin{prop}\label{prop.convexity} \cite{BaaLen} Each set $\mathcal{M}_{(K,M)}(\mathbb{G})$ is a metrizable compact convex subset of $\mathcal{M}^{\infty}(\mathbb{G})$, on which the natural $\mathbb{G}$-action is well defined and continuous.
\end{prop}

\vspace{0.2cm} 

Compactness and metrizability is proved in \cite{BaaLen}, Theorem $2$, and continuity of the $\mathbb{G}$-action in Proposition $2$ therein. The convexity statement is ensured by the inequality $\left| \alpha\Lambda + (1-\alpha)\Lambda ' \right| \leqslant \alpha \left| \Lambda \right| + (1-\alpha) \left| \Lambda ' \right|$ holding for any two measures $\Lambda , \Lambda '$ and $\alpha\in [ 0, 1]$.

\vspace{0.2cm} 

\begin{de} A measured dynamical system of translation-bounded measures is a measured dynamical system $(X, \mathbb{G},m)$, with $X$ a compact subset of some $\mathcal{M}_{(K,M)}(\mathbb{G})$ and the $\mathbb{G}$-action given by $\xymatrixcolsep{1.8pc}\xymatrix{X \times \mathbb{G} \ni  (\Lambda , t) \,  \ar@{|->}[r] &  \Lambda * \delta _t \in X}$.
\end{de}

\vspace{0.2cm} 

A way to generate dynamical systems of translation-bounded measures is to start with a particular measure $\mathsf{\Lambda}$ in some $\mathcal{M}_{(K,M)}(\mathbb{G})$, and consider the closure $X_{\mathsf{\Lambda}}$ of its $\mathbb{G}$-orbit in $\mathcal{M}_{(K,M)}(\mathbb{G})$ which we shall call its \textit{hull}. Once a $\mathbb{G}$-invariant probability measure $m$ on $X_{\mathsf{\Lambda}}$ is chosen one ends up with a system of the desired form. To a translation-bounded measure can be associated as in \cite{BaaLen} its \textit{rubber local isomorphism class}, or RLI-class, to be the collection of all translation-bounded measures locally isomorphic with $\Lambda$,
\vspace{0.1cm} 
\begin{align}\label{RLI.class} \text{RLI}(\mathsf{\Lambda}) := \left\lbrace \Lambda \in \mathcal{M}^{\infty}(\mathbb{G}) \, \vert \, X_\Lambda = X_{\mathsf{\Lambda}}\right\rbrace 
\end{align}

\vspace{0.1cm} 
It is not hard to show that the RLI-class of a translation-bounded measure $\mathsf{\Lambda}$ is a Borel $\mathbb{G}$-stable subset of $\mathcal{M}^{\infty}(\mathbb{G})$ (it is even a $G_\delta$), and moreover one clearly has $\text{RLI}(\mathsf{\Lambda}) \subseteq X_{\mathsf{\Lambda}}$. Given a dynamical system of translation-bounded measures $(X, \mathbb{G},m)$, it is clear that whenever $\Lambda$ belongs to $X$ then the entire class $\text{RLI}(\Lambda)$ also belongs to $X$, so that $X$ always partitions into disjoint RLI-classes. Obviously $(X, \mathbb{G})$ is minimal if and only if $X$ contains a single RLI-class. If $(X, \mathbb{G},m)$ is ergodic, one knows by Proposition \ref{prop:ergodicity} that almost all $\Lambda \in X$ have the support of $m$ as hull. These elements form, as it is easy to remark, a single RLI-class in $X$, yielding:



\vspace{0.2cm} 

\begin{prop} An ergodic system of translation-bounded measures $(X, \mathbb{G},m)$ admits a single \textsc{RLI}-class of full $m$-measure in $X$, formed of the $\Lambda \in X$ satisfying $X_\Lambda = Supp(m)$.
\end{prop}

\vspace{0.2cm} 

This will be of main importance in Section \ref{Section:weighted.meyer.sets}. Any dynamical system of translation-bounded measures $(X, \mathbb{G}, m)$ has the natural set of continuous functions on $X$ given by the restriction of functions (\ref{elemenetary.functions}) on the compact subset $X$,
\begin{align}\label{process}  \xymatrixcolsep{3pc}\xymatrix{  \mathcal{C}_c(\mathbb{G}) \ni \phi \ar@{|->}[r] & \mathcal{N} _\phi = \mathcal{N} _\phi ^X \in \mathcal{C}(X)  }, \qquad  \mathcal{N}_\phi (\Lambda ) = \int _\mathbb{G} \, \phi _- \, d\Lambda  
\end{align}

This map is $\mathbb{C}$-linear and moreover satisfies $\mathcal{N} _\phi (. * \delta _t) = \mathcal{N} _{\phi * \delta _t}$ whenever $t\in \mathbb{G}$.

\subsection{\textsf{Intensities}} ~ 

\vspace{0.2cm} 

Let $\mathcal{M}_{(K,M)}(\mathbb{G})$ be the compact and convex subset of $(K,M)$-translation-bounded measures in $\mathcal{M}^{\infty}(\mathbb{G})$. The following statement is a particular case of a general result on vector-valued integration theory, see \cite{Rudin}, Theorem 3.27 therein:\\

\begin{theo}\label{theo:intensity} Suppose that $X$ is a compact subset of $\mathcal{M}_{(K,M)}(\mathbb{G})$. Then any bounded complex measure $\mu \in \mathcal{M}^b(X)$ defines a translation-bounded measure $\mathsf{i}(\mu) \in \mathcal{M}^{\infty}(\mathbb{G})$ by 
\begin{align*} \int _{\mathbb{G}} \phi \, d \, \mathsf{i}(\mu) = \int _{X} \int _{\mathbb{G}} \phi \, d\Lambda \,d\mu(\Lambda)
\end{align*}
If moreover $\mu$ is a probability measure then $\mathsf{i}(\mu)$ belongs to $\mathcal{M}_{(K,M)}(\mathbb{G})$.\\
\end{theo}

Given a compact subset $X$ of some $\mathcal{M}_{(K,M)}(\mathbb{G})$ and $\mu$ a complex measure on $X$, the resulting translation-bounded measure $\mathsf{i}(\mu)$ is sometimes called the integral, or barycenter, or also first moment of $\mu$ on $X$. However we shall call it here \textit{intensity}. This terminology comes from point-set theory, where one studies \textit{Delone sets} of $\mathbb{G}$ (see for instance \cite{BaaLen}): Given a compact space $\mathbb{X}$ of Delone sets of $\mathbb{G}$ (with common radius of uniform discreteness), the intensity $I(\mu)$ of a bounded measure $\mu$ on $\mathbb{X}$ is the translation-bounded measure obtained (\cite{Gou, Gou2}) on any Borel set $A\subseteq \mathbb{G}$ by  
\begin{align*} I(\mu)(A) = \int _{\mathbb{X}} \sharp ( S \cap A) \, d \mu ( S )
\end{align*}

A compact space $\mathbb{X}$ of Delone sets can be identified with a compact subset $X$ of some $\mathcal{M}_{(K,M)}(\mathbb{G})$ by replacing a point set $S$ by the Dirac comb $\delta _S:= \sum _{t \in S} \delta _t$ (see \cite{BaaLen}, Section 4 and Theorem 4 therein). Hence a measure $\mu$ on $\mathbb{X}$ can alternatively be viewed as living on $X$, and its intensity as given just above is equal to the translation-bounded measure yield by Theorem \ref{theo:intensity}.

Recall that if $X$ is a compact $\mathbb{G}$-stable subset of $\mathcal{M}_{(K,M)}(\mathbb{G})$ then $\mathbb{G}$ acts on the space $\mathcal{M}^b(X)$ of bounded complex Radon measures of $X$ via the formula $\mu .t := \mu (.* \delta _{-t})$. The following is then straightforward to prove:

\vspace{0.2cm} 

\begin{prop}\label{prop:barycenter} Let $X$ be a compact $\mathbb{G}$-stable subset of $\mathcal{M}_{(K,M)}(\mathbb{G})$, with $\mathcal{M}^b(X)$ its space of bounded complex measures. Then the map
\begin{align*} \xymatrixcolsep{3pc}\xymatrix{ (\mathcal{M}^b(X), \mathcal{T}_{vague}) \ni \mu  \ar@{|->}[r] & \mathsf{i}(\mu) \in (\mathcal{M}^\infty (\mathbb{G}), \mathcal{T}_{vague}) }
\end{align*}
is continuous and $\mathbb{G}$-commuting, that is, $\mathsf{i}(\mu (.* \delta _{-t}))= \mathsf{i}(\mu ) * \delta _t$ for any $t\in \mathbb{G}$.
\end{prop}

\vspace{0.2cm} 

\subsection{\textsf{Mathematical diffraction and connection with dynamical spectrum}}\label{Par:diffraction}  ~ 

\vspace{0.2cm} 


%

We provide here a very concise survey on diffraction, and refer the reader to \cite{Hof, Sch, BaaLen, BaaMoo} for more details, comments and proofs. The diffraction of a translation-bounded measure $\Lambda$ is defined by first considering a Van Hove sequence \cite{Sch} of subsets $(\mathcal{A}_n)_{n\in \mathbb{N}}$, form which one considers the following vague limit in $ \mathcal{M}^{\infty}(\mathbb{G})$ of truncated auto-correlations
\begin{align}\label{formula.auto-correlation}  \gamma _\mathsf{\Lambda} := \lim _{n \rightarrow \infty} \frac{1}{\left|\mathcal{A}_n \right|} \mathsf{\Lambda}\vert _{\mathcal{A}_n} * \widetilde{\mathsf{\Lambda}\vert _{\mathcal{A}_n}}
\end{align}

The above sequence may not converges in $\mathcal{M}^{\infty}(\mathbb{G})$ but it always will once one extract an appropriate sub-sequence (one shows by mimicking the proof of Proposition 2.2 of \cite{Hof} and using Lemma 1.1 in \cite{Sch} that such sequence belongs to a certain $\mathcal{M}_{(K,M)}(\mathbb{G})$, and the claim comes from compacity and metrizability of this latter set). Thus by considering this sub-sequence in (\ref{formula.auto-correlation}) one can get rid of this difficulty. Moreover again up to extraction one can assume that the Van Hove sequence is \textit{tempered} (\cite{Lin}, Proposition 1.4). The limit $ \gamma _\mathsf{\Lambda}$ is the \textit{auto-correlation measure} of $\mathsf{\Lambda} \in \mathcal{M}^{\infty}(\mathbb{G})$ (with respect to the Van Hove sequence $(\mathcal{A}_n)_{n\in \mathbb{N}}$), and is translation-bounded. Such measure is moreover \textit{positive-definite}, hence Fourier transformable (\cite{VanDenBergForst}, Theorem 4.7): There exists a unique Borel measure $\widehat{\gamma} _\mathsf{\Lambda}$ on the Pontryagin dual $\widehat{\mathbb{G}}$, positive and translation-bounded, such that at any $\phi \in \mathcal{C}_c(\mathbb{G})$
\begin{align}\label{formula.diffraction} \int _{\widehat{\mathbb{G}}} \vert \widehat{\phi}_-  \vert ^2 \, d\widehat{\gamma} _\mathsf{\Lambda} = \int _{\mathbb{G}} \phi  * \widetilde{\phi} \, d\gamma _\mathsf{\Lambda}
\end{align}


Whenever it exists, we call $\widehat{\gamma} _\mathsf{\Lambda}$ the \textit{diffraction measure} of $\mathsf{\Lambda} \in \mathcal{M}^{\infty}(\mathbb{G})$ (with respect to the sequence $(\mathcal{A}_n)_{n\in \mathbb{N}}$). Within a given dynamical system of translation-bounded measures $(X, \mathbb{G},m)$ different measures will have, in principle, different diffraction measures. However when $m$ is ergodic there is a unique typical resulting diffraction, that is, a measure $\widehat{\gamma}$ on $\widehat{\mathbb{G}}$ such that $\widehat{\gamma} = \widehat{\gamma}_{\Lambda}$ for $m$-almost every $\Lambda \in X$ and along any tempered Van Hove sequence (see \cite{BaaLen}, Theorem $5$). In fact, it is possible to get rid of the ergodicity assumption and provide a certain definition for the diffraction measure of a measured dynamical system of translation-bounded measures $(X, \mathbb{G},m)$, in a way that supplies the $m$-almost sure diffraction in the case of ergodicity:\\

\begin{theo}\label{theo:diffraction} \cite{BaaLen} Let $(X, \mathbb{G},m)$ be a measured dynamical system of translation-bounded measures. Then there exists a unique positive measure $\widehat{\gamma}$ on $\widehat{\mathbb{G}}$ satisfying the equality for each $\phi , \psi \in \mathcal{C}_c(\mathbb{G})$
\begin{align}\label{formula.diffraction.dynamic} \int _{\widehat{\mathbb{G}}} \widehat{\phi }\overline{\widehat{\psi}} \, d \widehat{\gamma} = \int _X \mathcal{N}_{\phi}\overline{\mathcal{N}_{\psi}} \, dm
\end{align}
When $(X, \mathbb{G},m)$ is ergodic the measure $\widehat{\gamma}$ is its $m$-almost sure diffraction measure.\\
\end{theo}

The proof is given by combining Theorem 5(b) and Lemma 7 of \cite{BaaLen} together with the formula defining the diffraction (\ref{formula.diffraction}). The formula (\ref{formula.diffraction.dynamic}) defining the diffraction of a measured dynamical system of translation-bounded measures extends to the following formula, as shown in Proposition 7 in \cite{BaaLen}, for any Borel subset $\mathcal{P}\subseteq \widehat{\mathbb{G}}$
\begin{align}\label{formula.diffraction.dynamic.generalized} \int _{\mathcal{P}} \widehat{\phi }\overline{\widehat{\psi}} \, d \widehat{\gamma} = \int _X \mathsf{E}(\mathcal{P})(\mathcal{N}_{\phi})\overline{\mathsf{E}(\mathcal{P})(\mathcal{N}_{\psi})} \, dm
\end{align}

where $\mathsf{E}(\mathcal{P})$ stands for the spectral projector associated with $\mathcal{P}$ on $L^2(X,m)$. The theorem above provides an efficient tool to deal with the diffraction measure: As it was suggested in \cite{Dwo, Sch, BaaLen} and latter on explicitly formulated in \cite{DenMoo, LenMoo, BaaLenVEn}, the equality set in the theorem above is nothing but an isometric embedding
\begin{align}\label{diffraction.to.dynamic.map} \xymatrixcolsep{3pc}\xymatrix{\Theta : L^2(\widehat{\mathbb{G}}, \widehat{\gamma})\, \,  \ar@{^{(}->}[r] & \, \, L^2(X,m)}
\end{align}

where on the dense subspace of $L^2(\widehat{\mathbb{G}}, \widehat{\gamma})$ formed of the Fourier transforms $\widehat{\phi} $ of compactly supported functions $\phi \in \mathcal{C}_c(\mathbb{G})$ the mapping $\Theta$ writes $\Theta (\widehat{\phi}) := \mathcal{N}_{\phi}$. This is the so-called \textit{diffraction-to-dynamic map} of $(X, \mathbb{G},m)$, called after \cite{LenMoo}. Its range $\mathcal{H}_\Theta$ is the closed subspace of $L^2(X,m)$ spanned by continuous functions on $X$ of the form $\mathcal{N}_{\phi}$ with $\phi \in \mathcal{C}_c(\mathbb{G})$, which is in general not be the whole space $L^2(X,m)$. This subspace is $\mathbb{G}$-stable and thus its orthogonal projection, which we denote $P_\Theta$, commutes with the $\mathbb{G}$-representation, and consequently also with any spectral projector. Now using the mapping $\Theta$ one can check that formula (\ref{formula.diffraction.dynamic.generalized}) admits the following equivalent form: Denoting for $f\in L^{\infty}(\widehat{\mathbb{G}}, \widehat{\gamma})$ the associated multiplication operator on $L^2(\widehat{\mathbb{G}}, \widehat{\gamma})$ by $M_f$, one has for any Borel subset $\mathcal{P}\subseteq \widehat{\mathbb{G}}$
\begin{align}\label{formula.spectral projectors} \Theta \circ M_{\mathbb{I}_{\mathcal{P}}} \circ \Theta ^{-1} = \mathsf{E}(\mathcal{P})P_\Theta
\end{align}

An alternative way to observe this is to note that both $M_{\mathbb{I}_{\mathcal{P}}}$ and $\mathsf{E}(\mathcal{P})P_\Theta$ are projectors, which are in fact spectral projectors for appropriates $\mathbb{G}$-representations on $L^2(\widehat{\mathbb{G}}, \widehat{\gamma})$ and $\mathcal{H}_\Theta$ (see Section 3 of \cite{LenMoo} for more about these representations). These representations are intertwined by the diffraction to dynamic map, hence naturally yielding the intertwining formula (\ref{formula.spectral projectors}) for the associated spectral projectors. The fundamental connection between the dynamical and diffraction spectra of a measured dynamical system of translation-bounded measures states as follows, where point $(i)$ straightly follows from formula (\ref{formula.spectral projectors}) above whereas point $(ii)$ is shown in \cite{BaaLen}, Theorems 6, 7 and 9 therein:\\

\begin{theo}\label{theo:dynamical.diffraction.spectra} \cite{BaaLen} Let $(X, \mathbb{G}, m)$ be a measured dynamical system of translation-bounded measures, with diffraction measure $\widehat{\gamma}$ and projection-valued measure $\mathsf{E}$. Then:\\

(i) $\widehat{\gamma}$ is absolutely continuous with respect to $\mathsf{E}$. In particular the set $\mathcal{S}$ of Bragg peaks, the atoms of $\widehat{\gamma}$, belongs to the eigenvalue group $\mathcal{E}$, the atoms of $\mathsf{E}$.

\vspace{0.2cm}

(ii) $\widehat{\gamma}$ is a pure point measure if and only if $\mathsf{E}$ is a pure point measure. In this case the set of Bragg peaks $\mathcal{S}$ algebraically generates the eigenvalue group $\mathcal{E}$.\\
\end{theo}



\section{\textsf{The decomposition of translation-bounded measures}}\label{Section:decomposition}

\subsection{\textsf{Existence of the decomposition}}  ~ 

\vspace{0.2cm}

In this section we give the proof of our main result:\\

\begin{theo}\label{theo:decomposition} Let $(X, \mathbb{G},m)$ be an ergodic system of translation-bounded measures with diffraction $\widehat{\gamma}= \widehat{\gamma}_{\mathrm{p}} + \widehat{\gamma}_{\mathrm{c}}$. There exist two ergodic systems of translation-bounded measures: 
\vspace{0.2cm}

\begin{itemize}
\item \textit{$(X_{\mathrm{p}}, \mathbb{G},m_{\mathrm{p}})$ having diffraction $\widehat{\gamma}_{\mathrm{p}}$,}
\vspace{0.2cm}

\item \textit{$(X_{\mathrm{c}}, \mathbb{G},m_{\mathrm{c}})$ having diffraction $\widehat{\gamma}_{\mathrm{c}}$,}
\end{itemize}
\vspace{0.2cm}

which are Borel factors of $(X, \mathbb{G},m)$ under two Borel factor maps
$$\xymatrixcolsep{3.5pc}\xymatrix{  X \ni \Lambda \, \ar@{|->}[r] & \,\Lambda _{\mathrm{p}} \in X_{\mathrm{p}} } \qquad \xymatrixcolsep{3.5pc}\xymatrix{  X \ni \Lambda \, \ar@{|->}[r] & \,\Lambda _{\mathrm{c}} \in X_{\mathrm{c}} }$$
such that the equality $\Lambda = \Lambda _{\mathrm{p}}+ \Lambda _{\mathrm{c}}$ occurs for $m$-almost every $\Lambda\in X$..\\
\end{theo}


\begin{proof} Let us begin with the construction of the two desired dynamical systems. First, denoting $\mathcal{E}$ the eigenvalue group of $(X, \mathbb{G},m)$ one has by Proposition \ref{prop.construction.Borel.factor.map} a Borel factor map form $X$ to T$_{\mathcal{E}}$. Composing this factor map with the disintegration of the measure $m$ over T$_{\mathcal{E}}$ gives us a Borel $\mathbb{G}$-map 
\begin{align*} \xymatrixcolsep{3pc}\xymatrix{ \mu ^{\mathrm{p}}: X \ni \Lambda  \ar@{|->}[r] & \mu ^{\mathrm{p}}_\Lambda := \mu_{\pi (\Lambda )} \in \mathcal{M}^b(X) } 
\end{align*}

having values in the subset of probability measures on $X$. Now the Borel factor map form $X$ to T$_{\mathcal{E}}$ dualizes in a $\mathbb{G}$-commuting isometric embedding of $L^2(\text{T}_{\mathcal{E}})$ in $L^2(X,m)$ with range the closed Hilbert subspace $\mathcal{H}_{\mathrm{p}}$ spanned by eigenfunctions of the system $(X, \mathbb{G},m)$. The orthogonal projector onto $\mathcal{H}_{\mathrm{p}}$ is the spectral projector $P_{\mathrm{p}} :=\mathsf{E}(\mathcal{E})$ associated with the Borel subset $\mathcal{E} \subset \widehat{\mathbb{G}}$, and for each $f\in \mathcal{C}(X)$ the $L^2$-class of $\Lambda  \longmapsto \int _X f \, d\mu ^{\mathrm{p}}_\Lambda $ is by Proposition \ref{prop:projection.and.disintegration} equal to $P_{\mathrm{p}}(f)$. On the other hand one has another Borel map, which is straightforwardly shown to be a $\mathbb{G}$-map,
\begin{align*} \xymatrixcolsep{3pc}\xymatrix{ \mu ^{\mathrm{c}} : X \ni \Lambda  \ar@{|->}[r] & \mu ^{\mathrm{c}}_\Lambda := \delta _\Lambda - \mu ^{\mathrm{p}}_\Lambda \in \mathcal{M}^b(X) } 
\end{align*}

with the property that for each $f\in \mathcal{C}(X)$ the $L^2$-class of $\Lambda  \longmapsto \int _X f \, d\mu ^{\mathrm{c}}_\Lambda $ is given by $f -\mathsf{E}(\mathcal{E})(f)= \mathsf{E}(\widehat{\mathbb{G}}\backslash \mathcal{E})(f) = : P_{\mathrm{c}} (f)$, the orthogonal projection of $f$ onto the subspace $\mathcal{H}_{\mathrm{c}}$ orthogonal to $\mathcal{H}_{\mathrm{p}}$. It is obvious that both $ \mu ^{\mathrm{p}}$ and $ \mu ^{\mathrm{c}}$ are valued in the compact $\mathbb{G}$-stable subset $\mathcal{M}^b_2(X)$ of $\mathcal{M}^b(X)$ of signed Radon measures of total variation less of equal to $2$. Denote by $Cv_2(X)$ the image under the intensity map of Proposition \ref{prop:barycenter} of the compact $\mathbb{G}$-stable subset $\mathcal{M}^b_2(X)$: It is then a compact $\mathbb{G}$-stable subset of $\mathcal{M}^{\infty}(\mathbb{G})$. By composing the previous Borel $\mathbb{G}$-maps with the barycenter map one gets two Borel $\mathbb{G}$-maps
\begin{align*} \xymatrixcolsep{3pc}\xymatrix{\pi _{\mathrm{p}} :  X \ni \Lambda \ar[r] & \pi _{\mathrm{p}}(\Lambda) := \mathsf{i}(\mu ^{\mathrm{p}}_\Lambda ) \in Cv_2(X) } \\ \xymatrixcolsep{3pc}\xymatrix{ \pi _{\mathrm{c}} : X \ni \Lambda \ar[r] & \pi _{\mathrm{c}}(\Lambda) := \mathsf{i}(\mu ^{\mathrm{c}}_\Lambda) \in Cv_2(X) }
\end{align*}

Pushforwarding the ergodic measure $m$ by there Borel $\mathbb{G}$-maps yield two ergodic probability measures $m_{\mathrm{p}}$ and $m_{\mathrm{c}}$ supported on $Cv_2(X)$, whose support are two compact $\mathbb{G}$-stable subsets of $Cv_2(X)$ which we denote $X_{\mathrm{p}}$ and $X_{\mathrm{c}}$ respectively. The subset of $\Lambda \in X$ having images $ \pi _{\mathrm{p}}(\Lambda)$ and $\pi _{\mathrm{c}}(\Lambda)$ in $X_{\mathrm{p}}$ and $X_{\mathrm{c}}$ respectively is a full Borel subset, so after possibly modifying the maps $ \pi _{\mathrm{p}}$ and $\pi _{\mathrm{c}}$ on a set of measure 0 of $X$ we can assume that these are valued in $X_{\mathrm{p}}$ and $X_{\mathrm{c}}$ respectively. What we therefore have is two ergodic systems of translation-bounded measures $(X_{\mathrm{p}}, \mathbb{G},m_{\mathrm{p}})$ and $(X_{\mathrm{c}}, \mathbb{G},m_{\mathrm{c}})$ as well as two Borel factor maps 
$$\xymatrixcolsep{3.5pc}\xymatrix{  X  \, \ar@{->}^-{\pi _{\mathrm{p}}}[r] & \, X_{\mathrm{p}} } \qquad \xymatrixcolsep{3.5pc}\xymatrix{  X \, \ar@{->}^-{\pi _{\mathrm{c}}}[r] & \, X_{\mathrm{c}} }$$
Our aim in the sequel is then to show that these two Borel factors satisfy the desired properties. First we point a remark on the summands $\widehat{\gamma}_{\mathrm{p}}$ and $\widehat{\gamma}_{\mathrm{c}}$ of the diffraction $\widehat{\gamma}$: first these are the restriction of $\widehat{\gamma}$ on the Borel subsets $\mathcal{E}$ and $\widehat{\mathbb{G}}\backslash \mathcal{E}$ respectively. Moreover, formula (\ref{formula.diffraction.dynamic.generalized}) writes
\begin{align*} \int _{\mathcal{P}} \widehat{\phi }\overline{\widehat{\psi}} \, d \widehat{\gamma} = \int _X \mathsf{E}(\mathcal{P})(\mathcal{N}_{\phi})\overline{\mathsf{E}(\mathcal{P})(\mathcal{N}_{\psi})} \, dm
\end{align*}

with $\mathsf{E}(\mathcal{P})$ the spectral projector associated to $\mathcal{P}$ on $L^2(X,m)$ so in the particular cases of $\mathcal{P}= \mathcal{E}$ or $\widehat{\mathbb{G}}\backslash \mathcal{E}$ this gives, with the notations $P_{\mathrm{p}} :=\mathsf{E}(\mathcal{E})$ and $P_{\mathrm{c}} :=\mathsf{E}(\widehat{\mathbb{G}}\backslash \mathcal{E})$,
\begin{align}\label{formula.simplification.diffraction} \int _{\widehat{\mathbb{G}}} \widehat{\phi }\overline{\widehat{\psi}} \, d \widehat{\gamma}_{\mathrm{p}} = \int _X P_{\mathrm{p}}(\mathcal{N}_{\phi})\overline{P_{\mathrm{p}}(\mathcal{N}_{\psi})} \, dm \qquad \qquad \int _{\widehat{\mathbb{G}}} \widehat{\phi }\overline{\widehat{\psi}} \, d \widehat{\gamma}_{\mathrm{c}} = \int _X P_{\mathrm{c}}(\mathcal{N}_{\phi})\overline{P_{\mathrm{c}}(\mathcal{N}_{\psi})} \, dm
\end{align}
Let us then show that our two systems have diffraction $\widehat{\gamma}_{\mathrm{p}}'$ and $\widehat{\gamma}_{\mathrm{c}}'$ equal to $\widehat{\gamma}_{\mathrm{p}}$ and $\widehat{\gamma}_{\mathrm{c}}$ respectively. For either $\alpha =\mathrm{p}$ or $\mathrm{c}$ consider
\begin{align*}  \xymatrixcolsep{3pc}\xymatrix{  \mathcal{C}_c(\mathbb{G}) \ni \phi \ar@{|->}[r] & \mathcal{N}^\alpha _\phi  \in \mathcal{C}(X_\alpha )  }, \qquad \mathcal{N}^\alpha _\phi (\Lambda ) := \int \phi _- \, d \Lambda 
\end{align*}

It comes for each $\phi \in \mathcal{C}_c(\mathbb{G})$ that $\mathcal{N}^\alpha _\phi \circ \pi _\alpha $ has $L^2$-class equal to $P_\alpha(\mathcal{N}_\phi )$ in $L^2(X, m)$. Indeed it follows from the series of almost everywhere equalities
\begin{align*}\mathcal{N}^\alpha _\phi \circ \pi _\alpha  (\Lambda)  = \int _{\mathbb{G}} \phi _- \, d\pi _\alpha (\Lambda)  = \int _{\mathbb{G}} \phi _- \, d\, \mathsf{i}(\mu ^\alpha_{\Lambda }) = \int _{X} \left[  \int _{\mathbb{G}} \phi _- \, d\Lambda ' \right] d\mu ^\alpha_{\Lambda }(\Lambda ')  = \int _{X} \mathcal{N} _\phi \,d \mu  ^\alpha_{\Lambda}
\end{align*}

whose $L^2$-class is from what has been said earlier in this proof equal to $P_\alpha(\mathcal{N} _\phi)$ in $ L^2(X, m)$. One therefore has, invoking equality of Theorem \ref{theo:diffraction} combined with (\ref{formula.simplification.diffraction}),
\begin{align*}\int \widehat{\phi}\widehat{\overline{\psi}} \, d\widehat{\gamma}_{\alpha}'  = \int _{X_\alpha} \mathcal{N}^\alpha _\phi \overline{\mathcal{N}^\alpha _\psi} \, dm_\alpha  = \int _{X} P_\alpha(\mathcal{N} _\phi)\overline{P_\alpha(\mathcal{N} _\psi )} \, dm = \int \widehat{\phi}\widehat{\overline{\psi}} \, d\widehat{\gamma}_\alpha 
\end{align*}

yielding the desired equalities. It then remains to show the decomposition statement, that is, we need to show that $\Lambda = \pi _{\mathrm{p}}(\Lambda )+\pi _{\mathrm{c}}(\Lambda)$ for almost every $\Lambda\in X$, but this comes from the straightforward computation for any $\phi \in \mathcal{C}_c(\mathbb{G})$
\vspace{0.2cm}
\begin{align*}\int _{\mathbb{G}} \phi _- \, d(\pi _{\mathrm{p}}(\Lambda )+\pi _{\mathrm{c}}(\Lambda))  = \int _{\mathbb{G}} \phi _- \, d\pi _{\mathrm{p}}(\Lambda )+ \int _{\mathbb{G}} \phi _- \, d\pi _{\mathrm{c}}(\Lambda)  =  \mathcal{N}^{\mathrm{p}} _\phi (\pi_{\mathrm{p}} (\Lambda )) + \mathcal{N}^{\mathrm{c}} _\phi (\pi_{\mathrm{c}} (\Lambda )) 
\end{align*}

which is in turn equal for $m$-almost every $\Lambda \in X$ to
\begin{align*}
P_{\mathrm{p}}(\mathcal{N}_\phi )(\Lambda ) +  P_{\mathrm{c}}(\mathcal{N}_\phi )(\Lambda )  = \mathcal{N} _\phi (\Lambda )   = \int _{\mathbb{G}} \phi _- \, d\Lambda 
\end{align*}

Applying this to a countable dense collection of functions $\phi \in \mathcal{C}_c(\mathbb{G})$ one deduces the existence of a Borel subset of full measure in $X$ such that the equality $\Lambda = \pi _{\mathrm{p}}(\Lambda )+\pi _{\mathrm{c}}(\Lambda)$ occurs, yielding the proof.
\end{proof}

\vspace{0.2cm}

Let us make a few remarks on this result. Here we consider an ergodic system of translation-bounded measures $(X, \mathbb{G},m)$ with eigenvalue group $\mathcal{E}$ and diffraction $\widehat{\gamma}$.

\begin{rem} We shall point a remark about the eigenvalue groups of the two systems yield by Theorem \ref{theo:decomposition}. The first associated system $(X_{\mathrm{p}}, \mathbb{G},m_{\mathrm{p}})$ has by construction pure point diffraction, and thus by Theorem \ref{theo:dynamical.diffraction.spectra} pure point dynamical spectrum, with eigenvalue group $\mathcal{E}_{\mathrm{p}}$ a subgroup of $\mathcal{E}$. This latter inclusion may however be strict: Indeed the set of Bragg peaks of $(X_{\mathrm{p}}, \mathbb{G},m_{\mathrm{p}})$ is precisely the set $\mathcal{S}$ of atoms of $\widehat{\gamma}$, and therefore the eigenvalue group $\mathcal{E}_{\mathrm{p}}$ is the subgroup of $\mathcal{E}$ algebraically generated by $\mathcal{S}$, which might not be the entire group $\mathcal{E}$. On the other hand, The second associated system $(X_{\mathrm{c}}, \mathbb{G},m_{\mathrm{c}})$ has by construction no Bragg peak (not even at $0$), but it remains unclear to us whether its eigenvalue group is trivial or not (this latter must however belong to the group $\mathcal{E}$).
\end{rem}

\begin{rem}\label{remark:autocorrelation} The autoccorelation $\gamma$ of a generic element in $X$ is positive-definite, hence a weakly almost periodic measure on $\mathbb{G}$ \cite{ArgDLa}. As such, its admits a unique decomposition into the sum of a \textit{strongly almost periodic} measure $\gamma_{sap}$ and a \textit{null-weakly almost periodic} one $\gamma_{0wap}$, which are both Fourier transformable, and such that $\widehat{\gamma_{sap}}= \widehat{\gamma }_{\mathrm{p}}$ and $\widehat{\gamma_{0wap}}= \widehat{\gamma }_{\mathrm{c}}$, see \cite{ArgDLa} for definitions and proofs. In our situation the two systems $(X_{\mathrm{p}}, \mathbb{G},m_{\mathrm{p}})$ and $(X_{\mathrm{c}}, \mathbb{G},m_{\mathrm{c}})$ have diffraction $\widehat{\gamma }_{\mathrm{p}}$ and $\widehat{\gamma }_{\mathrm{c}}$ respectively, so it straightly follows that the auto-correlation of a generic element of $(X_{\mathrm{p}}, \mathbb{G},m_{\mathrm{p}})$ is exactly the strong almost periodic part $\gamma_{sap}$ of $\gamma$ while a generic element of $(X_{\mathrm{c}}, \mathbb{G},m_{\mathrm{c}})$ has auto-correlation the null weakly almost periodic part $\gamma_{0wap}$ of $\gamma$.
\end{rem}

\begin{rem}\label{remark:positive} The translation-bounded measures belonging in $(X_{\mathrm{p}}, \mathbb{G},m_{\mathrm{p}})$ are, for almost all of them, the intensity of a certain probability measure on $X$. It then a direct verification that if $X$ consists of positive translation-bounded measures then so does $X_{\mathrm{p}}$. In case translation-bounded measures in $X$ are only signed then the translation-bounded measures of both $X_{\mathrm{p}}$ and $X_{\mathrm{c}}$ are signed as well.\\
\end{rem}


In Theorem \ref{theo:decomposition} just above the uniqueness part, that is, part (ii) in the statement of Theorem 1 in the introduction, is absent. We shall prove this in Paragraph \ref{paragraph:uniqueness.decomposition} below. Before we found elegant to turn this uniqueness statement into an operator-theoretic formalism, as it is done in next paragraph, allowing us to set a proof by only invoking standard arguments from operator theory. 

\subsection{\textsf{From Borel factors to operators}}\label{paragraph:factor.to.operator} ~ 

\vspace{0.2cm} 

Let $(X, \mathbb{G},m)$ be a measured dynamical system of translation-bounded measures, with diffraction to dynamic map
\begin{align*}\xymatrixcolsep{3.5pc}\xymatrix{\Theta  : L^2(\widehat{\mathbb{G}}, \widehat{\gamma} )\, \,  \ar@{<->}[r] & \, \, \mathcal{H}_{\Theta } \subseteq L^2(X,m)}, \qquad \Theta (\widehat{\phi}) = \mathcal{N}_\phi
\end{align*}

for any $\phi \in \mathcal{C}_c(\mathbb{G})$. Recall that $P_\Theta$ stands for the orthogonal projection onto $\mathcal{H}_{\Theta }$. Now assume we are given a dynamical system of translation-bounded measures $(X_{\pi}, \mathbb{G},m_{\pi})$ which is a Borel factor of $(X, \mathbb{G},m)$ under a Borel factor map $\mathrm{\pi } : X  \longrightarrow X_{\pi}$. Suppose in addition that its diffraction measure $\widehat{\gamma}_{\pi}$ is absolutely continuous with respect to $\widehat{\gamma}$, with an essentially bounded Radon-Nikodym differential $f_{\pi}:= \frac{d\widehat{\gamma}_{\pi}}{d\widehat{\gamma}} \in L^{\infty}(\widehat{\mathbb{G}}, \widehat{\gamma})$. Let us denote
\begin{align*}  \xymatrixcolsep{3pc}\xymatrix{  \mathcal{C}_c(\mathbb{G}) \ni \phi \ar@{|->}[r] & \mathcal{N}^{\pi} _\phi  \in \mathcal{C}(X_{\pi})  }, \qquad \mathcal{N}^{\pi} _\phi (\Lambda ) := \int \phi _- \, d \Lambda 
\end{align*}

 We shall here construct a bounded linear operator $Q_{\pi}$ on $L^2(X,m)$, which characterizes the factor map $\pi $, hence the factor system $(X_{\pi}, \mathbb{G},m_{\pi})$ itself, in a unique way:

\vspace{0.2cm} 

\begin{prop}\label{prop:factor.to.operator} Let $\mathrm{\pi } : X  \longrightarrow X_{\pi} $ with $(X_{\pi}, \mathbb{G},m_{\pi})$ as above.
\vspace{0.2cm} 
\begin{itemize}
\item[(i)] There exists a unique $ Q_{\pi}\in \mathcal{B}(L^2(X,m))$ such that $\mathcal{N}^{\pi}_{ \phi }\circ \pi     = Q_{\pi}( \mathcal{N}_\phi )$ in the $L^2$ sense for any $ \phi \in \mathcal{C}_c(\mathbb{G})$ and subject to the condition that $Q_{\pi} = Q_{\pi}P_\Theta$.
\\
\item[(ii)] For $\pi$, $\pi '$ as above then $Q_{\pi} = Q_{\pi '}$ if and only if $ \pi = \pi '$ almost everywhere on $X$.
\\
\item[(iii)] The following equality holds, with $ M_{f_{\pi}}$ the multiplication operator by $f_\pi$ on $L^2(\widehat{\mathbb{G}}, \widehat{\gamma})$,
$$\Theta \circ M_{f_{\pi}} \circ \Theta ^{-1}= Q_{\pi}^* Q_{\pi}$$
\end{itemize}
\end{prop}

\vspace{0.2cm} 

\begin{proof} (i) First the factor map $\pi $ induces a $\mathbb{G}$-comuting isometry
\begin{align*} \xymatrixcolsep{3.5pc}\xymatrix{i_{\pi}:\, L^2(X_{\pi}, m_{\pi})\ni f\,  \ar@{|->}[r] & f \circ \pi  \in  L^2(X, m) }  
\end{align*}

Denote the diffraction to dynamic map of $(X_{\pi}, \mathbb{G},m_{\pi})$ by 
\begin{align*}\xymatrixcolsep{3.5pc}\xymatrix{\Theta _{\pi} : L^2(\widehat{\mathbb{G}}, \widehat{\gamma}_{\pi})\, \,  \ar@{<->}[r] & \, \, \mathcal{H}_{\Theta _{\pi}} \subseteq L^2(X_{\pi},m_{\pi})}
\end{align*}

Now $\widehat{\gamma}_{\pi} = f_{\pi} \widehat{\gamma}$ for some $f_{\pi} \in L^\infty (\widehat{\mathbb{G}}, \widehat{\gamma})$ and therefore $ \Vert f \Vert _{_{L^2 (\widehat{\mathbb{G}}, \widehat{\gamma}_{\pi})}} \leqslant \Vert f_{\pi} \Vert _{_{\infty , ess}}^{\frac{1}{2}} \Vert f \Vert _{_{L^2 (\widehat{\mathbb{G}}, \widehat{\gamma})}}$ for any $f \in L^2 (\widehat{\mathbb{G}}, \widehat{\gamma})$. This shows that the identity map from $ L^2 (\widehat{\mathbb{G}}, \widehat{\gamma})$ to $ L^2 (\widehat{\mathbb{G}}, \widehat{\gamma}_{\pi})$ is well-defined and a bounded linear map, and thus there exists a unique bounded operator $\widetilde{Q}_{\pi}$ making the diagram commutative
\begin{align*} \xymatrixcolsep{3.5pc}\xymatrixrowsep{3.5pc}\xymatrix{ L^2(\widehat{\mathbb{G}}, \widehat{\gamma})\, \, \ar@{->}_-{\mathrm{id}}[d]  \ar@{<->}^-{\Theta}[r] & \, \, \mathcal{H}_{\Theta }  \ar@{->}^-{\widetilde{Q}_{\pi}}[d]  \\
L^2(\widehat{\mathbb{G}}, \widehat{\gamma}_{\pi})\, \,  \ar@{<->}^-{\Theta _{\pi}}[r] & \, \, \mathcal{H}_{\Theta _{\pi}} } 
\end{align*}

Now the operator $Q_{\pi}:= i_{\pi} \circ \widetilde{Q}_{\pi} \circ P_\Theta$ is a well-defined linear operator of $L^2(X,m)$, which is bounded with operator norm $\Vert Q_{\pi} \Vert _{op} \leqslant \Vert f_{\pi} \Vert _{_{\infty , ess}}^{\frac{1}{2}}$ as it is not difficult to check.By construction one has $Q_{\pi} = Q_{\pi}P_\Theta$, and for any $\phi \in \mathcal{C}_c(\mathbb{G})$ the almost everywhere equalities
\begin{align}\label{elementary.equality.operators} Q_{\pi}( \mathcal{N} _\phi )=  i_{\pi} \circ \widetilde{Q}_{\pi} \circ \Theta (\widehat{\phi}) =  i_{\pi} (\Theta _{\pi} (\widehat{\phi}))) = i_{\pi} (( \mathcal{N}^{\pi} _\phi )) =  \mathcal{N}^{\pi} _\phi  \circ \pi 
\end{align}

Uniqueness of the operator $Q_{\pi}$ having these properties is clear, as it is uniquely defined on the whole subspace $\mathcal{H}_\Theta$ and must vanish on its orthogonal space, giving (i).

(ii) The given equivalence statement easily follows, as one can observe that  $\pi = \pi '$ almost everywhere on $X$ if and only if, for any $\phi \in \mathcal{C}_c(\mathbb{G})$, one has for almost all $\Lambda \in X$
\begin{align*} \int \phi _- \, d \, \pi (\Lambda) =  \int \phi _- \, d \, \pi ' (\Lambda),
\end{align*}

that is, according to equalities (\ref{elementary.equality.operators}), if and only is $Q_{\pi}( \mathcal{N} _\phi ) = Q_{\pi '}( \mathcal{N} _\phi )$ for any $\phi \in \mathcal{C}_c(\mathbb{G})$. This latter condition holds if and only if $Q_{\pi}$ and $Q_{\pi '}$ coincide on the closed subspace $\mathcal{H}_{\Theta }$, and since by construction one has always $Q_{\pi}= Q_{\pi} P_\Theta$, with $P_\Theta$ the orthogonal projection onto $\mathcal{H}_{\Theta }$, this is equivalent to have $Q_{\pi} = Q_{\pi '}$ on all $L^2(X,m)$, as desired.

(iii) For the last point it suffices to check the equality in scalar products against functions $\mathcal{N}_\phi$, $\mathcal{N}_\psi$ for $ \phi , \psi \in \mathcal{C}_c(\mathbb{G})$: One has
\vspace{0.2cm}
\begin{align*} \langle \Theta \circ M_{f_{\pi}} \circ \Theta ^{-1}(\mathcal{N}_\phi ), \mathcal{N}_\psi\rangle _{m} =  \langle M_{f_{\pi}} \widehat{\phi}, \widehat{\psi}\rangle _{\widehat{\gamma}} = \int f_{\pi}\widehat{\phi}\overline{\widehat{\psi}} \, d\widehat{\gamma} = \int \widehat{\phi}\overline{\widehat{\psi}} \, d\widehat{\gamma}_{\pi} = \langle \mathcal{N}^{\pi}_\phi, \mathcal{N}^{\pi}_\psi\rangle _{m_{\pi}}
\end{align*}

on one hand, equal by point (i) of this proof to
\begin{align*} \langle Q_{\pi} ( \mathcal{N}_\phi), Q_{\pi}( \mathcal{N}_\psi) \rangle _m = \langle Q_{\pi}^*Q_{\pi} ( \mathcal{N}_\phi),  \mathcal{N}_\psi \rangle _m 
\end{align*}

which settles the proof.
\end{proof}

\subsection{\textsf{Uniqueness of decomposition of measures}} \label{paragraph:uniqueness.decomposition}  ~ 

\vspace{0.2cm} 

Using the result of the earlier paragraph we shall now prove uniqueness of the decomposition set in Theorem \ref{theo:decomposition}:\\

\begin{theo}\label{theo:uniqueness.decomposition} Let $(X, \mathbb{G},m)$ be an ergodic system of translation-bounded measures. Then the two ergodic systems provided in Theorem \ref{theo:decomposition} exist in a unique way, as well as the two Borel factor maps up to almost everywhere equality.\\
\end{theo}



\begin{proof} Let $\mathcal{S}$ be the countable subset of atoms of $\widehat{\gamma}$, with complementary set $\widehat{\mathbb{G}}\backslash \mathcal{S}$. Let then $(X_{\mathrm{p}}', \mathbb{G},m_{\mathrm{p}}')$ and $(X_{\mathrm{p}}', \mathbb{G},m_{\mathrm{p}}')$ be Borel factors of $(X, \mathbb{G},m)$ under two Borel factor maps $\pi_{\mathrm{p}}'$ and $\pi_{\mathrm{c}}'$, with respective diffraction $\widehat{\gamma}_{\mathrm{p}} = \mathbb{I}_{\mathcal{S}}\widehat{\gamma} $ and $\widehat{\gamma}_{\mathrm{p}} = \mathbb{I}_{\widehat{\mathbb{G}}\backslash \mathcal{S}}\widehat{\gamma} $ and such that $\Lambda = \pi_{\mathrm{p}}'(\Lambda) + \pi_{\mathrm{c}}'(\Lambda)$ hold for almost every $\Lambda \in X$, as in Theorem \ref{theo:decomposition}. By Proposition \ref{prop:factor.to.operator} the factor maps $\pi_{\mathrm{p}}'$ and $\pi_{\mathrm{c}}'$ give rise to a pair of operators $Q$ and $Q'$ on $L^2(X,m)$. Then it will be sufficient to show that
\begin{align}\label{Q and Q' projectors} Q = \mathsf{E}(\mathcal{S})P_\Theta , \qquad Q' = \mathsf{E}(\widehat{\mathbb{G}}\backslash \mathcal{S})P_\Theta
\end{align}

Indeed by the equivalence set in Proposition \ref{prop:factor.to.operator} this will readily ensure uniqueness of the pair of Borel factors, and that of Borel factor maps up to almost everywhere equality. To show equalities (\ref{Q and Q' projectors}), first observe that $Q+Q'=P_\Theta$: Indeed this follows from the equalities holding for any $\phi \in \mathcal{C}_c(\mathbb{G})$ and almost every$ \Lambda \in X$
\begin{align*} Q(\mathcal{N}_\phi)(\Lambda)+Q'(\mathcal{N}_\phi)(\Lambda) = \mathcal{N}^{\pi_{\mathrm{p}}'}_\phi (\Lambda)+ \mathcal{N}^{\pi_{\mathrm{c}}'}_\phi (\Lambda) = \int \phi _- \, d \, \pi_{\mathrm{p}}'(\Lambda) + \int \phi _- \, d \, \pi_{\mathrm{c}}'(\Lambda) = \mathcal{N}_\phi (\Lambda)
\end{align*}

the latest equality being given by the almost everywhere equality $\Lambda = \pi_{\mathrm{p}}'(\Lambda) + \pi_{\mathrm{c}}'(\Lambda)$. Second, observe that the assumption on diffraction implies by point (iii) of Proposition \ref{prop:factor.to.operator} that $Q^*Q$ and $Q^{'*}Q'$ are equal to the operators $\Theta \circ M_{\mathbb{I}_{\mathcal{S}}} \circ \Theta ^{-1}$ and $\Theta \circ M_{\mathbb{I}_{\widehat{\mathbb{G}}\backslash \mathcal{S}}} \circ \Theta ^{-1}$ respectively. From formula (\ref{formula.spectral projectors}) one deduces that the absolute values of $Q$ and $Q'$ (see \cite{ReedSimon} for definition) satisfy $\vert Q \vert = \vert Q \vert ^2= Q^*Q = \mathsf{E}(\mathcal{S})P_\Theta$ whereas $\vert Q' \vert = \vert Q' \vert ^2= Q^{'*}Q' = \mathsf{E}(\widehat{\mathbb{G}}\backslash \mathcal{S})P_\Theta$. Now the bounded operators $Q$ and $Q'$ admit a polar decomposition (\cite{ReedSimon}, Theorem VI.10 therein) of in the form $Q = V \vert Q \vert$ and $Q' = V' \vert Q' \vert$, where $V$ and $V'$ are partial isometries with initial spaces the ranges of $\vert Q \vert$ and $\vert Q' \vert$ respectively, that is to say, the ranges of the projectors $\mathsf{E}(\mathcal{S})P_\Theta$ and $\mathsf{E}(\widehat{\mathbb{G}}\backslash \mathcal{S})P_\Theta$. Then to show equalities (\ref{Q and Q' projectors}) it only remains to show that these partial isometries are the identity on their respective initial space. Since $Q+Q'=P_\Theta$ one has for any $h$ in the initial space of $V$ that, reminding that $\mathsf{E}(\mathcal{S})P_\Theta$ and $\mathsf{E}(\widehat{\mathbb{G}}\backslash \mathcal{S})P_\Theta$ have orthogonal ranges, $$V(h)= V \mathsf{E}(\mathcal{S})P_\Theta (h) =  V \mathsf{E}(\mathcal{S})P_\Theta(h) +  V' \mathsf{E}(\widehat{\mathbb{G}}\backslash \mathcal{S})P_\Theta (h)= (Q+Q')(h) =h$$ whereas for any $h'$ in the initial space of $V'$ that $$V'(h')= V' \mathsf{E}(\widehat{\mathbb{G}}\backslash \mathcal{S})P_\Theta (h') =  V \mathsf{E}(\mathcal{S})P_\Theta(h') +  V' \mathsf{E}(\widehat{\mathbb{G}}\backslash \mathcal{S})P_\Theta(h')= (Q+Q')(h') =h'$$
as desired.
\end{proof}

Let us remark that, given dynamical system of translation-bounded measures $(X, \mathbb{G}, m)$ with diffraction $\widehat{\gamma}$, if one drops the requirement of forming a decomposition of (almost every) measures in $X$ then one may exhibit several different Borel factors having diffraction $\widehat{\gamma}_{\mathrm{p}}$ and $\widehat{\gamma}_{\mathrm{c}}$ respectively. Indeed given a pair of such Borel factors it suffices to consider for each another system of translation-bounded measures which is Borel conjugated and \textit{homometric}, that is, with same diffraction, which even in the purely diffractive situation (in this case any two homometric systems are Borel conjugated) seems very likely to exist in general \cite{Ter}.

\subsection{\textsf{Support of positive translation-bounded measures}}  ~ 

\vspace{0.2cm} 

In this paragraph we shall exclusively deal with positive translation-bounded measures. Here we will be interested, given an ergodic system $(X, \mathbb{G},m)$ of positive translation-bounded measures, by the support of the positive, according to Remark \ref{remark:positive}, translation-bounded measures of the system $(X_{\mathrm{p}}, \mathbb{G},m_{\mathrm{p}})$ resulting from Theorem \ref{theo:decomposition}. At first one expect, as a very consequence of the ergodicity property, that the support of these measures should be large. But large in which sense ? For instance, does these measures have \textit{relatively dense} support, or at least almost all of them ? Here a subset $B \subset \mathbb{G}$ is relatively dense if there exists a compact set $K \subset \mathbb{G}$ with $B+K = \mathbb{G}$. This is not a trivial point and it seems judicious to deal with a weaker notion than relative density.


Such a weaker notion is set in Proposition \ref{prop:density} just below. First select a Van Hove sequence $(\mathcal{A}_k)_{k\in \mathbb{N}}$ of $\mathbb{G}$, which always exist in $\sigma$-compact LCA groups \cite{Sch}. We let then such a sequence be chosen. Then one defines the \textit{density} of a subset $B \subseteq \mathbb{G}$ (along the Van Hove sequence $(\mathcal{A}_k)_{k\in \mathbb{N}}$) is, whenever it exists, given by
\begin{align*} dens(B):= \lim _{k\rightarrow \infty} \dfrac{\vert B \cap \mathcal{A}_k \vert}{\vert\mathcal{A}_k \vert}
\end{align*}

This definition of density may differ from the one used elsewhere in the literature, specially for point sets, whose density always vanish in our sense. The density of a subset $B \subseteq \mathbb{G}$ as defined above is, if it exists, a value comprised between $0$ and $1$, and often depends on the choice of Van Hove sequence made at the beginning. Using this one has:

\vspace{0.2cm} 

\begin{prop}\label{prop:density} Let $(X, \mathbb{G}, m)$ be a non-trivial ergodic system of positive translation-bounded measures. Then almost any $\Lambda \in X$ has the following property: for any $\varepsilon > 0$ there is a compact set $K \subset \mathbb{G}$ such that $dens(Supp(\Lambda) +K)$ is greater than $1-\varepsilon$.
\end{prop}

\vspace{0.2cm} 

\begin{proof} Consider an increasing sequence of symmetric open sets $U_n\subset \mathbb{G}$ with compact closure and whose union is all $\mathbb{G}$, and set $Q_n$ to be the subset of $\Lambda \in X$ such that $0 \in Supp(\Lambda)+U_n$. Then each $Q_n$ is Borel in $X$. For, consider a sequence $\phi _n \in \mathcal{C}_c(\mathbb{G})$ of continuous functions, with $\phi _n$ being nowhere vanishing inside $U_n$ and entirely vanishing outside for each $n\in \mathbb{N}$ (such functions always exist because $\mathbb{G}$ is metrizable). Then for each $\Lambda \in X$ and $s\in \mathbb{G}$ there is an equivalence between having $s\notin Supp(\Lambda )+U_n$, $\Lambda (U_n +s)= 0$ and $\int \phi _n (.- s) \, d \Lambda = 0$. Indeed equivalence between the two last conditions is obvious, whereas if $\Lambda (U_n+s) >0 $ then $Supp(\Lambda)$ must cross $U_n+s$, yielding $s\in Supp(\Lambda)-U_n= Supp(\Lambda)+U_n$, and in the other direction if there is an $s$ with $\Lambda (U_n+s)=0$ then $Supp(\Lambda)$ cannot intersect the open set $U_n+s$ and thus $s$ cannot belong to $Supp(\Lambda)+U_n$. As a result, one is allowed to write $Q_n = \mathcal{N}_{\phi _n} ^{-1}(]0; + \infty [)$, which is open and thus Borel in $X$.

Now the Pointwise Ergodic Theorem (see Theorem \ref{theo:Birkhoff}, and \cite{MulRic} for further details), when applied to the Borel indicator function $\mathbb{I}_{Q_n}$, yields a Borel subset of full measure $X^{(n)}\subseteq X$ such that one has for any $\Lambda \in X^{(n)}$ the convergence
\begin{align*}  \xymatrixrowsep{4pc}\xymatrix{\dfrac{1}{\vert \mathcal{A}_k \vert } \displaystyle \int _{\mathcal{A}_k} \mathbb{I}_{Q_n} (\Lambda .s ) \, ds \, \,  \ar@{->}^-{k \, \rightarrow \, \infty}[r] &  \displaystyle \, \, \int _X \mathbb{I}_{Q_n} \, dm \, = \,  m(Q_n) } 
\end{align*}

However $\mathbb{I}_{Q_n} (\Lambda .s ) \neq 0$ if and only if $0 \in Supp(\Lambda .s) + U_n= Supp(\Lambda ) -s + U_n$, that is, $s \in Supp(\Lambda ) + U_n$ so that the left terms are nothing but
\begin{align*} \dfrac{1}{\vert \mathcal{A}_k \vert } \displaystyle \int _{\mathcal{A}_k} \mathbb{I}_{Q_n} (\Lambda .s ) \, ds  \, = \, \dfrac{\vert (Supp(\Lambda ) + U_n) \cap \mathcal{A}_k \vert }{\vert \mathcal{A}_k \vert }
\end{align*}
so that we arrive by taking limit to the equality, which holds for any $\Lambda \in X^{(n)}$,
\begin{align*}  dens ( Supp(\Lambda ) + U_n) = m(Q_n)
\end{align*}
Let us show that $m(Q_n)$ converges to 1: First, an the sequence $U_n$ is increasing in $\mathbb{G}$ the sequence $Q_n$ is also increasing in $X$. Therefore $m(Q_n)$ converges to $m(\bigcup _n Q_n)$. Now a $\Lambda$ not belonging in $ \bigcup _n Q_n$ must satisfy $0 \notin Supp(\Lambda)+U_n$ for each integer $n$, which is only possible for the everywhere trivial measure. Since $(X, \mathbb{G}, m)$ was assumed non-trivial it follows that $\bigcup _n Q_n$ has complementary set of measure 0, so is of measure 1 in $X$, as desired.

As a result, for any $\Lambda$ belonging in the Borel susbset of full measure $X^{(\infty)}$ given by the countable intersection of the $X^{(n)}$ one has the convergence
\begin{align*}  \xymatrixrowsep{4pc}\xymatrix{dens ( Supp(\Lambda ) + U_n) \, = \, m(Q_n) \, \,  \ar@{->}^-{n \, \rightarrow \, \infty}[r] &  \displaystyle \, \, 1  }   \
\end{align*}
which gives the proof.
\end{proof}

This proposition in particular apply for the system $(X_{\mathrm{p}}, \mathbb{G},m_{\mathrm{p}})$ associated by Theorem \ref{theo:decomposition} to any given ergodic system $(X, \mathbb{G}, m)$ of positive translation-bounded measures. Now returning to relative density, to complete the above Proposition we have the following result:


\vspace{0.2cm} 

\begin{prop}\label{prop:relatively.dense.support} Let $(X, \mathbb{G}, m)$ be an ergodic system of positive translation-bounded measures. Assume that almost any measure in $X$ have relatively dense support. Then almost any measure in $X_{\mathrm{p}}$ also have relatively dense support.
\end{prop}

\vspace{0.2cm} 


\begin{proof} Let $U \subset \mathbb{G}$ be a symmetric open subset with compact closure. Then the collection $X_U \subseteq X$ of translation-bounded measures having $U$-relatively dense support is a $\mathbb{G}$-stable Borel subset of $X$. For, it is obviously $\mathbb{G}$-stable and moreover, from an argument of the proof of Proposition \ref{prop:density},
\begin{align}\label{equivalence} \Lambda \in X_U \qquad \Longleftrightarrow \qquad  \Lambda (U+s) >0 \text{ for each } s\in \mathbb{G}
\end{align}
Now considering a continuous function $\phi \in \mathcal{C}_c(\mathbb{G})$ nowhere vanishing inside $U$ and null outside $U$ (such a function exists since $\mathbb{G}$ is metrizable), one gets that the condition $\Lambda (U+s) >0 $ for any $s\in \mathbb{G}$ is equivalent to have $\int \phi (.+s)\, d \Lambda >0 $ for any $s\in \mathbb{G}$. Now consider a sequence $(K_k)_{k\in \mathbb{N}}$ of compact sets in $\mathbb{G}$ whose union covers $\mathbb{G}$. Since $\int \phi (.+s)\Lambda$ is continuous in the variable $s$, our condition is in turns equivalent to have $inf _{s\in K_k} \int \phi (.+s) \, d\Lambda \geqslant \delta _k $ for some $\delta > 0$, for each $k\in \mathbb{N}$. Again by continuity in the variable $s$ its infimum over $s\in K_k$ can be set on a countable dense subset of $K_k$, independent on $\Lambda$, showing that $F_k (\Lambda) :=  inf _{s\in K_k} \int \phi (.+s) \, d\Lambda $ is the infimum of a countable collection of continuous functions on $X$ and thus is Borel on $X$. One concludes that $X_U$ is Borel by observing that 
\begin{align*} X_U = \bigcap _{k\in \mathbb{N}} F_k ^{-1}\left( ]0, + \infty [\right) 
\end{align*}
Now consider an increasing sequence of symmetric open subset $U_n \subset \mathbb{G}$ with compact closure and whose union covers $\mathbb{G}$. Since almost any measure in $X$ have relatively dense support it follows that the increasing countable union of Borel sets $X_{U_n}$ has measure 1. Therefore some of those must have non zero measure, which we simply denote $X_U$, and since it is a Borel $\mathbb{G}$-stable subset then ergodicity ensure that $m(X_U)=1$. Now the Disintegration Theorem ensures, when applied to the Borel indicator function $\mathbb{I}_{X_U}$, that the set $X_U'$ of $\Lambda \in X$ such that $X_U$ has $\mu _{\pi (\Lambda)}$-measure 1 is also of full measure in $X$. Let moreover $X_U''$ be the Borel subset of full measure in $X$ such that the equality $\Lambda _{\mathrm{p}} = \mathsf{i}(\mu _{\pi (\Lambda)})$ holds. Then for any $\mathsf{\Lambda}$ belonging simultaneously to these three sets one has
\begin{align*} \mathsf{\Lambda}_\mathrm{p}(U +t)  =  \int _X \Lambda (U+t) \, d \mu _{\pi (\mathsf{\Lambda})}(\Lambda) =  \int _{X_U} \Lambda (U+t) \, d \mu _{\pi (\mathsf{\Lambda})}(\Lambda)
\end{align*}
which is the integral of a strictly positive function (by (\ref{equivalence})) with respect to a probability measure, and thus must be strictly positive for any $t\in \mathbb{G}$. Therefore the image of such measures in $X_{\mathrm{p}}$ have $U$-relatively dense support, so that the measures belonging to $X_{\mathrm{p}}$ and having $U$-relatively dense support is of full $m_{\mathrm{p}}$-measure, as desired.
\end{proof}

\section{\textsf{Decomposition of weighted Meyer sets}}\label{Section:weighted.meyer.sets}

In this section we focus on a very particular type of translation-bounded measures. Very often, the translation-bounded measures considered in the literature are \textit{Dirac combs}, with support being, form the most particular to the most general, a subset of a lattice \cite{Baa, BaaBirGri}, a point set with an extra geometric property such as uniform discreteness, finite local complexity or the Meyer property \cite{BaaMoo0, BaaMoo, BaaZin, Str3}, or ultimately a possibly dense (yet countable) subset \cite{Ric, LenRic}. In the remaining part of this work, the translation-bounded measures we will consider are Dirac combs supported on a Meyer set, simply referred as \textit{weighted Meyer sets}. Our main result here is the claim that for an ergodic system of weighted Meyer sets $(X, \mathbb{G},m)$, the two associated systems $(X_{\mathrm{p}}, \mathbb{G},m_{\mathrm{p}})$ and $(X_{\mathrm{c}}, \mathbb{G},m_{\mathrm{c}})$ also consist of weighted Meyer sets.


\subsection{\textsf{Weighted FLC sets}} ~ 

\vspace{0.2cm} 

A point set $S$ of $\mathbb{G}$ is called uniformly discrete if there is an open set $U$ such that any of its translates by an element of $\mathbb{G}$ contains at most one element of $S$. If one considers its difference set $$S-S:= \left\lbrace t-t' \, : \, t, t' \in S\right\rbrace $$ then this means that $0$ in isolated in $S-S$. A uniformly discrete set is called of \textit{finite local complexity} (FLC) if the difference set $S-S$ is closed and all its points are isolated. In studying a particular FLC set $\mathsf{S}$ one usually consider a whole ensemble of FLC sets $X_{\mathsf{S}}$ called its \textit{hull}, stable under the natural $\mathbb{G}$-action shifting sets point by point (see Section 2 of \cite{Sch}). The hull $X_{\mathsf{S}}$ of a FLC set $\mathsf{S}$ is then a compact space with jointly continuous $\mathbb{G}$-action when equipped with a topology, the so-called \textit{local topology}, for which a neighborhood basis at each $S\in X_{\mathsf{S}}$ is yield by
\vspace{0.2cm}
\begin{align*}\mathcal{O}_{U,K}(S):=\left\lbrace  S' \in X_{\mathsf{S}} \, : \, \, \exists \; s\in U \text{ such that } S \cap K \equiv (S ' -s)\cap K \right\rbrace 
\end{align*}

\vspace{0.2cm}
for $0 \in U\subset \mathbb{G}$ open and $K$ compact (see \cite{Sch} for details and a proof). An important remark that we will further need is that whenever a set $S$ belongs to the hull $X_{\mathsf{S}}$ of some FLC set $\mathsf{S}$ then its difference set $S-S$ is included in $\mathsf{S}-\mathsf{S}$.

On the other hand a FLC set $S$ obviously defines a translation-bounded measure $\delta_{S}$ by setting a Dirac mass at any point of $S$. It is therefore natural to look at its associated dynamical system of translation-bounded measures $(X_{\delta _S}, \mathbb{G})$, and as discussed in \cite{BaaLen}, Section 4 therein, this system is conjugated with $(X_{S}, \mathbb{G})$ under the natural map associating a FLC set its corresponding Dirac comb. In the light of this correspondence, a natural generalization of FLC sets is the following:

\vspace{0.2cm} 

\begin{de} A weighted FLC set is a translation-bounded measure supported on a FLC set.
\end{de}

\vspace{0.2cm} 

A weighted FLC set is thus always of the form
\begin{equation*}
 \Lambda = \sum _{t \in S} c_t \delta _t \qquad \mathrm{with} \qquad    \left\{
    \begin{split}
0 \leqslant \vert c_t \vert \leqslant M\\ S \, \, \, \mathrm{ FLC } \, \, \, \mathrm{ set} 
    \end{split}
  \right.
\end{equation*}

Weighted lattices \cite{Baa} are important examples of such measures.

\vspace{0.2cm} 

\begin{prop} If a weighted FLC set $\mathsf{\Lambda}$ is supported on some FLC set $\mathsf{S}$ then any $\Lambda \in X_{\mathsf{\Lambda}}$ is supported on some $S\in X_{\mathsf{S}}$.
\end{prop}

\vspace{0.2cm} 

\begin{proof} Let $\Lambda \in X_{\mathsf{\Lambda}}$ be chosen. It is a vague limit of a sequence of translates $\mathsf{\Lambda}*\delta _{s_n}$ for $(s_n)_{n\in \mathbb{N}}\subset \mathbb{G}$, each respective translate having support in $\mathsf{S}-s_n$. Since $X_{\mathsf{S}}$ is compact for the topology described earlier the sequence $\mathsf{S}-s_n$ must accumulates at some FLC set $S$, and after possibly extracting one can suppose that $\mathsf{S}- s_n$ also converges to $S$ in $X_{\mathsf{S}}$. Select a decreasing sequence $(U_k)_k$ of neighborhoods of $0$ in $\mathbb{G}$ whose intersection  is the origin, and furthermore let $(K_k)_k$ be an increasing sequence of compact sets whose interiors cover $\mathbb{G}$. Thus one is able to extract a subsequence $(s_{n_k})_k$ of $(s_n)_n$ such that for each $k\in \mathbb{N}$ one has $ S\cap K_k \equiv (\mathsf{S}-s_{n_k}-\epsilon _k)\cap K_k $ for some $\epsilon _k \in  U_k$. Let $s_k':= s_{n_k}+ \epsilon _k$: Then $\mathsf{\Lambda}*\delta _{s_k'}$ converges vaguely to $\Lambda$, each having support in $\mathsf{S}-s_k'$ where $S\cap K_k \equiv (\mathsf{S}-s_k')\cap K_k$ for all integer $k\in \mathbb{N}$. Therefore on each compact $K_{k_0}$ the restriction of $\mathsf{\Lambda}*\delta _{s_k'}$ on $K_{k_0}$ is eventually supported on the finite set $S\cap K_{k_0}$, and it follows that the restriction of $\Lambda$ on $K_{k_0}$ is supported on $S\cap K_{k_0}$. Therefore $\Lambda$ is a weighted Dirac comb supported on $S$, as desired.
\end{proof}

\vspace{0.2cm}
 
\begin{rem} Note that if $\mathsf{\Lambda}$ is a weighted FLC set with support the FLC set $S(\mathsf{\Lambda})$, then the mapping associating a $\Lambda \in X_{\mathsf{\Lambda}}$ its support is in general neither valued in $ X_{S(\mathsf{\Lambda})}$ nor continuous between these two spaces. One has in fact that it yields a well-defined and continuous map $X_{\mathsf{\Lambda}} \longrightarrow X_{S(\mathsf{\Lambda})}$ if and only if any non-zero coefficient of the Dirac comb $\mathsf{\Lambda}$ has absolute value bounded from below by some positive constant.\\
\end{rem}

Let $\mathsf{\Lambda}$ be a weighted FLC set. It will later be convenient to have at our disposal a "local topology" description of the topology of the hull $X_{\mathsf{\Lambda}}$. This is provided by considering the collection of subsets of $X_{\mathsf{\Lambda}}\times X_{\mathsf{\Lambda}}$ 
\vspace{0.2cm}
\begin{align*}\mathcal{O}_{U,K,\varepsilon}:=\left\lbrace  (\Lambda, \Lambda ') \in X_{\mathsf{\Lambda}}\times X_{\mathsf{\Lambda}} \, : \, \,  \inf _{s\in U} \left[  \sup _{v\in K}\left| \Lambda (v) - \Lambda ' (v+s)\right| \right]  < \varepsilon \right\rbrace 
\end{align*}

\vspace{0.2cm}
for $0 \in U\subset \mathbb{G}$ open, $K$ compact and $\varepsilon > 0$, which constitutes a basis of a uniformity on $X_{\mathsf{\Lambda}}$ (see \cite[Chapter 6]{Kelley} for more on uniformities). Here and all along this section we note $\Lambda (v)$ for the $\Lambda $-measure of the singleton $\lbrace v\rbrace$. This defines a unique topology $\mathcal{T}_{loc}$ which we shall call the local topology of $X_{\mathsf{\Lambda}}$, such that a neighborhood basis at any $\Lambda \in X_{\mathsf{\Lambda}}$ is provided by $\mathcal{O}_{U,K,\varepsilon}(\Lambda) :=\left\lbrace  \Lambda ' \in X_{\mathsf{\Lambda}} \, : \, \,  (\Lambda , \Lambda ') \in \mathcal{O}_{U,K,\varepsilon} \right\rbrace$ for $0 \in U\subset \mathbb{G}$ open, $K$ compact and $\varepsilon > 0$ \cite[Chapter 6, Theorem 5]{Kelley}. Since $\mathbb{G}$ is $\sigma$-compact and $1^{st}$ countable one easily extracts a countable family of sets $\mathcal{O}_{U,K, \varepsilon }$ forming a basis of the same uniformity, and by \cite[Chapter 6, Theorem 13]{Kelley} the local topology is consequently metrizable.
 
 \vspace{0.2cm}
 
\begin{prop}\label{prop:equivalence.topologies} Let $\mathsf{\Lambda}$ be a weighted FLC set. Then the vague topology coincides with the local topology on $X_{\mathsf{\Lambda}}$. In particular $(X_{\mathsf{\Lambda}},\mathcal{T}_{loc}) $ is compact.
\end{prop}

\vspace{0.2cm}



\begin{proof} We only need to show the continuity of the identity map from $(X_{\mathsf{\Lambda}},\mathcal{T}_{vague}) $ to $(X_{\mathsf{\Lambda}},\mathcal{T}_{loc}) $, and the result will follows from compacity of the former topology and Hausdorff property (which is straightforward to show since elements of $X_{\mathsf{\Lambda}}$ are atomic measures) of the latter. Since both topologies are metrizable it suffices to show that is a sequence $(\Lambda _n)_n$ converges vaguely to $\Lambda$ then it converges to $\Lambda$ in the local topology. Let thus a sequence $(\Lambda _n)_n$ converging vaguely to a limit $\Lambda$: To show that this sequence also converges to $\Lambda$ in the local topology is is sufficient to show that from any subsequence can be extracted another subsequence converging to $\Lambda$ in the local topology.

Denote by $\mathsf{S}$ the support of the weighted FLC set $\mathsf{\Lambda}$. Then let $(\Lambda _{n_k})_k$ be a subsequence, which still converges vaguely to $\Lambda$. Repeating the argument of the proof of Proposition \ref{} one can extract a subsequence $(\Lambda_{n_{k_l}})_l$ (still vaguely converging to $\Lambda$), as well a a sequence $(S_{k_l})_l$ in $X_{\mathsf{S}}$ with limit $S\in X_{\mathsf{S}}$, such that: $\Lambda_{n_{k_l}}$ is supported on $S_l$, $\Lambda$ is supported on $S$, and for a given decreasing sequence $(U_l)_l$ of neighborhoods of $0$ in $\mathbb{G}$ whose intersection is the origin and a given increasing sequence $(K_l)_l$ of compact sets whose interiors cover $\mathbb{G}$ one has for each $l\in \mathbb{N}$ that $$ S\cap K_l \equiv (S_{k_l}-\epsilon _l)\cap K_l \; \; \text{ for some } \; \; \epsilon _l \in  U_l$$
We claim now that the subsubsequence $(\Lambda_{n_{k_l}})_l$ converges to $\Lambda$ in the local topology: Indeed for any $0 \in U\subset \mathbb{G}$ open, $K$ compact and $\varepsilon > 0$ there exists an integer $L$ such that whenever $l\geqslant L$ then $$ S\cap K \equiv (S_{k_l}-s _l)\cap K \; \; \text{ for some } \; \; s _l \in  U_l \subseteq U$$

Therefore for $l\geqslant L$ each $\Lambda _{n_{k_l}}$ admits a $s _l \in U$ such that $\Lambda$ and $\Lambda _{n_{k_l}} * \delta _{s_l}$ are all, when restricted to the compact set $K$, supported on the finite set $S\cap K$, that is to say, such that 
\begin{align*}\left| \Lambda (v) - \Lambda _{n_{k_l}} (v-s_l)\right| =\left| \Lambda (v) - \Lambda _{n_{k_l}} * \delta _{s_l} (v)\right| = 0  < \varepsilon \qquad \forall \; v\in K \setminus S
\end{align*}

Hence we will be done once we find an integer $L'\geqslant L$ such that for $l\geqslant L'$ one also has $\vert \Lambda (p) - \Lambda _{n_{k_l}} (p- s_l)\vert  < \varepsilon$ at points $p \in S\cap K$. To do so consider an open set $U'$ such that $(S\cap K - S\cap K) \cap U' = \left\lbrace 0\right\rbrace $, and moreover consider an open set $0 \in U''\subset \mathbb{G}$ whose closure is strictly included in $U'$: Then one can exhibit a compactly supported continuous function $\phi$ identically equal to $1$ on $U''$ and vanishing outside $U'$, and we thus let $L'\geqslant L$ be such that 
\begin{align*} -s_l \in U'' \quad \text{ and } \quad  \left| \int \phi * \delta _p \, d\Lambda -  \int \phi * \delta _p \, d\Lambda _{n_{k_l}} \right|  < \varepsilon \qquad \forall \; p \in S \cap K \; \text{ and } \; l \geqslant L'
\end{align*}

which exists since $(s_l)_l$ converges to $0$ and $(\Lambda _{n_{k_l}} )_l$ converges vaguely to $\Lambda$. From our particular choice of $\phi$, for such an $L'$ one has for any $p \in S \cap K$ that
\begin{align*} \left| \Lambda (p) - \Lambda _{n_{k_l}} (p-s_l)\right| =   \left| \int \phi * \delta _p \, d\Lambda -  \int \phi * \delta _p \, d\Lambda _{n_{k_l}} \right|  < \varepsilon
\end{align*}

This shows that the subsubsequence $(\Lambda _{n_{k_l}})_l$ converges to $\Lambda$ in the local topology, which finishes the proof.
\end{proof}

\subsection{\textsf{Weighted Meyer sets and Cut $\& $ Project representations}} ~ 

\vspace{0.2cm}

A strengthening of the FLC property for point sets is to require two additional facts: First the uniformly discrete set $S$ should be relatively dense, meaning that there exists a compact set $K$ such $S+K$ covers the whole group $\mathbb{G}$ (we call such set a \textit{Delone set}), and second the difference set $S-S$ should itself be uniformly discrete. A uniformly discrete obeying these two additional properties is called a \textit{Meyer set}. Meyer sets have several different by equivalent characterizations, see \cite{Moo} as well as \cite{Lag, Lag2} for the case $\mathbb{G}= \mathbb{R}^d$ and the improvement \cite{Str3} for the general case. If $\mathsf{S}$ is a Meyer set then as an FLC set it has a hull $X_{\mathsf{S}}$, and it is not hard to show that any $S\in X_{\mathsf{S}}$ is also a Meyer set.\\

We now present the widely studied \textit{Cut $\&$ Project formalism} for point sets, see the different works \cite{Sch, HucRic}. Suppose we are given a triple $(H, \Gamma, s_H)$ where $H$ is a LCA group, $\Gamma$ a finitely generated subgroup of $\mathbb{G}$ and a group morphism $ \xymatrixcolsep{3pc}\xymatrix{ s_H: \Gamma \ar[r] & H  }$ with range $ s_H(\Gamma)$ dense in $H$, and whose graph $\mathcal{G}(s_H):= \left\lbrace (s_H(t),t)\in H \times \mathbb{G}\, : \, t\in \Gamma\right\rbrace$ is furthermore a lattice, that is, a discrete and co-compact subgroup of $H\times \mathbb{G}$. Such a triple is called a \textit{cut $\&$ project scheme} (CPS for short). The LCA group $H$ of a CPS is commonly called the internal space (or internal group), the subgroup $\Gamma$ of $\mathbb{G}$ the structure group of the CPS, and the morphism $s_H$ the *-map of the CPS. In addition,  compact topologically regular subset $W$ of $H$, that is, a compact set which is the closure of its interior $\mathring{W}$ in $H$, will be called a \textit{window}.

\vspace{0.2cm} 

\begin{de}\label{def:cut.and.project.representation} Let $S$ be a point set of $\mathbb{G}$. We call Cut $\&$ Project representation of $S$ any CPS $(H, \Gamma, s_H)$ such that $S$ belongs to $\Gamma$, and whose image $s_H(S)$ is relatively compact in $H$.
\end{de}

\vspace{0.2cm} 

Not all point sets admit a Cut $\&$ Project representation, see Theorem \ref{theo:existence.cut.project.representation} below. Whenever a point set $S$ belongs to $\Gamma$ for some CPS $(H, \Gamma, s_H)$ then it can thus be lifted in a subset of a lattice (namely $\mathcal{G}(s_H)$) in $H \times \mathbb{G}$, and the condition of relative compactness for $s_H(S)$ means that this lifting stand in a "not too thick" strip about $\mathbb{G}$ (when naturally embedded into $H \times \mathbb{G}$). Given some point set $S$, We will call a Cut $\&$ Project representation $(H, \Gamma, s_H)$ of $S$ \textit{irredundant} whenever the closure $W_S$ of $s_H(S)$ in $H$ is irredundant in $H$, that is, if there is no non-trivial element $w$ of $H$ satisfying $W_S + w= W_S$. It is always possible to turn a given Cut $\&$ Project representation into an irredundant one by simply modding out a certain compact subgroup of $H$, see \cite{BaaLenMoo, LeeMoo} for details and a proof. Thus a point set having a Cut $\&$ Project representation automatically also admits an irredundant Cut $\&$ Project representation.

\vspace{0.2cm} 

\begin{de}\label{defmodelset} A closed model set is a point set obtained from a CPS $(H, \Gamma, s_H)$ and a window $W$ in $H$ by
\begin{align*}\mathfrak{P}_{H}(W):= \left\lbrace t \in \Gamma \; : \; s_H(t ) \in W \right\rbrace 
\end{align*}

\end{de}

\vspace{0.2cm} 

It is a true fact that a closed model set is always a Meyer set. Moreover a closed model set $S$ always comes with a Cut $\&$ project representation, namely that of the CPS used to construct it, where $s_H(S)$ a compact closure $W$ in $H$.\\

\begin{theo}\label{theo:existence.cut.project.representation} Let $S$ be a point set of $\mathbb{G}$. The following assertions are equivalent:

\vspace{0.2cm}

$(i) \; S$ is a subset of a Meyer set,

$(ii) \; S$ admits a Cut $\&$ Project representation,

$(iii) \; S $ is a subset of a closed model set.\vspace{0.2cm}\\
If it holds true then the Cut $\&$ Project representation can be chosen irredundant.\\
\end{theo}

The hard part of this Theorem is $(i)\Rightarrow (ii)$, which is sufficient to prove when $S$ is itself a Meyer set, and which was proved by Meyer \cite[Chapter II, Section 5, Proposition 4]{Meyer}, later on followed by several works \cite{Lag, Moo, BaaMoo, Auj}. Assuming this let us then provide a short proof of the remaining statements:\\

\begin{proof} Assume $(i)$: Then $S$ is contained in some Meyer set $S'$, which by \cite{Str3}, Theorem ? admits a Cut $\&$ Project representation $(H, \Gamma, s_H)$. It follows that the closure $W_S$ of $s_H(S)$ is contained in the closure of $s_H(S')$ which is compact in $H$, and therefore $(H, \Gamma, s_H)$ is a Cut $\&$ Project representation of $S$, giving $(ii)$. Then assuming $(ii)$ one gives rise by considering any compact topologicaly regular subset $W$ of $H$ containing $s_{H}(S)$ to a closed model set $\mathfrak{P}_{H}(W):= \left\lbrace t \in \Gamma \; : \; s_{H}(t ) \in W \right\rbrace$ which contains $S$, yielding $(iii)$. As any closed model set is a Meyer set this immediately gives $(i)$.

Finally, following \cite{BaaLenMoo}, Section 9 therein, given a Cut $\&$ Project representation $(H, \Gamma, s_H)$ of $S$ with closure $W_S$ of $s_H(S)$ in $H$ one mods out the compact subgroup $K_S$ of elements $w$ of $H$ such that $W_S +w= W_S$, which results in a an irredundant representation $(H', \Gamma, s_{H'} )$ of $S$ with internal group $H'=H$ modulo $K_S$, $*-$map $s_{H'}= s_{H}$ modulo $K_S$ and $s_{H'}(S)$ having closure the compact subset $W_S'=W_S$ modulo $K_S$.
\end{proof}

\vspace{0.2cm}

\begin{de} A weighted Meyer set is a translation-bounded measure supported on a Meyer set.
\end{de}

\vspace{0.2cm} 

A weighted Meyer set is hence always of the form
\begin{equation*}
 \Lambda = \sum _{t \in S} c_t \delta _t \qquad \mathrm{with} \qquad    \left\{
    \begin{split}
0 \leqslant \vert c_t \vert \leqslant M\\ S \, \, \, \mathrm{ Meyer } \, \, \, \mathrm{ set} 
    \end{split}
  \right.
\end{equation*}

Obviously a weighted Meyer set is always a weighted FLC set, and its support admits \textit{by definition} a Meyer super-set containing it. A weighted Meyer set may fail to have a Meyer set support since it may not be relatively dense. Any subset (even the finite ones) of a Meyer set then yields a Dirac comb which is according to this definition a weighted Meyer set. A weighted Meyer set could have also been called \textit{weighted model set}, but since this terminology makes explicitly reference to a Cut $\&$ Project scheme representing it we preferred the former appellation.

\subsection{\textsf{The torus parametrization of a Cut $\& $ Project representation}} ~ 

\vspace{0.2cm} 

In \cite{BaaHerPle} the authors illustrated in a particular case a relation, which is of main importance here and more generally in the whole point sets theory, between Cut $\&$ Project representations of a Meyer set $S$ and its associated dynamical system $(X_S, \mathbb{G})$: They namely showed that existence of a Cut $\&$ Project representation of $S$ always gives rise to a Kronecker factor of the system $(X_S, \mathbb{G})$ (and thus to a group of eigenvalues for $(X_S, \mathbb{G})$). Such result has later on been stated and proved in its greatest generality by Schlottmann in \cite{Sch}, Section 4 therein. 

To be more precise, it is shown in \cite{Sch} that if a \textit{repetitive} Meyer set $S$ admits an irredundant Cut $\&$ Project representation $(H, \Gamma , s_H)$ then there is a (continuous) factor map from $X_S$ onto $\left[ H\times \mathbb{G}\right] _{\mathcal{G}(s_H)}  $, the compact Abelian group yield by the quotient of $ H\times \mathbb{G}$ by its subgroup $\mathcal{G}(s_H)$ and endowed with the $\mathbb{G}$-action "by rotation", set for $s\in \mathbb{G}$ on an element $[w,t]$ by $[w,t] .s := [w,t+s]$. Here repetitivity of a Meyer set means minimality of its dynamical system $(X_S, \mathbb{G})$. The factor map from $X_S$ onto the Kronecker factor issued from a Cut $\&$ Project representation of $S$ is called a \textit{torus parametrization}.

We propose here to set this result in a slightly more general situation, namely that of weighted Meyer sets, by carefully adapting the proof of \cite{Sch} to our setting. The original proof has been set with the assumption of repetitivity on the Meyer sets, and we shall not assume this here. As a result, the torus parametrization yield by a Cut $\&$ Project representation of a weighted Meyer set $\mathsf{\Lambda}$ will no longer be defined on the entire hull $X_{\mathsf{\Lambda}}$ but rather only on the rubber local isomorphism class RLI($\mathsf{\Lambda}$) (one will notices that repetitivity of $\mathsf{\Lambda}$, that is to say, minimality of $(X_{\mathsf{\Lambda}}, \mathbb{G})$, precisely means that RLI($\mathsf{\Lambda})=X_{\mathsf{\Lambda}}$). This will however be sufficient to prove the main result of this section, namely Theorem \ref{theo:decomposition.weighted.meyer.sets} in next paragraph (Theorem 4 in the introductory part).\\

\begin{theo}\label{theo:parametrization} Let $\mathsf{\Lambda}$ be a weighted Meyer set with support $S(\mathsf{\Lambda})$ having an irredundant Cut $\& $ Project representation $(H, \Gamma, s_H)$, with $W$ the closure of $s_H(S(\mathsf{\Lambda}))$ in $H$.\\
\\
$(i)$ For each $\Lambda \in X_{\mathsf{\Lambda}}$ there exists $ (w,t) \in H\times \mathbb{G}$ such that $S(\Lambda) \subseteq \mathfrak{P}_H(W +w)+t$,\\
\\
$(ii)$ When $\Lambda \in \mathrm{RLI}(\mathsf{\Lambda})$ the element $(w,t) \in H\times \mathbb{G}$ is unique modulo $\mathcal{G}(s_H)$,\\
\\
$(iii)$ The mapping $\xymatrixcolsep{2pc}\xymatrix{ \mathrm{RLI}(\mathsf{\Lambda}) \ni \Lambda \, \ar@{|->}[r] & \, [w,t] \in \left[ H\times \mathbb{G}\right] _{_{\mathcal{G}(s_H)}}  }$ is a continuous $\mathbb{G}$-map.\\
\end{theo}

\begin{proof} $(i):$ Let $\Lambda \in X_{\mathsf{\Lambda}}$, which we can assume to not be the trivial measure on $\mathbb{G}$, and supposed to be supported on the structure group $\Gamma$. For such $\Lambda$ denote by $W(\Lambda)$ the closure of $s_H(S(\Lambda))$ in $H$. Hence with respect to this notation $W$ is nothing but $W(\mathsf{\Lambda })$. Then it is not hard to show the equivalence of conditions, for $w\in H$,
\begin{align}\label{intersection} w\in  \bigcap _{\gamma \in S(\Lambda ) } s_H(\gamma )-W \qquad \Longleftrightarrow \qquad  W(\Lambda)\subseteq W+ w .
\end{align}

Let us show that any $\Lambda \in X_{\mathsf{\Lambda}}$ supported on $\Gamma$ admits an element $w_\Lambda \in H$ where these equivalent conditions hold: As $\Lambda \in X_{\mathsf{\Lambda}}$ one can find a sequence $(t _n)_n$ of elements in $\mathbb{G}$ such that $\mathsf{\Lambda}* \delta _{t_n}$ converges to $\Lambda$, which according to Proposition \ref{prop:equivalence.topologies} means that for each compact $K$ and $\varepsilon >0$ there is a $N_{K, \varepsilon }$ such that, after possibly slightly moving each $t_n$, one has for each $n\geqslant N_{K, \varepsilon }$ that $\vert \Lambda (v) - \mathsf{\Lambda}* \delta _{t_n} (v)\vert  < \varepsilon $ for any $v\in K$. Therefore, whenever $p \in S(\Lambda )$ then $\mathsf{\Lambda}* \delta _{t_n}(p)$ is eventually non-zero, so that $p$ eventually belongs to the support $S(\mathsf{\Lambda})+t_n$ of $\mathsf{\Lambda}* \delta _{t_n}$. Since $\Lambda$ is by assumption non-trivial one can pick up some $p_0\in S(\Lambda)$. Then the sequence $t_n$ eventually lies in $p_0-S(\mathsf{\Lambda}) \subset S(\Lambda)- S(\mathsf{\Lambda})$, which lies in the group $\Gamma$ since both $\Lambda$ and $\mathsf{\Lambda}$ are supported on this latter group. Thus $s_H(t_n)$ eventually makes sense, and since $p_0 \in S(\mathsf{\Lambda})+t_n$ eventually then one eventually gets $s_H(p_0 ) \in W+s_H(t _n)$. Thus the sequence $s_H(t _n)$ lies in the compact set $s_H(p_0 )- W$ eventually, and therefore accumulates at some element $w_\Lambda\in H$. We can suppose after possibly extracting a subsequence that $s_H(t _n)$ converges to $w_\Lambda$ in $H$. The latter must satisfy $s_H(p_0) \in W+w_\Lambda$, and since this argument is independent upon the choice of $p_0 \in S(\Lambda )$ we deduce that $s_H(S(\Lambda )) \subset W+w_\Lambda$, which yields $W(\Lambda ) \subset W+w_\Lambda$ as desired. Now given any $\Lambda \in X_{\mathsf{\Lambda}}$ one has for any chosen element $t\in S(\Lambda)$ that $\Lambda * \delta _t$ is supported on the structure group $\Gamma$, and from we just said there exists a $w\in H$ such that (\ref{intersection}) holds, yielding a $ (w,t) \in H\times \mathbb{G}$ with $S(\Lambda) \subseteq \mathfrak{P}(W +w)+t$, as desired.\\

$(ii):$ We show that when $\Lambda\in \mathrm{RLI}(\mathsf{\Lambda})$ with support in $\Gamma$ then the $w\in H$ satisfying (\ref{intersection}) is unique. To that end we know from the above analysis that the set $W(\Lambda)$ is contained in some translate $W+w$ of the compact set $W$ and thus is compact in $H$. Since $\mathsf{\Lambda} \in  \mathrm{RLI}(\Lambda)$ we can interchange the roles of $\Lambda $ and $\mathsf{\Lambda}$ in the previous argument shows that there equally exists some $w '\in H$ with $W\subseteq W(\Lambda)+ w '$. It follows that $W \subseteq W(\Lambda) +w ' \subseteq W +w + w '$ with $W$ compact, which forces $W= W+ w+ w '$. From the irredundancy assumption of $W$ one gets $w= -w '$, and this in turns gives that $W(\Lambda)= W+ w $. Such an equality can hold for at most one element $w$, giving unicity. Now suppose $\Lambda\in \mathrm{RLI}(\mathsf{\Lambda})$ admits $(w,t)$ and $(w',t')$ in $H\times \mathbb{G}$ for which inclusion set in the first point of the statement hold: 
\begin{align*} S(\Lambda) \subseteq \mathfrak{P}(W +w)+t \qquad \text{and} \qquad S(\Lambda) \subseteq \mathfrak{P}(W +w')+t'
\end{align*}

Then from what have been just said one has
\begin{align*} W(\Lambda -t) =  W +w \qquad \text{and} \qquad W(\Lambda -t') =  W +w'
\end{align*}

Given any $p \in S(\Lambda)$ one has
\begin{align*}t-t' = (p -t')-(p -t) \in  \mathfrak{P}_H(W +w') -\mathfrak{P}_H(W +w) \subseteq \Gamma - \Gamma = \Gamma
\end{align*}

Therefore $s_H(t-t') $ makes sense and is an element of $H$ such that  
\begin{align*} W+ w +s_H(t-t')=  W(\Lambda -t) + s_H(t-t') = W(\Lambda -t') =  W +w' 
\end{align*}

which shows by irredundancy that $w' = w +s_H(t-t')$. Thus the element $(w,t) \in H\times \mathbb{G}$ for which inclusion if point $(i)$ of the statement holds is unique modulo $\mathcal{G}(s_H)$, as desired.\\

$(iii):$ The mapping $\xymatrixcolsep{2pc}\xymatrix{ \mathrm{RLI}(\mathsf{\Lambda}) \ni \Lambda \, \ar@{|->}[r] & \, [w,t] \in \left[ H\times \mathbb{G}\right] _{\mathcal{G}(s_H)}  }$ is from the conclusion of point $(ii)$ well-defined. Moreover when $\Lambda \in \mathrm{RLI}(\mathsf{\Lambda})$ is such that $S(\Lambda) \subseteq \mathfrak{P}(W +w)+t$ then $S(\Lambda *\delta _s) \subseteq \mathfrak{P}(W +w)+t+s$ for any $s\in \mathbb{G}$, so by the unicity part of point $(ii)$ one deduces that the map is a $\mathbb{G}$-map. We show continuity at any given $\Lambda \in \mathrm{RLI}(\mathsf{\Lambda}) $. Let $(w,t)\in H\times \mathbb{G}$ such that $S(\Lambda) \subseteq \mathfrak{P}(W +w)+t$: Given a neighborhood $U$ of $0$ in $H$ and a neighborhood $U'$ of $0$ in $\mathbb{G}$, since $w$ satisfies
\begin{align}  \lbrace w \rbrace =  \bigcap _{p \in S(\Lambda )-t } s_H(p )-W 
\end{align}

there exists a sufficiently large compact set $K$ such that
\begin{align*}\bigcap _{p \in (S(\Lambda ) -t)\cap K} s_H(p) -W  \, \subseteq \,  w +U.
\end{align*}

Now the set $S(\Lambda )\cap (K+t)$ being finite one has a minimal value $M>0$ for $\vert \Lambda(v)\vert$ on that set. Then we are done if we can show that if $ \Lambda '\in  \mathrm{RLI}(\mathsf{\Lambda})$ is such that
\begin{align*}  \inf _{s\in U'}\left[  \sup _{v\in K+t}\vert \Lambda (v) - \Lambda '(v+s)\vert \right]  < \frac{M}{2} 
\end{align*}

(that is, if $ \Lambda '\in  \mathrm{RLI}(\mathsf{\Lambda})$ belongs to a neighborhood of $\Lambda$ uniquely defined by $\Lambda$, $U$ and $U'$) then it admits a representative $(w',t') \in H\times \mathbb{G}$ of its class such that 
\begin{align*} (w',t') \in (w,t)+ U\times U'
\end{align*}

This is in turns true since for such $\Lambda '$ one has an $s\in U'$ such that, letting $t':= t+ s$, $(S(\Lambda ) -t)\cap K$ is included in $(S(\Lambda ')-t')\cap K$, and thus
\begin{align*} w' \in  \bigcap _{p \in (S(\Lambda ')-t')\cap K} s_H(p) -W  \, \subseteq  \bigcap _{p \in (S(\Lambda ) -t)\cap K} s_H(p) -W  \, \subseteq \,  w +U
\end{align*}

where $w'$ is the unique element of $H$ in the intersection $ \bigcap _{p \in (S(\Lambda ')-t') } s_H(p )-W$. We therefore end up with a pair $(w',t')$ such that, from the selection of $w'$ satisfies 
\begin{align*} S(\Lambda') \subseteq \mathfrak{P}(W +w')+t'
\end{align*}

and thus is a representative in $H\times \mathbb{G}$ of the class of $\Lambda '$ in $\left[ H\times \mathbb{G}\right] _{\mathcal{G}(s_H)}$, such that $(w',t') \in (w,t)+ U\times U'$. This shows continuity at any $\Lambda \in \mathrm{RLI}(\mathsf{\Lambda}) $, as desired.
\end{proof}



\subsection{\textsf{The decomposition of weighted Meyer sets}} ~ 

\vspace{0.2cm} 

We shall show here our main result of this section, see Theorem \ref{theo:decomposition.weighted.meyer.sets} just below. We will need to consider here the Borel $\mathbb{G}$-maps yield by Theorem \ref{theo:decomposition}. These maps are not uniquely defined but by the uniqueness statement of Theorem \ref{theo:uniqueness.decomposition} any two choices agree $m$-almost everywhere on $X$. We assume here that such choice is made.\\

\begin{theo}\label{theo:decomposition.weighted.meyer.sets} Let $(X, \mathbb{G}, m)$ be an ergodic system of weighted Meyer sets. Then the following property holds true $m$-almost everywhere: Whenever $\Lambda$ is supported on a closed model set, then so are its summands $\Lambda _{\mathrm{p}}$ and $\Lambda _{\mathrm{c}}$.\\
\end{theo}

\begin{proof} Let $(X, \mathbb{G}, m)$ be an ergodic system of weighted Meyer sets with eigenvalue group $\mathcal{E}$, and let $(X_{\mathrm{p}}, \mathbb{G},m_{\mathrm{p}})$ and $(X_{\mathrm{c}}, \mathbb{G},m_{\mathrm{c}})$ be as in Theorem \ref{theo:decomposition}. Recall from Theorem \ref{theo:uniqueness.decomposition} that there exists a unique (up to almost everywhere equality) Borel factor map $\xymatrixcolsep{2pc}\xymatrix{  X  \ar[r] &  X_{\mathrm{p}} }$, and from the proof of Theorem \ref{theo:decomposition} that it is almost everywhere given by the composition of the three Borel $\mathbb{G}$-maps
\begin{align*}\xymatrixcolsep{3.5pc}\xymatrix{  X \, \ar[r]^-{\pi} & \, \mathrm{T}_{ \mathcal{E} } \, \ar[r]^-{\mu} &  \,  \mathcal{M}^{b}(X) \, \ar[r]^-{\mathsf{i}} & \, \mathcal{M}^{\infty}(\mathbb{G})  }
\end{align*}

\vspace{0.2cm}
for some choice of Borel factor map $\pi$ as in Proposition \ref{prop.construction.Borel.factor.map}, and $\mu$ the disintegration of $m$ over $\mathrm{T}_{ \mathcal{E} }$. The following proposition is at the core of our proof:

\vspace{0.2cm} 

\begin{prop}\label{prop:decomposition.weighted.meyer.sets} Almost any $\mathsf{\Lambda} \in X$ admits a Borel set $F_{\mathsf{\Lambda}}$ of full $\mu _{\pi (\mathsf{\Lambda})}$-measure such that, whenever $\mathsf{\Lambda}$ is supported on a closed model set $\Delta$, then so are any $\Lambda \in F_{\mathsf{\Lambda}}$.
\end{prop}

\vspace{0.2cm} 

\begin{proof} As $(X, \mathbb{G}, m)$ is a ergodic system the set
\vspace{0.1cm}
\begin{align*}X^{(0)}:= \left\lbrace \Lambda \in X \, \vert \, \, X_\Lambda = Supp(m)\right\rbrace 
\end{align*}

\vspace{0.1cm}
is a Borel set of full measure in $X$, and choosing any $\Lambda \in X^{(0)}$ yield $X^{(0)} = \mathrm{RLI}(\Lambda)$. Let $\Lambda \in X^{(0)}$ be now given. By construction of the compact Abelian group $\mathrm{T}_{ \mathcal{E} }$, each eigenvalue $\omega \in \mathcal{E}$ has an associated character $\chi_\omega $ on it, and thus the Borel map $\Phi _\omega := \chi_\omega \circ \pi$ is an eigenfunction with eigenvalue $\omega$ on $(X, \mathbb{G}, m)$, forming so a collection of Borel maps $ \Phi _\omega$ on $X$, with $\omega \in \mathcal{E}$. Now consider an irredundant Cut $\&$ Project representation of $\Lambda$ in a CPS $(H,\Gamma, s_H)$: By part $(iii)$ of Theorem \ref{theo:parametrization} there is a continuous $\mathbb{G}$-map $p$ from $X^{(0)} = \mathrm{RLI}(\Lambda)$ to $\left[ H\times \mathbb{G}\right] _{_{\mathcal{G}(s_H)}}$. An eigenvalue $\omega$ of the Kronecker action of $\mathbb{G}$ on $\left[ H\times \mathbb{G}\right] _{_{\mathcal{G}(s_H)}}$ corresponds to a continuous character $\xi _\omega$ on this latter, which lifts in a continuous function
\vspace{0.1cm}
\begin{align*}\xymatrixcolsep{2.5pc}\xymatrix{ \Phi _{\Lambda ,\omega}:= \xi _\omega \circ p:  X^{(0)} \, \ar[r] & \, \mathrm{T}_1 }  \text{ such that }  \Phi _{\Lambda ,\omega}(\Lambda ) = 1, \, \, \Phi _{\Lambda ,\omega}(\Lambda '.t) = \omega (t) \Phi _{\Lambda ,\omega}(\Lambda ')
\end{align*}

\vspace{0.1cm}
In particular since $X^{(0)}$ is of full measure $\Phi _{\Lambda ,\omega}$ is an eigenvalue of $(X, \mathbb{G}, m)$ (perhaps not continuously extendable on $X$) with eigenfunction $\omega$. Therefore there is one and only one function $\widetilde{\Phi }_\omega$ on $X^{(0)}$, continuous and an eigenfunction for $\omega$, such that\\

$\bullet \, \, \widetilde{\Phi }_\omega = \Phi _\omega$ almost everywhere on $X^{(0)}$,
\vspace{0.2cm}

$\bullet \, \, \widetilde{\Phi }_\omega = c.\Phi _{\Lambda ,\omega}$ everywhere on $X^{(0)}$, for some constant $c\in \mathrm{T}_1$.\\

Such a function $\widetilde{\Phi }_\omega$ does not depend on the used irredundant Cut $\&$ Project representation of $\Lambda$ because $\Phi _{\Lambda ,\omega}$ is continuous on $X^{(0)}$ and the $\mathbb{G}$-orbit of $\Lambda$ is dense in it. Moreover if the same element $\omega \in \mathcal{E}$ has an eigenfunction of the form $\Phi _{\Lambda ' ,\omega}$ for another $ \Lambda '\in X^{(0)}$ then again by continuity there is a constant $c\in \mathrm{T}_1$ such that $\Phi _{\Lambda ' ,\omega} = c.\Phi _{\Lambda ,\omega}$ everywhere on $X^{(0)}$. This shows that $\widetilde{\Phi }_\omega$, whenever it exists, is independent on the choice of $\Lambda \in X^{(0)}$ and on the irredundant Cut $\&$ Project representation of $\Lambda$. Let us denote $\mathcal{E}_0$ the subset of elements in $ \mathcal{E}$ having an associated function $\widetilde{\Phi }_\omega$ arising in the above way. Since the eigenvalue group $\mathcal{E}$ is countable then $\mathcal{E}_0$ is countable, and therefore the following Borel set has full mesure in $X$:
\vspace{0.1cm}
\begin{align*}X^{(1)}:= \left\lbrace \Lambda \in X^{(0)} \, \vert \, \, \widetilde{\Phi }_\omega (\Lambda) = \Phi _\omega (\Lambda) \, \, \forall \, \omega \in \mathcal{E}_0 \right\rbrace 
\end{align*}

\vspace{0.1cm}
Now, as $\mu$ is a disintegration of $m$ over the Borel factor $ \mathrm{T}_{ \mathcal{E} }$ one has by an application of the Disintegration Theorem the Borel subset of full measure in $X$
\vspace{0.1cm}
\begin{align*}X^{(2)}:= \left\lbrace \Lambda \in X \, \vert \, \,\mu _{\pi(\Lambda)} \text{ is supported on } \pi^{-1}(\pi(\Lambda)) \right\rbrace 
\end{align*}

\vspace{0.1cm}
and as a result the Borel set $X^{(3)}:=  X^{(1)} \cap X^{(2)}$ is of full measure in $X$. Moreover, applying the Disintegration Theorem to the Borel almost everywhere $1$ function $\mathbb{I}_{X^{(3)}}$ one deduces that the set
\vspace{0.1cm}
\begin{align*}X^{\infty}:= \left\lbrace \Lambda \in X^{(3)} \, \vert \, \, \mu _{\pi(\Lambda)}(X^{(3)})=1\right\rbrace 
\end{align*}

\vspace{0.1cm}
is a Borel set of full measure in $X$. Form the very construction of this Borel set, for any $\Lambda \in X^{\infty}$ the Borel set
\vspace{0.1cm}
\begin{align*} F_\Lambda := \pi ^{-1}(\pi(\Lambda)) \cap X^{(3)}
\end{align*}

\vspace{0.2cm}
has $\mu _{\pi (\Lambda )}$-measure equal to 1 and always contains the element $\Lambda$. We now claim that for any $\Lambda \in X^{\infty}$ the sets $F_\Lambda$ make our statement holding. For, let $\mathsf{\Lambda} \in X^{\infty}$ be given, and suppose it is supported on some closed model set $\Delta$, that is, let $(H, \Gamma , s_H)$ be a CPS and $V$ a compact set in $H$ such that $\mathsf{\Lambda}$ is supported on $\Delta:= \mathfrak{P}_H(V)$. Then $s_H(S(\mathsf{\Lambda}))$ makes sense and its closure $W$ in $H$ is contained in $V$, and thus is compact. Modding out the redundancy subgroup $\mathcal{R}_W$ in $H$ leads to an irredundant Cut $\&$ Project representation of $\mathsf{\Lambda}$ in a new CPS $(H', \Gamma , s_{H'})$ with same structure group, with closure of $s_{H'}(S(\mathsf{\Lambda}))$ in $H'$ being a compact set $W'$, and such that
\vspace{0.1cm}
\begin{align}\label{inclusion.model.sets} \mathfrak{P}_{H'}(W')= \mathfrak{P}_{H}(W) \subseteq \mathfrak{P}_H(V)=: \Delta
\end{align}

\vspace{0.1cm}
It is then obviously sufficient to show that
\vspace{0.1cm}
\begin{align}\label{desired.inclusion} S(\Lambda) \subseteq \mathfrak{P}_{H'}(W') \qquad \text{for any } \Lambda \in F_{\mathsf{\Lambda}}
\end{align}

\vspace{0.2cm}
Let us show this:
From Theorem \ref{theo:parametrization} there is a continuous map $$\xymatrixcolsep{3.5pc}\xymatrix{X^{\infty} \subseteq X^{(0)} = \mathrm{RLI}(\mathsf{\Lambda}) \ni \Lambda  \, \ar@{|->}[r]^-{p} & \, [w,t] \in \left[ H'\times \mathbb{G}\right] _{\mathcal{G}(s_{H'})}  }$$ where $p(\Lambda)=[w,t]$ is the unique $\mathcal{G}(s_H)$-class such that any representative $(w,t)\in H\times \mathbb{G}$ satisfies
\vspace{0.1cm}
\begin{align*}S(\Lambda) \subseteq \mathfrak{P}_{H'}(W'+w)+t
\end{align*}

\vspace{0.2cm}
Denote $\mathcal{E}'$ the eigenvalue group of the Kronecker action of $\mathbb{G}$ on $\left[ H'\times \mathbb{G}\right] _{\mathcal{G}(s_{H'})}$. Then $\mathcal{E}' \subseteq \mathcal{E}_0$ naturally, and each $\omega \in \mathcal{E}'$ yields functions $\Phi _\omega$, $\widetilde{\Phi }_\omega$ and $\widetilde{\Phi }_{\mathsf{\Lambda}, \omega}$ as done before. Now for $\Lambda \in F_{\mathsf{\Lambda}}$ one has $\mathsf{\Lambda},\Lambda \in \pi ^{-1}(\pi(\mathsf{\Lambda}))$ so satisfy $\Phi _\omega (\Lambda)= \Phi _\omega (\mathsf{\Lambda})$, and moreover $\mathsf{\Lambda},\Lambda \in X^{\infty} \subseteq X^{(1)}$ so satisfy as well $\widetilde{\Phi }_\omega (\Lambda) = \widetilde{\Phi }_\omega (\mathsf{\Lambda})$. It therefore comes that $\widetilde{\Phi }_{\mathsf{\Lambda}, \omega} (\Lambda) = \widetilde{\Phi }_{\mathsf{\Lambda}, \omega} (\mathsf{\Lambda})$ for any $ \omega \in \mathcal{E}'$, that is, $\mathsf{\Lambda}$ and $\Lambda$ are identified under any function of the form $\chi _\omega \circ p$ where $\chi _\omega$ is any character on the compact Abelian group $\left[ H'\times \mathbb{G}\right] _{\mathcal{G}(s_{H'})}  $. One must then have $p(\Lambda) = p(\mathsf{\Lambda}) = [0,0]$ in $\left[ H'\times \mathbb{G}\right] _{_{\mathcal{G}(s_{H'})}}  $. This precisely ensures that $S(\Lambda) \subseteq \mathfrak{P}_{H'}(W')$ for any $\Lambda \in F_{\mathsf{\Lambda}}$, yielding (\ref{desired.inclusion}) and thus the proof of Proposition \ref{prop:decomposition.weighted.meyer.sets}.
\end{proof}

From this we can easily settle the proof of Theorem \ref{theo:decomposition.weighted.meyer.sets}: Indeed let $\Lambda \in X$ such that the conclusion of Proposition \ref{prop:decomposition.weighted.meyer.sets} holds true for some Borel set $F_\Lambda$. Without restriction one can suppose $\Lambda $ to belong in Borel subset of full measure 
\vspace{0.1cm}
\begin{align*} \left\lbrace \Lambda \in X \, \vert \, \,\Lambda _{\mathrm{p}} = \mathsf{i}(\mu _{\pi(\Lambda)}) \text{ and } \Lambda _{\mathrm{c}} = \Lambda - \Lambda _{\mathrm{p}} \right\rbrace 
\end{align*}

\vspace{0.1cm}
We shall show that the supports of $\mathsf{\Lambda}_{\mathrm{p}}$ and $\mathsf{\Lambda}_{\mathrm{c}}$ satisfy
\vspace{0.1cm}
\begin{align}\label{support} S(\mathsf{\Lambda}_{\mathrm{p}}), S(\mathsf{\Lambda}_{\mathrm{c}}) \, \subseteq \bigcup _{\Lambda \in F_\mathsf{\Lambda} } S(\Lambda)
\end{align}

\vspace{0.1cm}
Indeed for any Borel set $B$ of $ \mathbb{G}$ one has $\mathsf{\Lambda}_{\mathrm{p}}(B) = \mathsf{i}(\mu _{\pi(\mathsf{\Lambda})})(B)$, and since $F_\mathsf{\Lambda} $ is of full $\mu _{\pi (\mathsf{\Lambda})}$-measure one deduces by definition of $\mathsf{i}$
\vspace{0.1cm}
\begin{align*} \mathsf{\Lambda}_{\mathrm{p}}(B)  = \int _{X} \Lambda (B) \, d \mu _{\pi(\mathsf{\Lambda})}(\Lambda) = \int _{ F_\mathsf{\Lambda} } \Lambda (B) \, d \mu _{\pi(\mathsf{\Lambda})}(\Lambda)
\end{align*}

\vspace{0.1cm}
which is null when $B$ belongs to the complementary set of right term of (\ref{support}). This shows the desired inclusion for $S(\mathsf{\Lambda}_{\mathrm{p}})$. Now since $\mathsf{\Lambda}_{\mathrm{c}}= \mathsf{\Lambda} - \mathsf{\Lambda}_{\mathrm{p}}$, $S(\mathsf{\Lambda}_{\mathrm{c}})$ is therefore supported on $S(\mathsf{\Lambda}) \cup S(\mathsf{\Lambda}_{\mathrm{p}})$, clearly included in the union of (\ref{support}), given the desired inclusion for $\mathsf{\Lambda}_{\mathrm{c}}$. Therefore, whenever our $\mathsf{\Lambda} $ is supported on some closed model set $\Delta$ then combining the inclusions (\ref{support}) with the statement of Proposition \ref{prop:decomposition.weighted.meyer.sets} one deduces
\begin{align*} S(\mathsf{\Lambda}_{\mathrm{p}}), S(\mathsf{\Lambda}_{\mathrm{c}}) \subseteq  \left[ \bigcup _{\Lambda \in F _\mathsf{\Lambda} } S(\Lambda)\right] \subseteq \Delta
\end{align*}
as desired.
\end{proof}

\vspace{0.2cm} 

\begin{rem} From the above result we know that, given an ergodic system $(X, \mathbb{G}, m)$ of weighted Meyer sets with auto-correlation $\gamma$, then once a generic $\Lambda \in X$ is supported in a closed model set $\Delta $ then so does $\Lambda _{\mathrm{p}}$ and $\Lambda _{\mathrm{c}}$. Therefore using formula (\ref{formula.auto-correlation}) of Paragraph \ref{Par:diffraction} for the auto-correlation, as well as the conclusion of Remark \ref{remark:autocorrelation}, one shows that both $\gamma$, its strong almost periodic part (the auto-correlation of a generic choice of $\Lambda _{\mathrm{p}} \in X _{\mathrm{p}}$) and its null-weakly almost periodic part (the auto-correlation of a generic choice of $\Lambda _{\mathrm{c}} \in X _{\mathrm{c}}$) are supported on the difference set $\Delta -\Delta$. This statement was obtained (among other things) in \cite{Str2}.
\end{rem}

\begin{rem} A lattice $\Gamma$ of $\mathbb{G}$ is always a closed model set, for instance coming from the Cut $\&$ Project scheme having trivial internal group and $\Gamma$ itself as lattice in $\lbrace 0\rbrace  \times  \mathbb{G} $. As a result if $X$ consists of weighted translates of some lattice $\Gamma$ then so does $X_{\mathrm{p}}$ and $X _{\mathrm{c}}$.
\end{rem}

\begin{rem} Suppose that $(X, \mathbb{G}, m)$ is an ergodic system of Dirac combs of Meyer sets (thus with constant weight equal to 1). Let $\Xi _t$ stands for the subset of Dirac combs in $X$ having $t$ in its support, for any $t\in \mathbb{G}$. Then it is not difficult to show, for $m$-almost every $\Lambda \in X$, the formula
\begin{align*} \Lambda _{\mathrm{p}}= \sum \Lambda _{\mathrm{p}}(t) \delta _t \qquad \text{with} \qquad \Lambda _{\mathrm{p}}(t) = \mu _{\pi (\Lambda )}(\Xi _t)
\end{align*}

Thus, on a probabilistic viewpoint, $\Lambda _{\mathrm{p}}$ can be interpreted as the sum of Dirac combs of elements $t\in \mathbb{G}$ with each coefficient $\Lambda _{\mathrm{p}}(t)$ being the $m$-expectation of the point $t$ to appear, conditioned by the event of belonging in the fiber of $\Lambda$ over $\mathrm{T}_{ \mathcal{E} } $. In the general weighted case a similar formula, yet less trivial, can be stated which involves the weight function $\xymatrixcolsep{2pc}\xymatrix{ c:  X \, \ar[r] & \, \mathbb{C} }$, $c(\Lambda ):= \Lambda (0)$, the measure of the singleton $0$.
\end{rem}

\begin{rem} Suppose that $(X, \mathbb{G}, m)$ is an ergodic system of weighted Meyer sets with non-negative weights. Then by Remark \ref{remark:positive} any weighted Meyer set of $(X_{\mathrm{p}}, \mathbb{G},m_{\mathrm{p}})$ also has non-negative weights. Suppose moreover that $m$-almost any measure in $X$ have relatively dense support. Then it follows from Proposition \ref{prop:relatively.dense.support} that $m_{\mathrm{p}}$-almost any weighted Meyer set in $X_{\mathrm{p}}$ also have relatively dense support.
\end{rem}

\section{\textsf{Operator methods for the analysis of Borel factors}}\label{Section:admissible.operators}

In this section we will have a second look on the process associating to certain Borel factors of a measured dynamical system of translation-bounded measures $(X, \mathbb{G}, m)$ an operator on $L^2(X,m)$, as done in Paragraph \ref{paragraph:factor.to.operator}. Our main result here is the identification of the class of operators involved in this process, see Theorem \ref{theo:correspondence}, as well as a relation between such operators and the diffraction of the corresponding Borel factors.

\subsection{\textsf{Admissible operators}} ~ 

\vspace{0.2cm} 


%
Let $(X, \mathbb{G},m)$ be a measured dynamical system of translation-bounded measures, with diffraction to dynamic map
\begin{align*}\xymatrixcolsep{3.5pc}\xymatrix{\Theta  : L^2(\widehat{\mathbb{G}}, \widehat{\gamma} )\, \,  \ar@{<->}[r] & \, \, \mathcal{H}_{\Theta } \subseteq L^2(X,m)}, \qquad \Theta (\widehat{\phi}) = \mathcal{N}_\phi
\end{align*}

for any $\phi \in \mathcal{C}_c(\mathbb{G})$. Recall that $P_\Theta$ stands for the $\mathbb{G}$-commuting orthogonal projection onto the Hilbert subspace $\mathcal{H}_{\Theta }$.

\vspace{0.2cm} 

\begin{de}\label{def.admissibility} A bounded operator $Q$ on the Hilbert space $L^2(X, m)$ is admissible if it is $\mathbb{G}$-commuting, if $Q = Q P_\Theta$ on $L^2(X, m)$ and if for each compact set $K\subset \mathbb{G}$ there is a constant $R_K\geqslant 0$ such that
\vspace{0.1cm}
\begin{align*} \displaystyle \Vert Q( \mathcal{N}_\phi ) \Vert_{_\infty ,ess} \leqslant R_K \Vert \phi \Vert _{_\infty } \qquad \qquad \forall \, \phi \in \mathcal{C}_K(\mathbb{G})
\end{align*}
\end{de}

\vspace{0.2cm} 

It is obvious that the identity operator is admissible in the above sense, and that the collection of admissible operators of the system $(X, \mathbb{G}, m)$ forms a linear subspace of the space of bounded operators on $L^2(X, m)$. 

\vspace{0.2cm} 

\begin{rem} Any bounded operator on $L^2(X, m)$ admits a certain weak form of admissibility: If $Q$ is any bounded operator then due to the continuity of the map  $\phi \longrightarrow \mathcal{N}_\phi $ one always has, for any compact set $K\subset \mathbb{G}$, a certain constant $R'_K\geqslant 0$ such that $\displaystyle \Vert Q( \mathcal{N}_\phi ) \Vert_{_2} \leqslant R'_K \Vert \phi \Vert _{_\infty }$ for all $ \phi \in \mathcal{C}_K(\mathbb{G})$.\\
\end{rem}

\vspace{0.2cm} 

Suppose that we are given a Borel factor map $\pi : X \longrightarrow X_\pi$ over a measured dynamical system of translation-bounded measures $(X_{\pi}, \mathbb{G}, m_{\pi} )$ having diffraction $\widehat{\gamma}_{\pi}\ll \widehat{\gamma}$ with essentially bounded Radon-Nikodym differential $f_{\pi}:= \frac{d\widehat{\gamma}_{\pi}}{d\widehat{\gamma}} \in L^{\infty}(\widehat{\mathbb{G}}, \widehat{\gamma})$. Then Proposition \ref{prop:factor.to.operator} gives rise to a bounded operator $Q_\pi $ on $L^2(X,m)$, which we aim here to show that it is admissible:

\vspace{0.2cm} 

\begin{prop} The operator $Q_\pi \in \mathcal{B}(L^2(X,m))$ associated by Proposition \ref{prop:factor.to.operator} to any Borel factor map $\pi : X \longrightarrow X_\pi$ as above is admissible.
\end{prop}

\vspace{0.2cm} 

\begin{proof} One has $Q_\pi = Q_\pi P_\Theta$ by construction, whereas $\mathbb{G}$-commutativity, that is to say, $U_t Q_\pi = Q_\pi U_t$ for any $t\in \mathbb{G}$, is straightforward to check on functions of the form $\mathcal{N}_\phi$ for $\phi \in \mathcal{C}_c(\mathbb{G})$ and therefore holds on the closure $\mathcal{H}_\Theta$, hence on the whole space $L^2(X,m)$. For each $\phi \in \mathcal{C}_K(\mathbb{G})$ supported on a given compact set $K\subset \mathbb{G}$ one has a bound holding for almost every $\Lambda\in X$
\begin{align*} \left| Q_\pi ( \mathcal{N} _\phi )(\Lambda)\right| = \left| \mathcal{N}^\pi _\phi \circ \pi (\Lambda) \right| \leqslant \Vert \mathcal{N}^\pi _\phi \Vert _{\infty} \leqslant R_K \Vert \phi \Vert _{\infty}
\end{align*}

where the constant $R_K$ only depends on the translation-bound of measures in $X_\pi$ and on the compact set $K$, settling the proof.

\end{proof}

\subsection{\textsf{From admissible operators to Borel factors}}  ~ 

\vspace{0.2cm} 

Given a dynamical system $(X, \mathbb{G}, m)$ of translation-bounded measures on $\mathbb{G}$, we shall now arrive at the reciprocal statement of Proposition \ref{prop:factor.to.operator}, that is, given an admissible operator $Q$ we shall reconstruct the underlying Borel factor of translation-bounded measures $ (X_Q, \mathbb{G}, m_Q)$. Anticipating things a bit, for such constructed dynamical system of translation-bounded measures we denote as usual
\begin{align*} \xymatrixcolsep{3.5pc}\xymatrix{  \mathcal{C}_c(\mathbb{G}) \ni \phi \ar@{|->}[r] & \mathcal{N}^Q _\phi \in \mathcal{C}(X_Q)  }, \qquad \mathcal{N}^Q _\phi (\Lambda_Q) = \int _{\mathbb{G}}\phi _- \, d\Lambda _Q
\end{align*}

The result concerning this is as follows:

\vspace{0.2cm} 

\begin{prop}\label{prop:operator.to.factor} Let $Q \in \mathcal{B}(L^2(X, m))$ be admissible. Then there exists a dynamical system $(X_Q, \mathbb{G}, m_Q)$ of translation bounded measures on $\mathbb{G}$ and a Borel factor map 
\begin{align*} \xymatrixcolsep{3pc}\xymatrix{ \pi _Q: X \ni \Lambda \ar@{|->}[r] &  \Lambda _Q \in X_Q}
\end{align*}

such that $ \mathcal{N}^Q_{ \phi }\circ \pi _Q   = Q( \mathcal{N}_\phi )$ in the $L^2$ sense for each $\phi \in \mathcal{C}_c(\mathbb{G})$.
\end{prop}

\vspace{0.2cm} 
 
\begin{proof} We proceed by decomposing the proof into five steps $\textbf{(i)-(v)}$.\\

 $\textbf{(i)}$ We fix here our notations: We denote $U$ to be an open relatively compact subset of $\mathbb{G}$, and $M> 0$ a constant such that $X$ consists of $(U,M)$-translation-bounded measures. We set $\mathbb{G}^0$ to be a dense countable subset of $\mathbb{G}$, and let $(K_N)_N$ be a nested sequence of compact subsets which are the closure of their interior in $\mathbb{G}$, with first term $K_0$ containing$U$ in its interior, and whose union is $\mathbb{G}$. Next, as the group $\mathbb{G}$ is locally compact second countable, it is possible to set a countable collection $(\phi_n^N)_{n\in \mathbb{N}}\subset \mathcal{C}_{K_N}(\mathbb{G})$ which is dense in $\mathcal{C}_{K_N}(\mathbb{G})$ for the uniform norm. In particular the countable collection $(\phi_n^N)_{n,N\in \mathbb{N}}$ is dense in $\mathcal{C}_c(\mathbb{G})$ for the inductive limit topology. Third we denote by $V$ the $\mathbb{G}^0$-stable $\mathbb{Q}+i\mathbb{Q}$-vector sub-space of $\mathcal{C}_c(\mathbb{G})$ spanned by $(\phi_n^N*\delta _t)_{n,N\in \mathbb{N}, t\in \mathbb{G}^0}$. It may be written as the countable union of countable sets
\begin{align*}V = \bigcup _{l=1}^\infty\left\lbrace \sum _{j_1,j_2, j_3=1}^l \lambda _{j_1,j_2,j_3}\, \phi ^{j_2}_{j_1}*\delta _{t_{j_3}} \, \vert \; \lambda _{j_1,j_2,j_3} \in \mathbb{Q} + i\mathbb{Q}, \, t_{j_3} \in \mathbb{G}^0 \right\rbrace 
\end{align*}

and is therefore itself countable. For each $N$ the $\mathbb{Q}+i\mathbb{Q}$-vector sub-space $V_N := V \cap \mathcal{C}_{K_N}(\mathbb{G})$ is in particular countable and dense in $\mathcal{C}_{K_N}(\mathbb{G})$ for the uniform norm. To finish these preparations we make the following: Since $Q$ is assumed admissible one has for each $N$ a constant $R_N \geqslant 0$ such that
\vspace{0.2cm}
\begin{align*} \Vert Q\left[ \mathcal{N}_\phi \right] \Vert _{_\infty , ess} \leqslant R_N\Vert \phi \Vert _{\infty } \qquad \qquad \forall \phi \in V_N
\end{align*}

\vspace{0.2cm}
Therefore we let, for each integer $N$ and each $\phi \in V_N$, $Q( \mathcal{N}_\phi )$ to be a chosen representative Borel function of the $L^2$-class $ Q\left[ \mathcal{N}_\phi \right]$ which is uniformly bounded on $X$ by the value $R_N\Vert \phi \Vert _{\infty }$. In particular $Q( \mathcal{N}_\phi ) (\Lambda)$ makes sense for each $\phi \in V$ and any $\Lambda\in X$, and one has for each integer $N$
\begin{align}\label{continuity} \vert Q( \mathcal{N}_\phi ) (\Lambda)\vert \leqslant R_N\Vert \phi \Vert _{\infty } \qquad \qquad \forall \phi \in V_N , \Lambda\in X
\end{align}

\vspace{0.2cm}

$\textbf{(ii)}$ We here construct our desired collection of measures $\Lambda_Q$ on $\mathbb{G}$, where $\Lambda\in X$. Given any two functions $\phi ,\phi '\in V$ and $\lambda \in \mathbb{Q}+i\mathbb{Q}$, the Borel set of points $x\in X$ where
\vspace{0.2cm}
\begin{align}\label{linearity} Q( \mathcal{N}_{\lambda \phi + \phi '} )(\Lambda) = \lambda Q( \mathcal{N}_\phi )(\Lambda) + Q( \mathcal{N}_{\phi '} )(\Lambda) 
\end{align}

\vspace{0.2cm}
is of full measure in $X$. As a byproduct the countable intersection
\vspace{0.2cm}
\begin{align*}X^{(0)}:= \bigcap _{\phi ,\phi '\in V, \, \lambda \in \mathbb{Q}+i\mathbb{Q} }\left\lbrace \Lambda \in X \, \vert \,  \text{ equation (\ref{linearity}) holds at $\Lambda$} \right\rbrace 
\end{align*}

\vspace{0.2cm}
is a Borel subset of full measure in $X$. Since $m$ is $\mathbb{G}$-invariant on $X$ the countable intersection $X^{(1)}:= \cap_{t\in \mathbb{G}^0} X^{(0)}*\delta_{_{-t}}\subset X^{(0)}$ is Borel of full measure in $X$. Moreover since $Q$ comutes with the unitary operator $U_t$ on $L^2(X,m)$ the set of $\Lambda\in X$ where
\vspace{0.2cm}
\begin{align}\label{equivariance}Q( \mathcal{N}_{\phi * \delta_{_{ t}}} )(\Lambda)= Q U_t ( \mathcal{N}_\phi)(\Lambda) \quad \text{    is equal to    } \quad U_t Q(\mathcal{N}_\phi )(\Lambda) =  Q( \mathcal{N}_\phi )(\Lambda * \delta _{_{t}})
\end{align}

\vspace{0.2cm}
holds whatever $\phi \in V$ and $t\in \mathbb{G}^0$ is a Borel set of full measure in $X$ and thus intersects $ X^{(1)}$ into a $\mathbb{G}^0$-stable Borel subset of full measure $X^{(2)}$ in $X$. Now for any $\Lambda\in X^{(2)}$ the mapping $ \xymatrixcolsep{2pc}\xymatrix{Q_\Lambda: \phi \ar@{|->}[r] & Q( \mathcal{N}_\phi )(\Lambda) }$ is well defined on the $\mathbb{Q}+i\mathbb{Q}$-vector space $V$, and is $\mathbb{Q}+i\mathbb{Q}$-linear. Thanks to (\ref{continuity}) $Q_\Lambda$ is uniformly continuous on each $\mathbb{Q}+i\mathbb{Q}$-vector space $(V_N, \Vert . \Vert _{_{\infty}})$ so extends in a continuous $\mathbb{Q}+i\mathbb{Q}$-linear form on each $(\mathcal{C}_{K_N}(\mathbb{G}), \Vert . \Vert _{\infty})$. It thus passes to the inductive limit in a continuous $\mathbb{Q}+i\mathbb{Q}$-linear form on $\mathcal{C}_c(\mathbb{G})$, and it directly comes that it is in fact $\mathbb{C}$-linear on $\mathcal{C}_c(\mathbb{G})$. Thus, by the Riesz Representation Theorem there exists for each $\Lambda\in X^{(2)}$ a complex Borel measure $\Lambda _Q\in \mathcal{M}(\mathbb{G})$ such that,
\begin{align}\label{formula} Q( \mathcal{N}_\phi )(\Lambda ) = \int _{\mathbb{G}} \phi _- \, d \Lambda _Q \qquad  \qquad \forall \phi \in V, \, \Lambda\in X^{(2)}.
\end{align}

Setting $ \xymatrixcolsep{2pc}\xymatrix{ \pi _Q: X  \ar[r] &  \mathcal{M}(\mathbb{G})}$ to be $\pi _Q(\Lambda ):= \Lambda _Q$ as just constructed when $\Lambda $ belongs to the $\mathbb{G}^0$-stable Borel subset of full measure $X^{(2)}$, and to be for instance the trivial measure when $\Lambda \in X\backslash X^{(2)}$, yields a mapping.\\



\textbf{(iii)} We show that $\pi _Q$ is valued in the compact set $\mathcal{M}_{(U',R_0)}(\mathbb{G})$ of $(U',R_0)$-translation-bounded measures, where $U'$ is any open subset of $U$ with $\overline{U'} \subset U$ in $\mathbb{G}$, and $R_0$ is the constant involved in (\ref{continuity}) at integer $N=0$. Such set $U'$ always exists: Indeed if one picks a point $s\in U$ then one results with a compact set $\lbrace s\rbrace $ disjoint from the closed set $\mathbb{G}\backslash U$, so $\mathbb{G}$ being a topological group it is a regular space and thus has an open subset $U'$ containing $s$ and disjoint from $\mathbb{G}\backslash U$, as desired. Let us now assume $\Lambda\in X^{(2)}$, and let $t\in \mathbb{G}^0$ be given: As any complex measure the measure $\Lambda _Q$ satisfies (see \cite{BaaLen}, Proposition 1 therein for validity of the first equality)
\begin{align*} \vert \Lambda _Q\vert(U +t) &  =  \sup \left\lbrace \left| \int _{\mathbb{G}} \phi   \, d\Lambda _Q \right| \, : \, \phi \in \mathcal{C}_{U+t}(\mathbb{G}), \, \Vert \phi \Vert _{\infty} \leqslant 1 \right\rbrace \\
&  =  \sup \left\lbrace \left| \int _{\mathbb{G}} \phi (.-t)   \, d\Lambda _Q \right| \, : \, \phi \in \mathcal{C}_{U}(\mathbb{G}), \, \Vert \phi \Vert _{\infty} \leqslant 1 \right\rbrace  \\
&  =  \sup \left\lbrace \left| \int _{\mathbb{G}} \phi *\delta _t   \, d\Lambda _Q \right| \, : \, \phi \in \mathcal{C}_{U}(\mathbb{G}), \, \Vert \phi \Vert _{\infty} \leqslant 1 \right\rbrace \\
& \leqslant  \sup \left\lbrace \left| \int _{\mathbb{G}} \phi *\delta _t   \, d\Lambda _Q \right| \, : \, \phi \in \mathcal{C}_{K_0}(\mathbb{G}), \, \Vert \phi \Vert _{\infty} \leqslant 1 \right\rbrace 
\end{align*}

The later being true since we supposed $K_0$ to contain $U$ from the beginning. This is in turn equal, with $V_0 :=  V \cap \mathcal{C}_{K_0}(\mathbb{G})$ as in $\textbf{(i)}$, dense in $\mathcal{C}_{K_0}(\mathbb{G})$, to 
\begin{align*}   \sup \left\lbrace \left| \int _{\mathbb{G}} \phi *\delta _t    \, d\Lambda _Q \right| \, : \, \phi \in V_0, \, \Vert \phi \Vert _{\infty} \leqslant 1 \right\rbrace 
\end{align*}

Now as $\phi \in V_0\subset V$ and $t\in \mathbb{G}^0$ then $\phi *\delta _t$ is again in $V$, and by (\ref{formula}) this quantity is
\vspace{0.2cm}
\begin{align*}  & \sup \left\lbrace \left| Q( \mathcal{N}_{\phi * t} )(\Lambda)\right| \, : \, \phi \in V_0, \, \Vert \phi \Vert _{\infty} \leqslant 1 \right\rbrace 
\end{align*}

\vspace{0.2cm}
which by (\ref{continuity}) is bounded by $R_0$. Therefore $ \vert \Lambda _Q\vert(U +t) \leqslant R_0$ holds for $t\in \mathbb{G}^0$. Now if an $s\in \mathbb{G}$ is given then from the choice of $U'$ one can in fact find a $t\in \mathbb{G}^0$ such that $U' +s \subset U+t$: Indeed one easily prove (by contradiction) that there is an open neighborhood $B$ of $0$ in $\mathbb{G}$ such that $U' + B\subset U$, so by density of $ \mathbb{G}^0$ one can pick up a $t\in \mathbb{G}^0$ with $s-t\in B$, yielding
\begin{align*} U'+s = U'+(s-t) +t \subset U'+B+t \subset U +t,
\end{align*}

as desired. This yields that $  \vert \Lambda _Q\vert(U' +s) \leqslant  \vert \Lambda _Q\vert(U +t) \leqslant R_0$ for any $s\in \mathbb{G}$ whenever $\Lambda$ belongs to $X^{(2)}$, that is, any such $\Lambda _Q$ is $(U',R_0)$-translation-bounded. By construction $\Lambda _Q$ is trivial when $\Lambda\in X\backslash X^{(2)}$ so is also $(U',R_0)$-translation-bounded in that case. One has therefore as desired
\begin{align*} \xymatrixcolsep{3.5pc}\xymatrix{ \pi _Q : X \ni \Lambda  \ar@{|->}[r] & \Lambda _Q \in \mathcal{M}_{(U',R_0)}(\mathbb{G}) }
\end{align*}

\vspace{0.2cm}

\textbf{(iv)} We show here that this map is a Borel $\mathbb{G}$-map. To show that it is Borel, since $\mathcal{M}(\mathbb{G})$ carries the vague topology it suffices to check that the map $  \Lambda \longmapsto \int _{\mathbb{G}} \phi  \, d (\Lambda _Q)  $ is Borel for each $\phi$ taken in the dense subspace $V$ of $\mathcal{C}_c(\mathbb{G})$. But this map is precisely the Borel map $Q( \mathcal{N}_\phi )$ on the Borel subset $X^{(2)}$ and the constant nul map on $ X\backslash X^{(2)}$, so the result follows. Now we infer that
\vspace{0.2cm}
\begin{align}\label{equivariance2} \Lambda _Q * \delta _t = (\Lambda * \delta _t)_Q \qquad \qquad \forall \Lambda \in X^{(2)} , \, t\in \mathbb{G}^0
\end{align}

\vspace{0.2cm}
Here $(\Lambda * \delta _t)_Q $ makes sense since $ X^{(2)}$ is $\mathbb{G}^0$-stable. This is due to the very construction of $ X^{(2)}$ since for such measures one has (\ref{equivariance}) holding, yielding at each $\phi \in V$
\begin{align*} \int _{\mathbb{G}} \phi _- \, d(\Lambda _Q * \delta _t)  =   \int _{\mathbb{G}} \phi _-* \delta _t  \, d(\Lambda _Q) = Q( \mathcal{N}_{\phi * \delta _t} ) (\Lambda ) = Q( \mathcal{N}_\phi ) (\Lambda * \delta _t) = \int _{\mathbb{G}} \phi _- \, d((\Lambda * \delta _t)_Q)
\end{align*}

so that by density of $V$ in $\mathcal{C}_c(\mathbb{G})$ the equality $\Lambda _Q * \delta _t = (\Lambda * \delta _t)_Q$ occurs on the whole space $\mathcal{C}_c(\mathbb{G})$, whenever $x\in X^{(2)} $ and $t\in \mathbb{G}^0$. Now to see that $\pi _Q$ is a $\mathbb{G}$-map observe that if, given $t_0\in \mathbb{G}$, one constructs another mapping $\pi '_Q$ as we just did, but with some $\mathbb{G}^{0'}$ containing $t_0$ playing the role of $\mathbb{G}^0$, some $\mathbb{G}^{0'}$-stable countable $\mathbb{Q}+i\mathbb{Q}$-vector subspace $V'\subset \mathcal{C}_c(\mathbb{G})$ containing our space $V$ together with Borel representatives $Q(\mathcal{N}_\phi)'$ for each $\phi \in V'$ (which possibly differ from $Q(\mathcal{N}_\phi)$ when $\phi \in V \subseteq V'$) then we also end up with a Borel map $ \pi '_Q : X \ni \Lambda \longmapsto\Lambda '_Q \in  \mathcal{M}_{(U',R_0)}(\mathbb{G}) $ together with a Borel subset of full measure $X^{(2) '}$ in $X$ such that, combining (\ref{formula}) and (\ref{equivariance2}) in this case,
\vspace{0.2cm}
\begin{align*} Q( \mathcal{N}_\phi )'(\Lambda ) = \int _{\mathbb{G}} \phi _- \, d \Lambda '_Q   \, \, \text{ and }\, \, \Lambda '_Q * \delta _t = (\Lambda * \delta _t)'_Q \, \, \forall  \phi \in V', \, \Lambda\in X^{(2)'} , \, t\in \mathbb{G}^{0'}
\end{align*}

\vspace{0.2cm}
Now the countable collection $V$ being contained in $V'$, and since the Borel maps $Q( \mathcal{N}_\phi )$ and $Q( \mathcal{N}_\phi )'$ are representatives of a same $L^2$-class so that they coincide almost everywhere on $X$, one have for each $\Lambda$ in the Borel set of full measure
\begin{align*} X^{(2) } \cap X^{(2) '} \cap \bigcap_{\phi \in V}\left\lbrace Q( \mathcal{N}_\phi ) = Q( \mathcal{N}_\phi )'\right\rbrace 
\end{align*}

equalities of $\Lambda _Q$ with $\Lambda _Q'$ on $V$, and thus on the whole space $\mathcal{C}_c(\mathbb{G})$. This shows that $\pi _Q = \pi _Q'$ almost everywhere on $X$, and since from its construction one has $\pi _Q'(\Lambda) *\delta_{t_0} = \pi _Q'(\Lambda * \delta_{t_0})$ for almost every $\Lambda$ one deduces that $\pi _Q(\Lambda) *\delta_{t_0} = \pi _Q(\Lambda * \delta_{t_0})$ for almost every $\Lambda \in X$, for any $t_0 \in \mathbb{G}$, as desired.\\

$\textbf{(v)}$ We now set properly the dynamical system $(X_Q, \mathbb{G},m_Q)$ and the stated Borel $\mathbb{G}$-map $\pi _Q$, and prove validity of formula of the statement. Let $m_Q$ be the push-forward probability measure of $m$ through $\pi_Q$, which is supported on the compact $\mathbb{G}$-stable set $\mathcal{M}_{(U',R_0)}(\mathbb{G}) $ and ergodic. Its support $X_Q$ is then a compact $\mathbb{G}$-stable set of $(U',R_0)$-translation-bounded measures, yielding our desired system $(X_Q, \mathbb{G},m_Q)$. It now suffices to modify $\pi_Q$, if needed, on a set of null measure in order to have a Borel $\mathbb{G}$-map properly valued in $X_Q$, while having a Borel subset of full measure $X^\infty$ where $\pi _Q$ as modified still satisfies form (\ref{formula})
\begin{align*} \pi _Q(\Lambda ) = \Lambda _Q \quad \text{with} \quad Q( \mathcal{N}_\phi )(\Lambda ) = \int _{\mathbb{G}} \phi _- \, d \Lambda _Q \quad \forall \, \phi \in V, \, \Lambda\in X^{\infty}
\end{align*}

This in particular ensure the equality $\mathcal{N}^Q_{ \phi }\circ \pi _Q   = Q( \mathcal{N}_\phi )$ in the $L^2$ sense for each $\phi \in V$. Since the linear maps
\begin{align*}\xymatrixcolsep{2pc}\xymatrix{\mathcal{C}_c(\mathbb{G}) \ni \phi  \ar@{|->}[r] &  \mathcal{N}^Q_{ \phi }\circ \pi _Q \in L^2(X,m) } \quad \xymatrixcolsep{2pc}\xymatrix{\mathcal{C}_c(\mathbb{G}) \ni \phi  \ar@{|->}[r] &  Q(\mathcal{N}_{ \phi })\in L^2(X,m) }
\end{align*}

are continuous and coincide on the dense collection $V$ as we just saw, they must agree everywhere, concluding the proof.
\end{proof}

The Borel map $\pi _Q$ as set in the above Proposition is uniquely defined up to almost everywhere equality due to the formula $ \mathcal{N}^Q_{ \phi }\circ \pi _Q   = Q( \mathcal{N}_\phi )$ required to hold for each $\phi \in \mathcal{C}_c(\mathbb{G})$. The constructed Borel factor of translation-bounded measures $(X_Q, \mathbb{G},m_Q)$ is however uniquely defined (if one requires $m_Q$ to have full support).

\subsection{\textsf{Diffraction properties}} ~ 

\vspace{0.2cm} 

We check here the absolute continuity property for the diffraction $\widehat{\gamma}_Q$ arising from the system $(X_Q, \mathbb{G}, m_Q)$ previously constructed from an admissible operator $Q$:

\vspace{0.2cm} 

\begin{prop}\label{prop:diffraction} Any admissible operator $Q\in \mathcal{B}(L^2(X,m))$ has its associated system $(X_Q, \mathbb{G},m_Q)$ having diffraction of the form $\widehat{\gamma}_Q = f_Q .\widehat{\gamma}$ for some $f_Q\in L^{\infty} (\widehat{\mathbb{G}}, \widehat{\gamma})$.
\end{prop}

\vspace{0.2cm} 

\begin{proof} One has the following set of equalities for any two $\phi , \psi \in \mathcal{C}_c(\mathbb{G})$, the first ensured by Theorem \ref{theo:diffraction} and the second coming from the equalities $ \mathcal{N}^Q_{ \phi }\circ \pi _Q   = Q( \mathcal{N}_\phi )$ and $ \mathcal{N}^Q_{ \psi }\circ \pi _Q   = Q( \mathcal{N}_\psi )$,
\vspace{0.2cm}
\begin{align}\label{formulation.diffraction}\int \widehat{\phi}\overline{\widehat{\psi}} \, d\widehat{\gamma}_Q  = \int _{X_Q} \mathcal{N} ^Q_\phi \overline{ \mathcal{N}^Q_\psi} \, d m_Q =  \int _X Q( \mathcal{N} _\phi ) \overline{ Q( \mathcal{N}_\psi ) } \, dm
\end{align}

\vspace{0.2cm}
It directly follows from this equality that for each $\phi \in \mathcal{C}_c(\mathbb{G})$
\vspace{0.2cm}
\begin{align*}  \Vert \widehat{\phi} \Vert _{_{L^2 (\widehat{\mathbb{G}}, \widehat{\gamma}_Q)}} = \Vert Q( \mathcal{N} _\phi ) \Vert _{_{L^2(X,m)}} \leqslant \Vert Q \Vert _{_{op}} \Vert  \mathcal{N} _\phi  \Vert _{_{L^2(X,m)}} = \Vert Q \Vert _{_{op}} \Vert \widehat{\phi} \Vert _{_{L^2 (\widehat{\mathbb{G}}, \widehat{\gamma})}}
\end{align*}

\vspace{0.2cm}
where $\Vert Q \Vert _{_{op}}$ is the operator norm of $Q$ as a bounded operator on $L^2(X,m)$. Therefore the identity map
\begin{align*} \xymatrixcolsep{3.5pc}\xymatrix{ L^2 (\widehat{\mathbb{G}}, \widehat{\gamma})  \ar[r] & L^2 (\widehat{\mathbb{G}}, \widehat{\gamma}_Q) } 
\end{align*}

is well-defined and uniformly continuous on the subspace formed of functions $ \widehat{\phi}$ with $ \phi \in \mathcal{C}_c(\mathbb{G})$, which is dense in $L^2 (\widehat{\mathbb{G}}, \widehat{\gamma})$ since $\widehat{\gamma}$ is translation-bounded. Thus the identity map is well-defined and of operator norm less or equal to $\Vert Q \Vert _{_{op}}$ on all $L^2 (\widehat{\mathbb{G}}, \widehat{\gamma})$. This means that $ \Vert f \Vert _{_{L^2 (\widehat{\mathbb{G}}, \widehat{\gamma}_Q)}} \leqslant \Vert Q \Vert _{_{op}} \Vert f \Vert _{_{L^2 (\widehat{\mathbb{G}}, \widehat{\gamma})}}$ for any $f \in L^2 (\widehat{\mathbb{G}}, \widehat{\gamma})$, and in particular one has for each $f \geqslant 0$ in $ L^1 (\widehat{\mathbb{G}}, \widehat{\gamma})$ the inequality
\begin{align}\label{absolute.continuity} 0 \leqslant \int f \, d\widehat{\gamma}_Q \leqslant \Vert Q \Vert _{_{op}}^2 \int f \, d\widehat{\gamma}
\end{align}

Applying this to the indicator function of compact sets of the locally compact Abelian group $\widehat{\mathbb{G}}$ shows that $\widehat{\gamma}_Q$ is absolutely continuous with respect to $\widehat{\gamma}$, and thus admits a Radon-Nikodym differential $f_Q:= \frac{d\widehat{\gamma}_Q}{d\widehat{\gamma}}$ in $ L^1_{loc} (\widehat{\mathbb{G}}, \widehat{\gamma})$, which by an application of (\ref{absolute.continuity}) is positive, essentially bounded by $\Vert Q \Vert _{_{op}}^2$, yielding the proof.
\end{proof}

\subsection{\textsf{The correspondence between Borel factors and admissible operators}} ~ 

\vspace{0.2cm} 

Let as usual $(X, \mathbb{G},m)$ be a measured dynamical system of translation-bounded measures with diffraction to dynamic map
\begin{align*}\xymatrixcolsep{3.5pc}\xymatrix{\Theta  : L^2(\widehat{\mathbb{G}}, \widehat{\gamma} )\, \,  \ar@{<->}[r] & \, \, \mathcal{H}_{\Theta } \subseteq L^2(X,m)}, \qquad \Theta (\widehat{\phi}) = \mathcal{N}_\phi
\end{align*}

for any $\phi \in \mathcal{C}_c(\mathbb{G})$. Then a combined formulation of Propositions \ref{prop:factor.to.operator}, \ref{prop:operator.to.factor} and \ref{prop:diffraction} is the following correspondence Theorem:\\
 
\begin{theo}\label{theo:correspondence} Let $(X, \mathbb{G}, m)$ be a dynamical system of translation-bounded measures with diffraction $\widehat{\gamma}$. There is a bijective correspondence between:\\ 

\begin{enumerate}
\item[(a)] Borel factor maps $\pi$ over some dynamical system of translation-bounded measures having diffraction $\widehat{\gamma}_\pi \ll \widehat{\gamma}$, with Radon-Nikodym derivative $f_\pi :=\frac{d \widehat{\gamma}_\pi}{d \widehat{\gamma}} \in L^{\infty } (\widehat{\mathbb{G}}, \widehat{\gamma})$,

\vspace{0.2cm}

\item[(b)] admissible operators $Q$ on $L^2(X,m)$.
\end{enumerate}
Given such a Borel factor map $\pi$ with admissible operator $Q_\pi$, the multiplication operator $M_{f_\pi}$ by $f_\pi$ on $L^2(\widehat{\mathbb{G}}, \widehat{\gamma})$ satisfies
$$\Theta \circ M_{f_\pi} \circ \Theta ^{-1}= Q_{\pi}^* Q_{\pi}$$. In particular, given a Borel subset $\mathcal{P}\subseteq \widehat{\mathbb{G}}$ then $f_\pi = \mathbb{I}_{\mathcal{P}}$ if and only is $\vert Q_{\pi} \vert = \mathsf{E}(\mathcal{P})P_\Theta$.
\end{theo}

\vspace{0.3cm}




Here, $\vert Q \vert$ denotes the absolute value of the operator $Q$, that is, the unique positive operator whose square is $Q^*Q$, see for instance \cite[Section VI.4]{ReedSimon} for definition and existence.

\vspace{0.2cm} 

\begin{rem} If we are given a Borel factor map $\pi$ with associated dynamical system of translation-bounded measures having diffraction $f_\pi. \widehat{\gamma}$, $f_\pi \in L^\infty (\widehat{\mathbb{G}}, \widehat{\gamma})$, and letting $Q_\pi$ its associated admissible operator, then it is not difficult to show the equality of norms
\begin{align*}\Vert Q_\pi \Vert _{op} = \Vert f_\pi \Vert ^{\frac{1}{2}}_{L^\infty (\widehat{\mathbb{G}}, \widehat{\gamma})}
\end{align*}

Indeed the operator norm is a $C^*$-norm so gives $\Vert Q_\pi \Vert _{op}^2 = \Vert Q_\pi^*Q_\pi \Vert _{op}$, equal to the operator norm of the pullback of $Q_\pi^*Q_\pi$ under the isometry $\Theta$, in turn equal to the operator norm of the multiplication operator by $f_\pi$ on $L^2 (\widehat{\mathbb{G}}, \widehat{\gamma})$ according to Theorem \ref{theo:correspondence}, which is nothing but $\Vert f_\pi \Vert _{L^\infty (\widehat{\mathbb{G}}, \widehat{\gamma})}$.
\end{rem}










\begin{rem} As it is easy to note the collection of admissible operators of a given dynamical system of translation-bounded measures $(X, \mathbb{G}, m)$ is a vector space of bounded operators on $L^2(X,m)$. This has the following meaning when one considers the associated Borel factor maps: Given two admissible operators $Q, Q'$ and a complex number $\lambda$ one can consider the mapping $\pi _Q + \lambda \pi _{Q'}$ form $X$ to $\mathcal{M}^\infty (\mathbb{G})$ defined in the obvious way, which is a Borel $\mathbb{G}$-map valued in a certain compact space of translation-bounded measures. In fact, what one has is the equality (holding almost everywhere)
\vspace{0.2cm}
\begin{align*}  \pi _{\lambda Q} = \lambda \pi _{Q} \qquad \qquad  \pi _{Q +  Q'} =\pi _Q +  \pi _{Q'}
\end{align*}

\vspace{0.2cm}
There is a similar yet less trivial result about diffraction measures, namely
\vspace{0.2cm}
\begin{align*}\widehat{\gamma}_{\lambda Q} = \vert \lambda \vert ^2 \widehat{\gamma}_{Q} \qquad \qquad  \widehat{\gamma}_{Q+ Q'}= \widehat{\gamma}_{Q} + \widehat{\gamma}_{Q'} + \omega
\end{align*}

\vspace{0.2cm}
where $\omega$ is a positive measure of the form $g.\widehat{\gamma}$ such that the multiplication operator $M_g$ is equal to the pullback of $Q^*Q' + Q'^*Q$ on $L^2 (\widehat{\mathbb{G}}, \widehat{\gamma})$ under the isometry $\Theta$. In particular, whenever two admissible projectors $P$ and $P'$ are orthogonal, meaning that $PP'=P'P=0$ on $L^2(X,m)$, then the sum $P+P'$ is again an admissible projector and there exist therefore dynamical systems of translation-bounded measures associated with $P$, $P'$ and $P+P'$ respectively. It is then not hard to show that one almost surely has
\begin{align*}\pi_{P+P'} = \pi_{P} + \pi_{P'} \qquad \qquad \widehat{\gamma}_{P+P'} = \widehat{\gamma}_{P} + \widehat{\gamma}_{P'}
\end{align*}

This generalizes to any finite sums of pairwise orthogonal admissible projectors on $L^2(X,m)$.
\end{rem}

\subsection{\textsf{Examples of admissible operators: Convolution operators}} ~ 

\vspace{0.2cm} 


We shall see here that, given a dynamical system of translation-bounded measures $(X, \mathbb{G}, m)$, one can define a whole class of admissible projectors, each associated with some bounded measure $\sigma \in \mathcal{M}^b(\mathbb{G})$. Given such $\sigma$ define on $L^2(X,m)$ the convolution operator $Q_\sigma$ by the operator-valued integral 
\vspace{0.2cm}
\begin{align*}Q_\sigma := \int _{\mathbb{G}} U_t  \, d\sigma (t)
\end{align*}

\vspace{0.2cm}
Such operator is well-defined, bounded on $L^2(X,m)$ and commutes with the unitary representation of $\mathbb{G}$. Recall that the convolution product between a bounded measure $\sigma$ and a compactly supported continuous function $\phi$ is a continuous function belonging to $L^1(\mathbb{G})$. From \cite{ArgDLa}, Theorem 1.2 therein, one knows that the convolution $\Lambda * \sigma$ of any translation-bounded measure $\Lambda$ with any bounded measure $\sigma$ is well-defined, and is again a translation-bounded measure. Also recall that any $\sigma \in \mathcal{M}^b(\mathbb{G})$ is Fourier transformable, with Fourier transform given by a density $\widehat{\sigma}$ (its usual Fourier-Stieltjes transform) belonging to the space $\mathcal{C}_u ^b(\widehat{\mathbb{G}})$ of bounded uniformly continuous functions on $\widehat{\mathbb{G}}$ (\cite{Rudin}, Paragraph 1.3.3 therein).

\vspace{0.2cm}

\begin{prop} The operator $Q_\sigma$ is admissible, with dynamical system of translation-bounded measures $(X_\sigma , \mathbb{G}, m_\sigma )$ and continuous $\mathbb{G}$-map of the form
\vspace{0.2cm}
\begin{align*} \xymatrixcolsep{3.5pc}\xymatrix{  X \ni \Lambda \, \ar@{|->}[r] & \, \Lambda_{\sigma}:= \Lambda * \sigma \in  X_{\sigma}  }
\end{align*}

\vspace{0.2cm}
Moreover $(X_\sigma , \mathbb{G}, m_\sigma )$ has diffraction $\vert \widehat{\sigma} \vert ^2 \widehat{\gamma}$ with $\vert \widehat{\sigma} \vert ^2 \in \mathcal{C}_u ^b(\widehat{\mathbb{G}})$.
\end{prop}
 
\vspace{0.2cm}

\begin{proof}  Let us show that $Q_\sigma ( \mathcal{N}_\phi ) = \mathcal{N}_{\sigma * \phi}$ at each $\phi \in \mathcal{C}_c(\mathbb{G})$. Let $\phi \in \mathcal{C}_K(\mathbb{G})$ for some compact $K \subset \mathbb{G}$: One has at any $\Lambda \in X$
\begin{align*}Q_\sigma ( \mathcal{N}_\phi )(\Lambda) = \int _{\mathbb{G}} U_t (\mathcal{N}_\phi)(\Lambda)  \, d\sigma(t) = \int _{\mathbb{G}} \mathcal{N}_\phi(\Lambda * \delta_t)  \, d\sigma(t) 
\end{align*}

in turns equal to 
\begin{align*}\int _{\mathbb{G}} \int _{\mathbb{G}} \phi (-s) \, d(\Lambda * \delta_t)(s)  \, d\sigma(t) = \int _{\mathbb{G}} \int _{\mathbb{G}} \phi (-s -t) \, d\Lambda (s)  \, d\sigma(t) = \int _{\mathbb{G}} \int _{\mathbb{G}} \phi (-s -t) \,d\sigma(t) \, d\Lambda (s)   
\end{align*}
 
which is precisely equal to $\mathcal{N}_{\sigma * \phi}(\Lambda)$. This in particular gives, with $K_\sigma $ and $K_\sigma  '$ appropriate constants only depending on $\sigma $, $X$ and on the domain $K$,
\vspace{0.2cm}
\begin{align*} \Vert Q_\sigma ( \mathcal{N}_\phi ) \Vert _{\infty, ess} =  \Vert \mathcal{N}_{\sigma * \phi} \Vert _{\infty, ess} \leqslant K_\sigma  \Vert \sigma* \phi \Vert _{\infty} \leqslant K_\sigma  '\Vert  \phi \Vert _{\infty} 
\end{align*}

\vspace{0.2cm} 
showing admissibility. Now for (almost every) $\Lambda \in X$ the measure $\Lambda_{\sigma}$ is uniquely defined according to $$ \int _{\mathbb{G}} \phi _- \, d\Lambda _{\sigma} =  \mathcal{N}_\phi ^{Q_\sigma}(\Lambda_{\sigma}) = Q_\sigma ( \mathcal{N}_\phi )(\Lambda) =  \mathcal{N}_{\sigma * \phi}(\Lambda)  = \int _{\mathbb{G}} \phi _- \, d(\sigma *\Lambda ) $$
after computation, so $\pi _\sigma$ and $(X_\sigma , \mathbb{G}, m_\sigma )$ are of the stated form. The diffraction $\widehat{\gamma}_\sigma$ of $(X_\sigma , \mathbb{G}, m_\sigma )$ satisfies for any pair $\phi , \psi \in \mathcal{C}_c(\mathbb{G})$
\begin{align*} \int _{\widehat{\mathbb{G}}} \widehat{\phi} \overline{\widehat{\psi}} \, d \widehat{\gamma}_\sigma = \int _{X} Q_\sigma ( \mathcal{N}_\phi ) \overline{Q_\sigma ( \mathcal{N}_\psi )} \, dm = \int _{X}  \mathcal{N}_{\sigma * \phi}  \overline{ \mathcal{N}_{\sigma * \psi} } \, dm = \int _{\widehat{\mathbb{G}}} \widehat{\sigma* \phi} \overline{\widehat{\sigma * \psi}} \, d \widehat{\gamma} =  \int _{\widehat{\mathbb{G}}} \widehat{\phi} \overline{\widehat{\psi}} \vert \widehat{\sigma} \vert ^2  \, d \widehat{\gamma}
\end{align*}

showing that $\widehat{\gamma}_\sigma = \vert \widehat{\sigma} \vert ^2 \widehat{\gamma}$, as desired.\\
\end{proof}

\section{\textsf{Appendix: Existence of a Maximal Kronecker Borel factor}}



We wish here, as we could not find any reference for this statement which seems to be classical though, to provide a proof of Proposition \ref{prop.construction.Borel.factor.map}, that is, to show that for any ergodic dynamical system $(X, \mathbb{G}, m)$ with eigenvalue group $\mathcal{E}\subset \widehat{\mathbb{G}}$ the Kronecker system $(\text{T}_{ \mathcal{E} }, \mathbb{G}, m_{Haar})$ is a Borel factor of $(X, \mathbb{G}, m)$.\\
\\
\textit{Proof of Proposition \ref{prop.construction.Borel.factor.map}.} By ergodicity each eigenfunction $ \Phi _\omega \in L^2(X,m) $, with $ \omega \in \mathcal{E} $, is unique up to a multiplicative constant and can moreover be chosen of absolute value almost everywhere constant equal to $1$ on $X$. Also, as discussed in Paragraph \ref{paragraph:dynamical.systems} each $\Phi _\omega$ can also be chosen such that the equality $\Phi _\omega (x .t) = \omega (t)\Phi _\omega (x)$ occurs everywhere on a $\mathbb{G}$-stable Borel set $\tilde{X}$ and for any $  t\in \mathbb{G} $. Now the celebrated Pointwise Ergodic Theorem (see for instance \cite[Theorem 2.14]{MulRic}) states as:

\vspace{0.2cm}

\begin{theo}\label{theo:Birkhoff}(Pointwise Ergodic Theorem) Let $(X, \mathbb{G}, m)$ be an ergodic dynamical system. Then for any tempered Van Hove sequence $(\mathcal{A}_n)_{n\in \mathbb{N}}$ in $\mathbb{G}$ and any Borel $m$-integrable function $f$ on $X$, one has for $m$-almost every $x \in X$
\begin{align*}\displaystyle \int _{X}f \, dm = \lim_{n\rightarrow \infty} \frac{1}{\left|\mathcal{A}_n\right| } \int _{\mathcal{A}_n} f(x.t)\, dt
\end{align*}
\end{theo}

\vspace{0.2cm}

Using this theorem we can prove:

\vspace{0.2cm}

\begin{lem}\label{lem.multiplicative.property} The eigenfunctions $ \Phi _\omega$ can be chosen to satisfy
\begin{align*}
 \Phi _\omega.\Phi _{\omega'}=\Phi _{\omega.\omega'} \quad \text{ m-almost everywhere } \forall\, \, \omega , \omega ' \in \mathcal{E} .
 \end{align*}
\end{lem}

\vspace{0.2cm}

\begin{proof} Consider a family $( \Phi _\omega )_{\omega \in \mathcal{E}}$ of eigenfunctions of modulus constant equal to $1$ on $X$, and such that the equality $\Phi _\omega (x.t) = \omega (t)\Phi _\omega (x)$ occurs on a $\mathbb{G}$-stable Borel set $\tilde{X}$ and for any $  t\in \mathbb{G} $. Since this collection of functions is countable there exists, by ergodicity of $m$, a Borel subset of full measure on $X$ such that $\vert \Phi _\omega \vert =1$ and such that the conclusion of the Pointwise Ergodic Theorem holds for any function in the collection
\begin{align*}\displaystyle \left\lbrace  \left| \lambda_1\Phi _\omega . \Phi _{\omega '} - \lambda _2\Phi _{\omega .\omega '}\right| \, : \, \omega \in \mathcal{E}, \, \lambda _1, \lambda _2 \in \mathbb{C}\right\rbrace 
\end{align*}

Indeed one can find such a Borel subset where the conclusion of the Pointwise Ergodic Theorem holds when $\lambda _1, \lambda _2$ are taken in a countable dense subset of $\mathbb{C}$, and one easily shows that on such set the conclusion of the Pointwise Ergodic Theorem also holds for any parameters $ \lambda _1, \lambda _2 \in \mathbb{C}$. Now pick up such an element $x_0$ and set $\Phi _\omega' := \overline{\Phi _\omega (x_0)}\Phi _\omega$: Then $ \Phi _\omega '. \Phi _{\omega '}' = \Phi _{\omega .\omega '}'$ along the orbit of $x_0$ as it is straightforward to observe, and it follows from the very choice of $x_0$ that 
\begin{align*}\displaystyle \int _{X}\vert \Phi _\omega '. \Phi _{\omega '}' - \Phi _{\omega .\omega '}'\vert  \, dm = \lim_{n\rightarrow \infty} \frac{1}{\vert \mathcal{A}_n \vert} \int _{\mathcal{A}_n} \vert \Phi _\omega '. \Phi _{\omega '}' - \Phi _{\omega .\omega '}'\vert  (x_0.t)\, dt =0
\end{align*}
which shows that $ \Phi _\omega '. \Phi _{\omega '}' = \Phi _{\omega .\omega '}'$ almost everywhere on $X$. Thus one has a new family $( \Phi _\omega ' )_{\omega \in \mathcal{E}}$ of eigenfunctions, still of modulus constant equal to $1$ on $X$ and such that the equality $\Phi _\omega '(x.t) = \omega (t)\Phi _\omega (x)'$ occurs on a $\mathbb{G}$-stable Borel set $\tilde{X}$ and for any $  t\in \mathbb{G} $, and obeying the stated equality.
\end{proof}

We continue our proof of Proposition \ref{prop.construction.Borel.factor.map}: Considering a family $( \Phi _\omega )_{\omega \in \mathcal{E} }$ of eigenfunctions as chosen and moreover satisfying the statement of Lemma \ref{lem.multiplicative.property}, one ends up with a $\mathbb{G}$-invariant Borel subset $$ \displaystyle X^0:= \bigcap _{\omega, \omega '\in \mathcal{E}} \left\lbrace \Phi _\omega . \Phi _{\omega '} = \Phi _{\omega.\omega '}\right\rbrace $$ of full $m$-measure in  $X$. For each $x$ in this Borel set, define $\pi (x)$ to be the unique character on $\mathcal{E}$ defined by $ \xymatrixcolsep{2pc}\xymatrix{ \omega  \ar@{|->}[r] & \Phi _\omega (x) }$, and to be the identity element elsewhere on $X$. This defines a map $\xymatrixcolsep{2.5pc}\xymatrix{ \pi : X \ar[r] &  \text{T}_{ \mathcal{E} }}$ obeying the desired properties: Indeed it is a Borel map which is $\mathbb{G}$-equivariant in the sense that for any $t\in \mathbb{G}$ one has $\pi (x.t) = \pi (x).\mathsf{i}(t)$ for $m$-almost every $x\in X$. From this latter property it comes that the probability measure $m$ push forwarded on $\text{T}_{ \mathcal{E} }$ is translation invariant by the dense subgroup $\mathsf{i}(\mathbb{G})$ and thus is the Haar measure on the compact Abelian group $\text{T}_{ \mathcal{E} }$, completing the proof of the statement.

\vspace{0.4cm}

\textbf{Acknowledgments.}  The author is grateful to Daniel Lenz for his careful reading of the manuscript.

\vspace{0.3cm}

\bibliographystyle{plain}

\bibliography{Biblio}

\end{document}